\documentclass{amsart}

\usepackage{cancel}
\usepackage{soul}
\usepackage[normalem]{ulem}

\usepackage{amssymb,amsmath,graphicx}
\usepackage{xcolor}
\usepackage{enumerate}
\newtoks\prt

\newtheorem{thm}{Theorem}[section]

\newtheorem{lemma}[thm]{Lemma}
\newtheorem{prop}[thm]{Proposition}
\newtheorem{cor}[thm]{Corollary}

\newtheorem{obs}[thm]{Observation}
\newtheorem{example}[thm]{Example}

\newtheorem*{question}{Question}
\theoremstyle{definition}
\newtheorem{example2}[thm]{Example}
\newtheorem{remark}[thm]{Remark}

\def\eqn#1$$#2$${\begin{equation}\label#1#2\end{equation}}
 \def\J#1#2#3{ \left\{ #1,#2,#3 \right\} }

\def\1{\boldsymbol{1}}

\def\B{\mathcal B}
\def\C{\mathcal C}

\def\U{\mathcal U}

\def\P{\mathcal P}

\def\ce{\mathbb C}

\def\en{\mathbb N}
\def\er{\mathbb R}

\def\Ha{\mathbb H}

\def\TT{\mathbb T}

\def \b{\boldsymbol{b}}

\def \e{\boldsymbol{e}}
\def \x{\boldsymbol{x}}

\def \f{\boldsymbol{f}}
\def \g{\boldsymbol{g}}
\def \h{\boldsymbol{h}}
\def \uu{\boldsymbol{u}}
\def \vv{\boldsymbol{v}}

\def \p{\boldsymbol{p}}

\def\span{\operatorname{span}}

\def \reg {\partial _{\kern1pt\text{reg}}}

\def\inv{\diamond}

\def\ip#1#2{\left\langle#1,#2\right\rangle}
\def\dt{\operatorname{dt}}

\newcommand{\norm}[1]{\left\|#1\right\|}

\newcommand{\wscl}[1]{\overline{#1}^{w^*}}

\newcommand{\abs}[1]{\left|#1\right|}
\newcommand{\setsep}{;\,}

\numberwithin{equation}{section}

\definecolor{green}{rgb}{0,0.5,0}

\title{Order type relations on the set of tripotents in a JB$^*$-triple}

\author[J. Hamhalter]{Jan Hamhalter}

\author[O.F.K. Kalenda]{Ond\v{r}ej F.K. Kalenda}

\author[A.M. Peralta]{Antonio M. Peralta}

\address{Czech Technical University in Prague, Faculty of Electrical Engineering, Department of Mathematics, Technicka 2, 166 27, Prague 6,
Czech Republic}
\email{hamhalte@math.feld.cvut.cz}
\address{Charles University, Faculty of Mathematics and Physics, Department of
Mathematical Analysis, Sokolovsk{\'a} 86, 186 75 Praha 8, Czech Republic}
\email{kalenda@karlin.mff.cuni.cz}
\address{Instituto de Matem{\'a}ticas de la Universidad de Granada (IMAG), Departamento de An{\'a}lisis Matem{\'a}tico, Facultad de
Ciencias, Universidad de Granada, 18071 Granada, Spain.}
\email{aperalta@ugr.es}



\begin{document}

\begin{abstract}
We introduce, investigate and compare several order type relations on the set of tripotents in a JB$^*$-triple.
The main two relations we address are $\le_h$ and $\le_n$.  We say that $u\le_h e$ (or $u\le_n e$) if $u$ is a self-adjoint (or normal) element of the Peirce-2 subspace associated to $e$ considered as a unital JB$^*$-algebra with unit $e$.
It turns out that these relations need not be transitive, so we consider their transitive hulls as well. Properties of these transitive hulls appear to be closely connected with types of von Neumann algebras, with the results on products of symmetries, with determinants in finite-dimensional Cartan factors, with finiteness and other structural properties of JBW$^*$-triples.
\end{abstract}

\maketitle

\section{Introduction}

A {\em JB$^*$-triple} is a complex Banach space $E$ equipped with a continuous mapping $\J{\cdot}{\cdot}{\cdot}:E^3\to E$ ({\em triple product}) which is symmetric and bilinear in the outer variables and conjugate linear in the second variable and satisfies, moreover, the following properties:
\begin{enumerate}[$(a)$]
    \item $\J xy{\J abc}=\J{\J xya}bc-\J a{\J yxb}c+\J ab{\J xyc}$ for $x,$ $y,$ $a,$ $b,$ $c\in E$ ({\em Jordan identity});
    \item for any $a\in E$ the operator $L(a,a):x\mapsto \J aax$ is a hermitian operator with non-negative spectrum;
    \item $\norm{\J xxx}=\norm{x}^3$ for $x\in E$.
\end{enumerate}
We recall that an operator $T$ on a Banach space is {\em hermitian} if $\norm{e^{i\alpha T}}=1$ for each $\alpha\in\er$.

Any C$^*$-algebra becomes a JB$^*$-triple if we equip it with the triple product defined by $\J abc=\frac12(ab^*c+cb^*a)$. More generally, any closed subspace of a C$^*$-algebra which is stable under the above-defined triple product, is a JB$^*$-triple (cf. \cite{harris1974bounded,kaup1983riemann}). Such spaces are called \emph{JC$^*$-triples}. However, there are some JB$^*$-triples which are not of this form (known as {\em exceptional JB$^*$-triples}, cf. Section~\ref{sec:spinetc} below).

The triple product on a C$^*$-algebra is an algebraic structure which captures the metric structure -- by the Kadison-Paterson-Sinclair theorem a linear bijection of two C$^*$-algebras is an isometry if and only if it preserves the triple product (see \cite{Paterson-Sinclair} or \cite[Theorem  2.2.19]{Cabrera-Rodriguez-vol1}). The same holds for linear bijections between JB$^*$-triples by Kaup's theorem (see \cite[Proposition 5.4]{kaup1977-maan} or \cite[Theorem 5.6.57]{Cabrera-Rodriguez-vol2}). This is closely related to another characterization of JB$^*$-triples as those complex Banach spaces such that biholomorphic  selfmaps of the unit ball act transitively on the ball (see \cite[Theorem 5.4]{kaup1983riemann} or \cite[Theorem 5.6.68]{Cabrera-Rodriguez-vol2}). The triple product then naturally arises from these biholomorphic maps (cf. \cite[Fact 5.6.29]{Cabrera-Rodriguez-vol2}). This witnesses that JB$^*$-triples are a natural class of Banach spaces in which the algebraic and metric structures are tightly connected.

It is known (see \cite[Corollary 3]{Friedman-Russo-GN}) that the triple product of each JB$^*$-triple $E$ satisfies the following property \begin{equation}\label{eq triple product is contractive} \|\{x,y,z\}\|\leq \|x\| \ \|y\| \ \|z\|, \hbox{ for all } x,y,z\in E.
\end{equation}

A JB$^*$-triple which is also a dual Banach space is called a \emph{JBW$^*$-triple}. A result by Barton and Friedman proves that any JBW$^*$-triple has a unique (isometric) predual (see e.g. \cite[Theorem 5.7.38]{Cabrera-Rodriguez-vol2}) and the triple product is separately weak$^*$-to-weak$^*$ continuous (see \cite[Theorem 5.7.20]{Cabrera-Rodriguez-vol2}). Moreover, an original result due to Dineen states that the bidual of any JB$^*$-triple $E$ is a JBW$^*$-triple and its triple product extends that on $E$ (see \cite[Proposition 5.7.10]{Cabrera-Rodriguez-vol2}).

 An important role in the study of JB$^*$-triples and especially JBW$^*$-triples is played by tripotents. Let us recall that a \emph{tripotent} in a JB$^*$-triple $E$ is an element $e\in E$ satisfying $\J eee=e$. In a C$^*$-algebra this formula reduces to $ee^*e=e$, which is a characterization of partial isometries. We also recall that an element $e$ of a C$^*$-algebra $A$ is a \emph{partial isometry} if $p_i(e)=e^*e$ and  $p_f(e)=ee^*$ are projections (i.e., self-adjoint idempotents). Then $p_i(e)$ is called the \emph{initial projection} and $p_f(e)$ is the \emph{final projection}. 
 
 There is a natural partial order (denoted by $\le$) on tripotents, defined in terms of orthogonality, and   generalizing the standard partial order on projections. In \cite{Finite} we studied two weaker preorders on tripotents (denoted by $\le_2$ and $\le_0$). The preorder $\le_2$ was used in \cite{hamhalter2019mwnc} (without giving the notation) to study the strong$^*$ topology and is implicitly mentioned already in  \cite[Lemma 1.14(1)]{horn1987ideal}. If $A$ is a unital C$^*$-algebra, then $e\le 1$ means that $e$ is a projection in $A$ and $e\le_2 1$ is valid for any partial isometry in $A$. More concretely, two partial isometries $e,v\in A$ satisfy $e\le_2 v$ if $e e^* A e^* e \subseteq vv^* A v^* v$. There is a large gap between these two relations and there are some intermediate types of partial isometries -- for example the self-adjoint ones (i.e., satisfying $e^*=e$)
 or the normal ones (i.e., those satisfying $e^*e=ee^*$). As we shall see later, the triple product of $A$ can be employed to characterize normality. Namely, if $p$ is a projection in $A$ and $e$ is a partial isometry in $p A p$ (i.e. $\{p,p,e\} =e$), then $e$ is a normal element (i.e. $e e^* = e^* e$) if and only if $\{e,e,p\}$ is a projection, or equivalently a partial isometry (cf. page \pageref{eq first motivation for leqn}). 
 
 On the other hand, if $e$ and $v$ are two partial isometries in $A$ with $v\in ee^* A e^* e,$ and the latter is regarded as a C$^*$-algebra with product $x \cdot_e y := x e^* y$ and involution $x^{*_e} = e x^* e$ ($x,y\in A$), then $v$ is self-adjoint with respect to the new structure if and only if $\{e,v,e \} = v$ (actually this condition also guarantees that $v\in ee^* A e^* e$). 
 
 Since the previous characterizations of normality and self-adjointness are given in terms of triple products, we shall abstract their meaning to define two new relations $\le_n$ and $\le_h$ for tripotents in a general JB$^*$-triple (see section \ref{sec:relations} for the concrete definitions).

In the present paper we define and study several order type relations on tripotents inspired by the mentioned gap.
 
 The paper is organized as follows: In the rest of the introductory section we recall some background information on JB$^*$-triples, JB$^*$-algebras, tripotents, the usual partial order $\le$ and the two above-mentioned preorders.
 
 In Section~\ref{sec:relations} we introduce the intermediate relations, give their
 basic properties and characterizations and compare them to each other.
 
 In Section~\ref{sec:JBW*} we look at the relations in JBW$^*$-triples and present several auxilliary tools to study them.
 
 Then, in several subsequent sections we provide a detailed study of the relations in the individual summands from the standard representation of JBW$^*$-triples recalled in \eqref{eq:representation of JBW* triples} below.
 
 In the final section we give an overview of the results and open problems.
 
 \subsection{JB$^*$-algebras and JBW$^*$-algebras}

Recall that  a {\em JB$^*$-algebra} is a complex Banach space $B$  equipped with a product $\circ$ and an involution $*$ satisfying the following properties.
\begin{enumerate}[$(a)$]
    \item $(B,+,\circ)$ is a (possibly) non-associative complex algebra;
    \item $x\circ y=y\circ x$ for $x,y\in B$;
    \item  $(x\circ x)\circ (y\circ x)=((x\circ x)\circ y)\circ x$ for $x,y\in B$ (\emph{Jordan identity});
    \item $\norm{x\circ y}\le\norm{x}\norm{y}$ for $x,y\in B$;
    \item $*$ is an involution on the algebra $(B,+,\circ)$;
    \item $\norm{2(x\circ x^*)\circ x-(x\circ x)\circ x^*}=\norm{x}^3$ for $x\in B$.
\end{enumerate}
Note that the conditions $(a)$--$(c)$ are the axioms defining {\em complex Jordan algebra} (cf. \cite[\S2.4.1]{hanche1984jordan}), if we add the condition $(d)$, we get a {\em complex Jordan Banach algebra}.

Any C$^*$-algebra becomes a JB$^*$-algebra if equipped with the Jordan product $x\circ y=\frac12(xy+yx)$. More generally, any closed subspace of a C$^*$-algebra which is stable under involution and the Jordan product is a JB$^*$-algebra. There are some JB$^*$-algebras which are not of this form (named {\em exceptional} JB$^*$-algebras, cf. Section~\ref{subsec:c6} below).

Further, any JB$^*$-algebra  becomes a JB$^*$-triple when equipped with the triple product
\begin{equation}\label{eq:triple product in JB*-algebra}
    \J xyz=(x\circ y^*)\circ z+x\circ(y^*\circ z)-(x\circ z)\circ y^*,
\end{equation}
see \cite[Theorem 4.1.45]{Cabrera-Rodriguez-vol1}.
Note that the condition $(f)$ from the definition of JB$^*$-algebras yields the condition $(c)$ from the definition of JB$^*$-triples.

An element $a$ in a unital JB$^*$-algebra $B$ is called \emph{invertible} if there exists a (unique) element $b$ (called the \emph{Jordan inverse} of $a$ and denoted by $a^{-1}$) satisfying $a\circ b = 1$ and $a^2\circ b = a,$ equivalently, the mapping $U_a : B\to B$ defined by 
$$U_a (x) =  2 (a\circ x) \circ a - a^2\circ x\ (=\J a{x^*}a)$$  is invertible 
(cf. \cite[3.2.9]{hanche1984jordan} or \cite[\S 4.1.1]{Cabrera-Rodriguez-vol1}). Each element $u\in B$ whose Jordan inverse is $u^*$ is called \emph{unitary}.  

Similarly as in the case of triples, a JB$^*$-algebra which is a dual Banach space is called a \emph{JBW$^*$-algebra}. Again, the predual is (isometrically) unique and, moreover, the Jordan product is separately weak$^*$-to-weak$^*$ continuous and the involution is weak$^*$-to-weak$^*$ continuous (cf. \cite[Theorem 4.4.16 and Corollaries 4.5.4 and 4.1.6]{hanche1984jordan} or \cite[Theorem 5.1.29, Corollary 5.1.41 and Fact 5.1.42]{Cabrera-Rodriguez-vol2}).

\subsection{Tripotents, Peirce decomposition and three preorders}

If $u$ is a tripotent in a JB$^*$-triple $E$, it generates a decomposition of $E$ in terms of the eigenspaces of the operator $L(u,u)$ (recall that it is defined by $L(u,u)x=\J uux$ for $x\in E$). This operator has eigenvalues contained in the set $\{0,\frac12,1\}$ and the mentioned decomposition is formed by the following  {\em Peirce subspaces}:
$$E_j(u)=\left\{x\in E\setsep \J uux=\frac j2 x\right\}\mbox{ for }j=0,1,2.$$
It is known that $E=E_2(u)\oplus E_1(u)\oplus E_0(u)$ and that the canonical projections (called {\em Peirce projections} and denoted by $P_j(u)$, $j=0,1,2$) have norm one or zero \cite[Corollary 1.2]{Friedman-Russo}. Further, if $E$ is a JBW$^*$-triple, the Peirce subspaces are weak$^*$-closed and the Peirce projections are weak$^*$-to-weak$^*$ continuous since they can be described in terms of the triple product (cf. the concrete expression in \eqref{eq description of Peirce projections} below).

Moreover, it is easy to check, that
\begin{equation}
\J{E_j(u)}{E_k(u)}{E_l(u)}\subset E_{j-k+l}(u),
\end{equation}
where the right-hand side is defined to be $\{0\}$ if $j-k+l\notin\{0,1,2\}$.
Moreover, it is known (but not obvious) that
\begin{equation}
\J{E_2(u)}{E_0(u)}E=\J{E_0(u)}{E_2(u)}E=\{0\}.    
\end{equation}
The two above rules are known and will be referred to as the {\em Peirce arithmetics} or the {\em Peirce calculus}. It easily follows that $E_j(u)$ is a JB$^*$-subtriple of $E$ for $j=0,1,2$. 

The following formulas for the Peirce projections may be easily deduced from the definitions.
\begin{equation}\label{eq description of Peirce projections}
\begin{aligned}
P_2(u)x&=2L(u,u)^2x-L(u,u)x,\\
P_1(u)x&=4(L(u,u)x-L(u,u)^2x),\\
P_0(u)x&=x-3L(u,u)x+2L(u,u)^2x.
\end{aligned}    
\end{equation}
Another useful formula for $P_2(u)$ is
\begin{equation}
    P_2(u)x=Q(u)^2x \mbox{ where }Q(u)x=\J uxu\mbox{ for }x\in E.
\end{equation}

A tripotent $u$ is called {\em complete} if $E_0(u)=\{0\}$ and it is called {\em unitary} if $E=E_2(u)$. Recall that each unital JB$^*$-algebra $B$ can be also regarded as a JB$^*$-triple with the triple product \eqref{eq:triple product in JB*-algebra}, so in this case we have two notions of unitary elements. Fortunately, they coincide, that is, an element $u\in B$ is unitary as an element in a unital JB$^*$-algebra if and only if it is unitary in the triple sense (cf. \cite[Proposition 4.3]{braun1978holomorphic} or \cite[Theorem 4.2.28]{Cabrera-Rodriguez-vol1}).

In a JB$^*$-triple there need not be any complete tripotent (in fact, there need not be any nonzero tripotent, take for example the non-unital C$^*$-algebra $\C_0(\er)$); but in a JBW$^*$-triple there is an abundance of complete tripotents, as they are exactly the extreme points of the unit ball.

On the other hand, JBW$^*$-triples need not contain any unitary element. For example, the space of $1\times 2$ complex matrices
(with the structure of the space of linear functionals on the two-dimensional Hilbert space) is a JBW$^*$-triple without unitary elements. In fact, JB$^*$-triples with a unitary element are just the triples coming from unital JB$^*$-algebras (see \cite[examples in page 525]{kaup1983riemann} or \cite[Theorem 4.1.55]{Cabrera-Rodriguez-vol1}).

Indeed, if $E$ is a JB$^*$-triple with a unitary tripotent $e$, it becomes a unital JB$^*$-algebra if it is equipped with the operations
\begin{equation}\label{eq:algebra from unitary}
x\circ_e y=\J xey\mbox{ and }x^{*_e}=\J exe.\end{equation} In particular, for each tripotent $v$ in a JB$^*$-triple $F$ the Peirce-$2$ subspace $F_2(v)$ is a unital JB$^*$-algebra. Furthermore, $v$ is called an \emph{abelian
tripotent} if the subtriple $F_2(v)$ is an associative JB$^*$-algebra --equivalently, a commutative unital  C$^*$-algebra -- (cf. \cite{horn1987classification,horn1988classification}).  

Now we recall definitions of three preorders studied in \cite{Finite}.
Let $E$ be a JB$^*$-triple and let $e,u\in E$ be two tripotents.
We say that
\begin{itemize}
    \item $u\le e$ if $e-u$ is a tripotent orthogonal to $u$;
    \item $u\le_2 e$ if $u\in E_2(e)$;
    \item $u\le_0 e$ if $E_0(e)\subset E_0(u)$.
\end{itemize}

Here $\le$ is the standard partial order on tripotents used in \cite{Friedman-Russo, BattagliaOrder1991, horn1987ideal,  horn1987classification, horn1988classification} (and elsewhere). Recall that tripotents $e_1,e_2\in E$ are orthogonal if $L(e_1,e_2)=0$ (or, equivalently, $e_1\in E_0(e_2)$) and that this relation is symmetric.

Relations $\le_2$ and $\le_0$ are preorders -- reflexive and transitive, but not antisymmetric (see \cite{Finite}).
The relation $\le_2$ was used  already in \cite[Sections 6 and 7]{hamhalter2019mwnc}  and \cite{HKPP-BF}, without introducing the notation.

 Following the notation of \cite{Finite}  we will write $u\sim_2 e$ if $u\le_2 e$ and $e\le_2 u$. If $u\le_2 e$ and $e\not\le_2 u$, we write $u<_2 e$. The relations $\sim_0$ and $<_0$ have the analogous meaning. In this paper we shall focus on 
a variety of relations  lying in between $\le$ and $\le_2$. We will not consider the relations $\le_0$ and $\sim_0$
as they have very different nature and were studied in \cite{Finite} (mainly in Section 2).

The following proposition summarizes known characterizations of the partial order gathered from different papers and authors.

\begin{prop}{\rm\cite[Proposition 2.4]{Finite}}\label{P:order char} Let $u,e$ be two tripotents in a JB$^*$-triple $E$. The following assertions are equivalent. 
\begin{enumerate}[$(i)$]
    \item $u\le e$;
    \item $u=\J ueu$;
    \item $u=\J uue$;
    \item $u=P_2(u)e$;
    \item $L(e-u,u)=0$;
    \item $L(u,e-u)=0$;
    \item $u$ is a projection in the JB$^*$-algebra $E_2(e)$;
    \item $E_2(u)$ is a JB$^*$-subalgebra of $E_2(e)$.
\end{enumerate}
\end{prop}

We will need also the following easy properties.

\begin{prop}{\rm\cite[Proposition 2.5]{Finite}}\label{P:order properties}
Let $E$ be a JB$^*$-triple. The relation $\le$ is a partial order on the set of all tripotents in $E$. Moreover, given tripotents $u,v,e\in E$ the following holds.
\begin{enumerate}[$(a)$]
    \item If $u\le e$, then $\alpha u\le \alpha e$ for any complex unit $\alpha$;
    \item If $u\le e$, $v\le e$ and $u,v$ are orthogonal, then $u+v\le e$.
 \end{enumerate}
\end{prop}

\section{Intermediate order type relations}\label{sec:relations}

In this section we introduce a variety of order type relations on tripotents lying in between the preorder $\le_2$ and the standard partial order $\le$. 

Two main relations which inspired our research are the following ones.
We say that
\begin{itemize}
    \item $u\le_h e$ if $u=\J eue$;
    \item $u\le_n e$ if $u=\J eeu$ and $\J uue$ is a tripotent. 
\end{itemize}

Note  that $u\le_h e$ if and only if $u$ is a self-adjoint tripotent in $E_2(e)$. Indeed, it follows by Peirce arithmetic that $\J eue\in E_2(e)$, so the equality $u=\J eue$ implies $u\in E_2(e)$. Further, $\J eue=u^{*_e}$ in the JB$^*$-algebra $E_2(e)$.

Further, $u\le_n e$ means that $u$ is a `normal tripotent' in $E_2(e)$:
If $E$ is a C$^*$-algebra and $e$ is a projection, it means that $u$ is a normal partial isometry in $E_2(e)$.
Indeed, the equality $u=\J eeu$ implies $u\in E_2(e)$. Further, let us analyze the assumption that $\J uue$ is a tripotent (i.e., a partial isometry). Note that
$$\J uue=\frac12(uu^*e+eu^*u)=\frac12(p_f(u)e+ep_i(u))=\frac12(p_f(u)+p_i(u)).$$
It is easy to check that the arithmetic mean of two projections is a partial isometry if an only if these two projections coincide.  Therefore, $\J uue$ is a tripotent if and only if $p_f(u)=p_i(u)$, i.e, $uu^*=u^*u$, in other words if and only if $u$ is a normal operator.\label{eq first motivation for leqn}

The relations $\le_h$ and $\le_n$ are clearly reflexive, but they are not transitive -- as witnessed by counterexamples below. Therefore we will consider also their transitive hulls.

In fact, the interval between $\le$ and $\le_2$ has a richer structure. In several subsections we will describe and analyze several relations including the above-defined relations $\le_h$ and $\le_n$. Later we will compare them.

\subsection{Modifications of the classical partial order by a multiple}
It is obvious that $u\sim_2 \alpha u$ whenever $u$ is a tripotent and $\alpha$ is a complex unit. But $u$ and $\alpha u$ are connected much more than by the coincidence of their Peirce decomposition. On the other hand, unless $\alpha=1$, they are incomparable with respect to $\le$. This inspires definitions of the following two preorders.

Let $E$ be a JB$^*$-triple and let $e,u\in E$ be two tripotents.
We say that
\begin{itemize}
    \item $u\le_r e$ if $u\le e$ or $-u\le e$;
    \item $u\le_c e$ if there is a complex unit $\alpha$ with $\alpha u\le e$.
\end{itemize}
The relations $\sim_r$, $<_r$, $\sim_c$ and $<_c$ have the obvious meaning.
In the following two propositions we collect properties of the relations $\le_r$ and $\le_c$.

\begin{prop}\label{P:charact ler}
Let $E$ be a JB$^*$-triple. Then the relation $\le_r$ is a preorder on the set of tripotents in $E$. Moreover, given tripotents $u,e\in E$, the following holds.
\begin{enumerate}[$(a)$]
    \item Let $u,e\in E$ be two tripotents. Then $u\le_r e$ if and only if either $u$ or $-u$ is a projection in  the JB$^*$-algebra $E_2(e)$.
    \item For two tripotents $e,u\in E$ we have $e\sim_r u$ if and only if either $e=u$ or $e=-u$.
\end{enumerate}
\end{prop}

\begin{proof} Reflexivity of $\le_r$ is obvious. Transitivity follows easily from the transitivity of $\le$ using Proposition~\ref{P:order properties}$(a)$.
Thus $\le_r$ is indeed a preorder.

$(a)$ This follows from the definition and Proposition~\ref{P:order char} (using  property $(vii)$). 

$(b)$ The `if part' is obvious. To see the converse assume $e\sim_r u$. We distinguish the following cases:

If $e\le u$ and $u\le e$, then $e=u$. If $-e\le u$ and $-u\le e$, then $u\le -e$ (by Proposition~\ref{P:order properties}$(a)$), so $u=-e$.

Assume $e\le u$ and $-u\le e$. Then $u\le -e$, hence $e\le -e$. We deduce that $-e\le e$ as well, thus $e=-e$, so $e=0$. It follows that $u=0$ as well.

The fourth case is similar.
\end{proof}

\begin{prop}\label{P:charact lec} Let $E$ be a JB$^*$-triple. Then the relation $\le_c$ is a preorder on the set of tripotents in $E$. Moreover, given tripotents $u,e\in E$, the following holds.
\begin{enumerate}[$(a)$]
    \item Let $u,e\in E$ be two tripotents. Then $u\le_c e$ if and only if there is a projection $p\in E_2(e)$ and a complex unit $\alpha$ such that $u=\alpha p$.
    \item For two tripotents $e,u\in E$ we have $e\sim_c u$ if and only $u=\alpha e$ for a complex unit $\alpha$. 
    \end{enumerate}
\end{prop}

\begin{proof}
Reflexivity of $\le_c$ is obvious.  Transitivity follows easily from the transitivity of $\le$ using Proposition~\ref{P:order properties}$(a)$. 
Thus $\le_c$ is indeed a preorder.

$(a)$ This follows from the definition and Proposition~\ref{P:order char} (using property $(vii)$). 

$(b)$ The `if part' is obvious. To see the converse assume $e\sim_c u$. It means that there are two complex units $\alpha,\beta$ such that $\alpha u\le e$ and $\beta e\le u$. By Proposition~\ref{P:order properties}(a) we get $e\le \overline{\beta}u$, thus $\alpha u\le\overline{\beta}u$.
If $u=0$, then necessarily $e=0$ as well. Assume $u\ne0$. Then 
$$\alpha u=\J{\alpha u}{\alpha u}{\overline{\beta} u}= \overline{\beta} \J uuu=\overline{\beta} u,$$
thus $\alpha=\overline{\beta}$. The inequalities $\alpha u\le e\le\overline{\beta}u$ then yield $e=\alpha u$.
\end{proof}

\subsection{The relation $\le_h$ and its transitive hull}

In this subsection we  provide some characterizations of the relation $\le_h$ and its variants.
We recall that $u\le_h e$ if $u=\J eue$. We write $u\sim_h e$ if $u\le_h e$ and $e\le_h u$ and $u<_h e$ if $u\le_h e$ and $e\not\le_h u$.
Note that even though we use an order-like notation, these relations are not transitive (see Example~\ref{ex:leh not transitive} below).

 We start by two propositions characterizing $\le_h$ and $\sim_h$.

\begin{prop}\label{P:charact leh}
Let $E$ be a JB$^*$-triple and let $u,e\in E$ be two tripotents. The following assertions are equivalent.
\begin{enumerate}[$(i)$]
    \item $u\le_h e$;
    \item $u$ is a self-adjoint element in the JB$^*$-algebra $E_2(e)$;
    \item $u=p-q$, where $p,q\in E_2(e)$ are two mutually orthogonal projections;
   \item $u=v-w$, where $v$ and $w$ are two orthogonal tripotents with $v\le e$ and $w\le e$.
\end{enumerate}
\end{prop}

\begin{proof}
$(i)\Rightarrow(ii)$ Assume $u\le_h e$, i.e., $u=\J eue$. It means that $u=Q(e)u$, hence $P_2(e)u=Q(e)^2u=u$. Therefore $u\in E_2(e)$. Moreover,
$$u^{*_e}=\J eue=u,$$
so $u$ is self-adjoint.

$(ii)\Rightarrow(iii)$ This is well known. Let us give a proof for the sake of completeness.
Let us work in the unital JB$^*$-algebra $E_2(e)$. Assume that $u$ is a self-adjoint tripotent. Let $B$ be the JB$^*$-subalgebra generated by $u$ and by the unit (which is $e$). By \cite[Lemma 2.4.5, Theorem 3.2.2 and Remark 3.2.3]{hanche1984jordan} $B$ is associative, so it is a unital commutative C$^*$-algebra, i.e., it may be represented as a $C(K)$ space for a compact $K$. The element $u$ is self-adjoint, so it is a real-valued function. Moreover, it is a tripotent, hence $u^3=u$. It follows that $u$ attains only values $0,1,-1$. Therefore $p=\frac12(u^2+u)$ and $q=\frac12(u^2-u)$ are characteristic functions of disjoint sets, hence mutually orthogonal projections and $u=p-q$.

$(iii)\Rightarrow(iv)$ This is obvious.

$(iv)\Rightarrow(i)$ If $v\le e$ and $w\le e$, by Proposition~\ref{P:order char} the elements $v$ and $w$ are self-adjoint in $E_2(e)$ (we use  property $(vii)$), hence
$$v=\J eve\mbox{ and }w=\J ewe.$$
If $u=v-w$, we deduce that $u=\J eue$.
\end{proof}

\begin{prop}\label{P:charact simh}
Let $E$ be a JB$^*$-triple and let $u,e\in E$ be two tripotents. The following assertions are equivalent.
\begin{enumerate}[$(i)$]
    \item $u\sim_h e$;
    \item $u\le_h e$ and $u\sim_2 e$;
    \item There are two orthogonal tripotents $v,w\in E$ such that $e=v+w$ and $u=v-w$;
    \item $\frac12(e+u)$ and $ \frac12(e-u)$ are tripotents.
\end{enumerate}
\end{prop}

\begin{proof}
$(i)\Rightarrow(ii)$ Assume $u\sim_h e$. Then clearly $u\le_h e$. Moreover,  by Proposition~\ref{P:charact leh} we have $u\in E_2(e)$ and $e\in E_2(u)$, i.e., $u\sim_2 e$.

$(ii)\Rightarrow(iii)$ Assume $u\le_h e$ and $u\sim_2 e$. Then $E_2(e)=E_2(u)$ (by \cite[Proposition 2.3]{Finite}, cf. also \cite[Proposition 6.5]{hamhalter2019mwnc}).

Further, using Proposition~\ref{P:charact leh} 
we get two orthogonal tripotents $v,w$ such that $v\le e$, $w\le e$ and $u=v-w$. By Proposition~\ref{P:order properties}$(b)$ we have $v+w\le e$. Finally, clearly $L(v+w,v+w)=L(v-w,v-w)$, so $P_2(v+w)=P_2(v-w)$. It follows that $E_2(v+w)=E_2(v-w)=E_2(u)=E_2(e)$. Hence $e-(v+w)$  is a tripotent orthogonal to $v+w$ which belongs to $E_2(v+w)$, so necessarily $e=v+w$.

$(iii)\Rightarrow(i)$ This follows from Proposition~\ref{P:charact leh} (using property $(iv)$).

$(iii)\Rightarrow(iv)$ Let $v,w$ be given by $(iii)$. Then it is easy to observe that $\frac12(e+u)=v$ and $\frac12(e-u)=w$.

$(iv)\Rightarrow(iii)$ Set $v=\frac12(e+u)$ and $w=\frac12(e-u)$. Assuming $(iv)$, $v$ and $w$ are tripotents. Moreover, $v+w=e$ and $v-w=u$. Since $e$ and $u$ are tripotents,
\cite[Lemma 3.6]{isidro1995real}
shows that $v$ and $w$ are orthogonal.
\end{proof}

\begin{example}\label{ex:leh not transitive}
The relations $\le_h$ and $\sim_h$ are not transitive on $M_2$.
\end{example} 

\begin{proof}
Let 
$$e=\begin{pmatrix} 1 & 0 \\ 0 & 1 \end{pmatrix},\quad u=\begin{pmatrix} 0 & -1 \\ -1 & 0 \end{pmatrix}, \quad
v=\begin{pmatrix} i & 0 \\ 0 & -i \end{pmatrix}.$$
Then $e,u,v$ are unitary matrices, so tripotents in $M_2$ satisfying $e\sim_2 u\sim_2 v$.
Moreover, clearly $u\le_h e$ ($e$ is the unit matrix and $u$ is self-adjoint). We further have
$v\le_{h} u$ as
$$\J uvu=uv^*u=\begin{pmatrix} 0 & -1 \\ -1 & 0 \end{pmatrix}\begin{pmatrix} -i & 0 \\ 0 & i \end{pmatrix}\begin{pmatrix} 0 & -1 \\ -1 & 0 \end{pmatrix}
=\begin{pmatrix} 0 & -i \\ i & 0 \end{pmatrix}\begin{pmatrix} 0 & -1 \\ -1 & 0 \end{pmatrix}=v.$$
So, $v\le_h u\le_h e$. By Proposition~\ref{P:charact simh} (using equivalence $(i)\Leftrightarrow(ii)$) we deduce $v\sim_h u\sim_h e$.

However, $e$ and $v$ are incomparable for $\le_h$. Indeed, since $e$ is the unit matrix and $v^*=-v\ne v$, we have $v\not\le_he$. Using again Proposition~\ref{P:charact simh} we deduce that $e\not\le_h v$ as well. 
\end{proof}

We continue by a lemma on factorization of $\le_h$ via $\le$ and $\sim_h$.

\begin{lemma}\label{L:factor leh} Let $E$ be a JB$^*$-triple and $u,e\in E$ two tripotents. Consider the following statements.
\begin{enumerate}[$(i)$]
    \item $u\le_h e$;
    \item there is a tripotent $v$ such that $u\sim_h v$ and $v\le e$;
    \item there is a tripotent $w$ such that $e\sim_h w$ and $u\le w$.
\end{enumerate}
Then 
$$(i)\Leftrightarrow(ii)\Rightarrow(iii).$$
The implication $(iii)\Rightarrow(ii)$ fails for example in $E=M_2$.
\end{lemma}

\begin{proof}
$(i)\Rightarrow(ii)\&(iii)$ If $u\le_h e$, by Proposition~\ref{P:charact leh} $u=p-q$, where $p,q$ are orthogonal tripotents, $p\le e$, $q\le e$. It is enough to take $v=p+q$ and $w=e-2q$ (and use Proposition~\ref{P:charact simh}).

$(ii)\Rightarrow(i)$ If $u\sim_h v$, by Proposition~\ref{P:charact simh} there are orthogonal tripotents $p,q$ such that $u=p-q$ and $v=p+q$. Since $v\le e$, clearly $p\le e$ and $q\le e$. Due to Proposition~\ref{P:charact leh} this completes the proof.

A counterexample to $(iii)\Rightarrow(i)$ is given in Example~\ref{ex:len not transitive}$(c)$ below.
\end{proof}

Since the relation $\le_h$ is not transitive, it natural to consider its transitive hull $\le_{h,t}$, i.e., $u\le_{h,t} e$ if there are tripotents
$$e=v_0,v_1,\dots,v_k=u$$
such that
$v_j\le_h v_{j-1}$ for $j=1,\dots,k$.

Then $\le_{h,t}$ is clearly a preorder, hence the relation $\sim_{h,t}$ defined by $$e\sim_{h,t}u \equiv^{\mbox{def}} e\le_{h,t}u\mbox{ and }u\le_{h,t}e$$
is an equivalence relation.
The symbol $<_{ht}$ has then the obvious meaning.

The following lemma provides a factorization of the preorder $\le_{ht}$.
\begin{lemma}\label{L:factor leht}
Let $E$ be a JB$^*$-triple and $u,e\in E$ two tripotents. Then $u\le_{h,t} e$ if and only if there is a tripotent $v\in E$ with $u\le v$ and $v\sim_{h,t}e$.
\end{lemma}

\begin{proof}
The `if' part is obvious. Let us prove the `only if part'. 
Assume that $u\le_{h,t}e$. Then there are tripotents
$$e=v_0,v_1,\dots,v_k=u$$
such that
$v_j\le_h v_{j-1}$ for $j=1,\dots,k$. We will prove the statement by induction on $k$. The case $k=1$ follows from the implication $(i)\Rightarrow(iii)$ of Lemma~\ref{L:factor leh}. Assume that $k>1$ and the statement holds for $k-1$. Then there is a tripotent $w'\in E$ such that $w'\sim_{h,t} e$ and $v_{k-1}\le w'$. Note that then $u=v_k\le_h w'$. Indeed, $u\in E_2(v_{k-1})$ and $v_{k-1}=P_2(v_{k-1})(w')$, hence
$$\J u{w'}u=\J u{v_{k-1}}u=u,$$
where the first equality follows by Peirce arithmetics.
Thus, using again the implication $(i)\Rightarrow(iii)$ of Lemma~\ref{L:factor leh} we get a tripotent $w$ with $w\sim_h w'$ and $u\le w$. Then $w\sim_{h,t}e$ and the proof is complete.
\end{proof}

We continue by a characterization of the relation $\sim_{h,t}$.

\begin{prop}\label{P:simht-char}
Let $E$ be a JB$^*$-triple and $u,e\in E$ two tripotents. The following assertions are equivalent.
\begin{enumerate}[$(i)$]
    \item $u\sim_{h,t}e$;
    \item $u\le_{h,t}e$ and $u\sim_2 e$;
    \item there are tripotents $v_1,\dots,v_k\in E$ such that
    $$u=v_1\sim_h v_2\sim_h\dots\sim_h v_k=e.$$
\end{enumerate}
In particular, the relation $\sim_{h,t}$ coincides with the transitive hull of $\sim_h$.
\end{prop}

\begin{proof} The `in particular' part follows from the equivalence $(i)\Leftrightarrow(iii)$. So, let us prove the equivalences:

$(i)\Rightarrow(ii)$ Assume $u\sim_{h,t} e$. Then $u\le_{h,t} e$ and $e\le_{h,t} u$. Hence the first statement of $(ii)$ is obviously fulfilled.
Further, since $u_1\le_h u_2$ implies $u_1\le_2 u_2$ (by Proposition~\ref{P:charact leh}) and the relation $\le_2$ is transitive, we deduce that $u\le_2 e$ and $e\le_2u$, i.e., $u\sim_2 e$.

$(ii)\Rightarrow(iii)$
 Assume $u\le_{h,t}e$ and $u\sim_2 e$. It follows that there is a finite sequence
$$u=v_1\le_h v_2\le_h v_3\le_h\dots\le_h v_k=e.$$
Hence,
$$u=v_1\le_2 v_2\le_2 v_3\le_2\dots\le_2 v_k=e$$
(by Proposition~\ref{P:charact leh}). Since $u\sim_2e$, we deduce that
$$u=v_1\sim_2 v_2\sim_2 v_3\sim_2\dots\sim_2 v_k=e.$$
Finally, we apply Proposition~\ref{P:charact simh} and get
$$u=v_1\sim_h v_2\sim_h v_3\sim_h\dots\sim_h v_k=e.$$
Hence, $(iii)$ holds.

$(iii)\Rightarrow(i)$ This is obvious.
\end{proof}

\subsection{Modification of the relation $\le_h$ by a multiple}

Let $E$ be a JB$^*$-triple and let $e,u\in E$ be two tripotents. It is clear that $u\le_h e$ implies $-u\le_h e$. It is natural to define the following weaker relation.

We say that $u\le_{hc} e$ if $\alpha u\le_h e$ for a complex unit $\alpha$. The relations $\sim_{hc}$ and $<_{hc}$ have the obvious meaning.

Although the relation $\le_{hc}$ has originally rather an auxilliary role, later we will see that it is really natural because its transitive hull in many cases coincides with $\le_2$.

Note that the relations $\le_{hc}$ and $\sim_{hc}$ are very close to $\le_h$ and $\sim_h$. So, natural modifications of the results from the previous subsection hold. In the following proposition we give some characterizations of $\le_{hc}$ and $\sim_{hc}$, which may be completed by further properties in an obvious way.

\begin{prop}\label{P:hc-charact}
Let $E$ be a JB$^*$-triple and let $u,e\in E$ be two tripotents.
\begin{enumerate}[$(a)$] 
    \item The following assertions are equivalent.
    \begin{enumerate}[$(i)$]
        \item $u\le_{hc} e$;
        \item $u=\alpha\J eue$ for a complex unit $\alpha$;
        \item$ u$ is a scalar multiple of a self-adjoint element of $E_2(e)$.
    \end{enumerate}
      \item  The following assertions are equivalent.
      \begin{enumerate}[$(i)$]
          \item $u\sim_{hc}e$;
          \item $u\le_{hc} e$ and $u\sim_2 e$;
          \item $\alpha u\sim_h e$ for a complex unit $\alpha$.
      \end{enumerate}
    
\end{enumerate}
\end{prop}

\begin{proof}
$(a)$ The equivalence $(i)\Leftrightarrow(iii)$ follows from the definitions and Proposition~\ref{P:charact leh}.

$(i)\Rightarrow(ii)$ Assume $u\le_{hc} e$. Then there is a complex unit $\alpha$ such that $\alpha u\le_h e$, i.e., 
$$\alpha u=\J e{\alpha u}e=\overline{\alpha} \J eue,$$
thus
$$u=\overline{\alpha}^2\J eue.$$
It remains to observe that $\overline{\alpha}^2$ is a complex unit.

$(ii)\Rightarrow(i)$ Assume $u=\alpha \J eue$ for a complex unit $\alpha$. Then there is a complex unit $\beta$ with $\overline{\beta}^2=\alpha$. Hence
$$\J e{\beta u}e=\overline{\beta} \J eue=\overline{\beta}\overline{\alpha} u=\beta u,$$
thus $\beta u\le_h e$ and so $u\le_{hc} e$.

$(b)$ The implication $(i)\Rightarrow(ii)$ can be proved exactly in the same way as the analogous implication from Proposition~\ref{P:charact simh}.

$(ii)\Rightarrow(iii)$ Assume $u\le_{hc}e$ and $u\sim_2e$. By the definition there is a complex unit $\alpha$ with $\alpha u\le_{h}e$. Since clearly $\alpha u\sim_2 u$, we deduce that $\alpha u\sim_2e$. Thus, by Proposition~\ref{P:charact simh} we deduce that $\alpha u\sim_h e$.

$(iii)\Rightarrow(i)$ Assume that $\alpha u\sim_h e$. Then $\alpha u\le_h e$, hence $u\le_{hc}e$. Further, $e\le_h \alpha u$, thus $\overline{\alpha}e\le_h u$, therefore $e\le_{hc}u$. We conclude that $u\sim_{hc}e$.
\end{proof}

\begin{example}\label{ex:lehc not transitive}
\begin{enumerate}[$(a)$]
    \item There are tripotents $e,u,v\in M_2$ such that $e\sim_h u$, $u\sim_h v$, $e$ and $v$ are incomparable with respect to $\le_h$ and $e\sim_{hc} v$.
     \item There are tripotents $e,u,v\in M_2$ such that $e\sim_h u$, $u\sim_h v$ and $e$, $v$ are incomparable with respect to $\le_{hc}$. In particular, the relations $\le_{hc}$ and $\sim_{hc}$ are not transitive on $M_2$.
\end{enumerate}
\end{example}

\begin{proof}
$(a)$ The matrices from Example~\ref{ex:leh not transitive} work.

$(b)$ Set
Let 
$$e=\begin{pmatrix} 1 & 0 \\ 0 & 1 \end{pmatrix},\quad u=\begin{pmatrix} 0 & 1 \\ 1 & 0 \end{pmatrix}, \quad
v=\begin{pmatrix} \frac1{\sqrt2} & -\frac1{\sqrt2} \\ \frac1{\sqrt2} & \frac1{\sqrt2} \end{pmatrix}.$$
Then $e,u,v$ are unitary matrices, so tripotents in $M_2$ satisfying $e\sim_2 u\sim_2 v$.
Moreover, clearly $u\le_h e$ ($e$ is the unit matrix and $u$ is self-adjoint). Thus $u\sim_h e$ by Proposition~\ref{P:charact simh}.

Further, $v\le_h u$ as
$$\J uvu=uv^*u=\begin{pmatrix} 0 & 1 \\ 1 & 0 \end{pmatrix}\begin{pmatrix} \frac1{\sqrt2} & \frac1{\sqrt2} \\ -\frac1{\sqrt2} & \frac1{\sqrt2} \end{pmatrix}\begin{pmatrix} 0 & 1 \\ 1 & 0 \end{pmatrix}
=\begin{pmatrix} -\frac1{\sqrt2} & \frac1{\sqrt2} \\ \frac1{\sqrt2} & \frac1{\sqrt2} \end{pmatrix}\begin{pmatrix} 0 & 1 \\ 1 & 0 \end{pmatrix}=v$$
and hence $u\sim_h v$ by Proposition~\ref{P:charact simh}.

Thus we have $e\sim_h u$ and $u\sim_h v$. However, $e$ and $v$ are incomparable for $\le_{hc}$. Indeed, by Propositon~\ref{P:hc-charact}$(b)$ it is enough to prove that $v\not\le_{hc}e$. But this is clear, as $e$ is the unit matrix and $v$ is not a scalar multiple of a self-adjoint element.
\end{proof}

Since the relation $\le_{hc}$ is not transitive, we will consider its transitive hull $\le_{hc,t}$. The relations $\sim_{hc,t}$ and $<_{hc,t}$ then have the obvious meaning. The following proposition summarizes properties of these relations.

\begin{prop}\label{P:hct-charact}
Let $E$ be a JB$^*$-triple and let $u,e\in E$ be two tripotents.
\begin{enumerate}[$(a)$]
    \item $u\le_{hc,t} e$ $\Leftrightarrow$ $\alpha u\le_{h,t} e$ for a complex unit $\alpha$.
    \item The following assertions are equivalent.
    \begin{enumerate}[$(i)$]
        \item $u\sim_{hc,t}e$;
        \item $u\le_{hc,t} e$ and $u\sim_2 e$;
        \item $\alpha u\sim_{h,t} e$ for a complex unit $\alpha$.
    \end{enumerate}  
    \end{enumerate}
\end{prop}

\begin{proof}
$(a)$  The implication `$\Leftarrow$' is obvious. To prove the converse one assume that $u\le_{hc,t} e$. It means there there are tripotents
$$u=v_1,v_2,\dots,v_k=e$$
in $E$ such that
$$v_1\le_{hc} v_2\le_{hc}\dots\le_{hc} v_k.$$
By the very defintion there are complex units $\alpha_1,\dots,\alpha_{k-1}$ such that
$$\alpha_1 v_1\le_h v_2,\alpha_2 v_2\le_h v_3,\dots,\alpha_{k-1} v_{k-1}\le_h v_k.$$
Now, its clear that
$$\alpha_1\cdots\alpha_{k-1} v_1\le \alpha_2\cdots\alpha_{k-1}v_2\le_h \dots \le_h \alpha_{k-1}v_{k-1}\le_h v_k,$$
therefore
$$\alpha_1\cdots\alpha_{k-1} u\le_{h,t} e.$$

$(b):$ $(i)\Rightarrow(ii)$ Assume that $u\sim_{hc,t} e$. It means that $u\le_{hc,t}e$ and $e\le_{hc,t}u$. Thus the first condition in $(ii)$ is obviously true.

Furhter, by $(a)$ we get two complex units $\alpha,\beta$ such that $\alpha u\le_{h,t} e$ and $\beta e\le_{h,t}u$. In particular, $\alpha u\le_2 e$ and $\beta e\le_2 u$ (by Proposition~\ref{P:hc-charact}$(a)$). Since clearly $\alpha u\sim_2 u$ and $\beta e\sim_2 e$, we deduce that $u\sim_2 e$.

$(ii)\Rightarrow(iii)$ Assume $u\le_{hc,t}$ and $u\sim_2e$. By $(a)$ we get a complex unit $\alpha$ with $\alpha u\le_{h.t}e$. Since $\alpha u\sim_2 u$, we deduce that $\alpha u\sim_2 e$. Now, by Proposition~\ref{P:simht-char} we get $\alpha u\sim_{h,t} e$.

$(iii)\Rightarrow(i)$ This follows easily from $(a)$.
\end{proof}

\begin{cor}\label{cor:simhct=trhull}
The relation $\sim_{hc,t}$ coincides with the transitive hull of $\sim_{hc}$.
\end{cor}

\begin{proof}
It is clear that $\sim_{hc,t}$ is finer than the transitive hull of $\sim_{hc}$. Conversely, assume that $u\sim_{hc,t}e$. By Proposition~\ref{P:hct-charact} there is a complex unit $\alpha$ with $\alpha u\sim_{h,t}e$. Now we may conclude by using Proposition~\ref{P:simht-char}.
\end{proof}

\subsection{The relation $\le_n$ and its transitive hull}

Recall that $u\le_n e$ means that 
$u=\J eeu$ and $\J uue$ is a tripotent.  The symbols $\sim_n$ and $<_n$ will have the obvious meaning. The following proposition contains a basic characterization of the relation $\le_n$.

\begin{prop}\label{P:charact len} Let $E$ be a JB$^*$-triple and let $u,e\in E$ be two tripotents. Then the following assertions are equivalent:
\begin{enumerate}[$(i)$]
\item $u\le_n e$;
\item $u\in E_2(e)$ and $\J uue$ is a tripotent satisfying $\J uue\le e$;
\item $u\in E_2(e)$ and $u\circ_e u^{*_e}$ is a projection in $E_2(e)$.
\end{enumerate}
\end{prop}

\begin{proof}
$(i)\Rightarrow(iii)$ Assume $u\le_n e$. Then $u=\J eeu$, hence $u\in E_2(e)$. 
Moreover, $\J uue=u\circ_e u^{*_e}$, so it is a positive element in the JB$^*$-algebra $E_2(e)$.
By the assumption it is a tripotent. Being positive, it is self-adjoint, so by Proposition~\ref{P:charact leh} it is the difference of a pair of mutually orthogonal projections in $E_2(e)$. But since it is even positive, it must be a projection.

The implications $(iii)\Rightarrow(ii)\Rightarrow(i)$ are obvious.
\end{proof}

We continue by a characterization of $\sim_n$, which is very easy.

\begin{prop}\label{p:simn=sim2}
Let $E$ be a JB$^*$-triple and let $u,e\in E$ be two tripotents.
Then
$$u\sim_n e \Leftrightarrow u\sim_2 e.$$
\end{prop}

\begin{proof}
The implication $\Rightarrow$ is obvious. To prove the converse assume $u\sim_2 e$. It means that $\J uue=e$ and $\J eeu=u$, so by the very definition $u\sim_n e$.
\end{proof}

The following lemma is a key observation which helps to understand the relation $\le_n$.

\begin{lemma}\label{L:len implies sim2} Let $E$ be a JB$^*$-triple and $u,e\in E$ two tripotents.
If $u\le_n e$, then $u\sim_2\J uue$.
\end{lemma}

\begin{proof} Assume $u\le_n e$. We know that both $u$ and $\J uue$ are tripotents belonging to $E_2(e)$. So, without loss of generality we may assume that $E$ is a unital JB$^*$-algebra and $e=1$.

So, $u$ is a tripotent and $u^*\circ u$ is a projection in $E$. Let $N$ be the unital JB$^*$-subalgebra of $E$ generated by $u$ and $u^*$. By \cite[Theorems 2.4.13 and 2.4.14]{hanche1984jordan} (or \cite[Corollary 2.2]{Wright1977}), $N$ is a JC$^*$-algebra, hence we may assume it is a JB$^*$-subalgebra of a $C^*$-algebra $A$. Note that $p=u^*u\in A$ is the inital projection of $u$ and $q=uu^*\in A$ is the final projection. By the assumption we know that $\frac12(p+q)=u\circ u^*$ is a projection. Since projections are extreme points of the the positive portion of the unit ball of each C$^*$-algebra (by \cite[Theorem 4]{KadisonAnn51}), we deduce that $p=q$.
Hence $u\circ u^*=p=q$ is simultaneously the initial projection and the final projection of $u$ in $A$. Now it easily follows $A_2(u)=A_2(u^*\circ u)$, thus $u\sim_2 u^*\circ u$ in $A$ (hence in $N$ and thus in $E$).
\end{proof}

The following example provides an analysis of the definition of $\le_n$ and, moreover, illustrates non-transitivity of $\le_n$.

\begin{example}\label{ex:len not transitive} 
\begin{enumerate}[$(a)$]
    \item The assumption $\J uue\le e$ does not imply $u\in E_2(e)$.
    \item If $P_1(e)u=0$, then $\J uue=\J{P_2(e)u}{P_2(e)u}e$. In this case $\J uue\le e$ if and only if $P_2(e)u\le_n e$.
    \item There are tripotents $e,u,v\in M_2$ such that $e\sim_h u$, $v\le u$, but $v\not\le_n e$. In particular, the relation $\le_n$ is not transitive.
\end{enumerate}
\end{example}

\begin{proof}
$(a)$ Set $E=M_2$. Let 
$$e=\begin{pmatrix} 1&0\\0&0\end{pmatrix}, \quad u=\begin{pmatrix} 1&0\\0&1\end{pmatrix}, \quad
v=\begin{pmatrix} \frac1{\sqrt2}&\frac1{\sqrt2}\\-\frac1{\sqrt2}&\frac1{\sqrt2}\end{pmatrix}.$$
Then $e,u$ are projections, hence tripotents in $M_2$. Moreover, $v$ is a unitary matrix, hence it is a tripotent as well.
Clearly $u\notin E_2(e)$ and $v\notin E_2(e)$. Moreover,
$\J uue=e\le e$ and
$$\J vve = \frac12(vv^*e+ev^*v)=\frac12(e+e)=e\le e.$$

$(b)$ The Peirce arithmetic easily implies
$$\begin{aligned}
P_2(e)\J uue&=\J{P_2(e)u}{P_2(e)u}e+\J{P_1(e)u}{P_1(e)u}e,\\ P_1(e)\J uue&=\J{P_1(e)u}{P_2(e)u}e+\J{P_0(e)u}{P_1(e)u}e,\\
 P_{0} (e)\J uue&=0.
\end{aligned}
$$
Thus, if $P_1(e)u=0$, then $\J uue=\J{P_2(e)u}{P_2(e)u}e\in E_2(e)$. Moreover, in this case $P_2(e)u$ is a tripotent
as
$$\begin{aligned}P_2(e)u+P_0(e)u&=u=\J uuu \\&= \J {P_2(e)u+P_0(e)u}{P_2(e)u+P_0(e)u}{P_2(e)u+P_0(e)u}\\&= \J{P_2(e)u}{P_2(e)u}{P_2(e)u}+ \J{P_0(e)u}{P_0(e)u}{P_0(e)u},\end{aligned}$$
so the assertion easily follows from Proposition \ref{P:charact len}.

$(c)$ Let 
$$e=\begin{pmatrix} 1 & 0 \\ 0 & 1 \end{pmatrix},\quad u=\begin{pmatrix} 0 & -1 \\ -1 & 0 \end{pmatrix}, \quad
v=\begin{pmatrix} 0 & -1 \\ 0 & 0 \end{pmatrix}.$$
Then $e,u$ are unitary matrices, so tripotents. Moreover, in Example \ref{ex:leh not transitive} we proved that $u\sim_h e$.
$v$ is clearly a partial isometry. We claim that $v\le u$. Indeed,
$$\J vuv= vu^*v=\begin{pmatrix} 0 & -1 \\ 0 & 0 \end{pmatrix}\begin{pmatrix} 0 & -1 \\ -1 & 0 \end{pmatrix}\begin{pmatrix} 0 & -1 \\ 0 & 0 \end{pmatrix}=\begin{pmatrix} 1 & 0 \\ 0 & 0 \end{pmatrix}\begin{pmatrix} 0 & -1 \\ 0 & 0 \end{pmatrix}=v.$$
However, $v\not\le_n e$, as
$$\J vve= \frac12(vv^*e+ev^*v)=\begin{pmatrix} \frac12 & 0 \\ 0 & \frac12 \end{pmatrix},$$
which is not a tripotent.
\end{proof}

Since the relation $\le_n$ is not transitive, we define $\le_{n,t}$ to be its transitive hull. The symbols $\sim_{n,t}$ and $<_{n,t}$ have the obvious meaning.

Notice that $\sim_{n,t}$ coincides with $\sim_2$. Indeed, since $\le_2$ is finer than $\le_{n,t}$, obviously $\sim_{n,t}$ is finer than $\sim_2$. The converse inclusion follows from Proposition~\ref{p:simn=sim2}.

The following lemma provides a factorization of the relations $\le_n$ and $\le_{n,t}$.

\begin{lemma}\label{L:factor len} Let $E$ be a JB$^*$-triple and $u,e\in E$ two tripotents. Consider the following statements.
\begin{enumerate}[$(i)$]
    \item $u\le_n e$;
    \item there is a tripotent $v$ such that $u\sim_2 v$ and $v\le e$;
    \item there is a tripotent $w$ such that $e\sim_2 w$ and $u\le w$;
    \item $u\le_{n,t}e$.
\end{enumerate}
Then 
$$(i)\Leftrightarrow(ii)\Rightarrow(iii)\Leftrightarrow(iv).$$
The implication $(iii)\Rightarrow(ii)$ fails for example in $E=M_2$.
\end{lemma}

\begin{proof}
$(i)\Rightarrow(ii)$ This follows from Lemma~\ref{L:len implies sim2}, it is enough to take $v=\J uue$.

$(ii)\Rightarrow(i)$  Assume that $u\sim_2 v$ and $v\le e$. Then $u\in E_2(v)\subset E_2(e)$ and, moreover,
$$\J uue=\J uuv =v.$$
This gives the desired conclusion.

$(ii)\Rightarrow(iii)$ Since $v\le e$, we get that $e-v$ is a tripotent orthogonal to $v$. Since $u\sim_2 v$, the Peirce projections of $u$ and $v$ coincide. Thus, $u$ is orthogonal to $e-v$. It follows that $w=e-v+u$ is a tripotent satisfying $u\le w$. Moreover, $w\sim_2 e$ as
$$L(w,w)=L(e-v,e-v)+L(u,u)=L(e-v,e-v)+L(v,v)=L(e,e),$$ where we applied that $u\sim_2 v \Rightarrow L(v,v) = L(u,u)$ (cf. \cite[Proposition 2.3]{Finite} or \cite[Proposition 6.6]{hamhalter2019mwnc}).

A counterexample to $(iii)\Rightarrow(i)$ was given in Example~\ref{ex:len not transitive}$(c)$.

$(iii)\Rightarrow(iv)$ This is obvious.

$(iv)\Rightarrow(iii)$ Assume that $u\le_{n,t}e$. Then there are tripotents
$$e=v_0,v_1,\dots,v_k=u$$
such that
$v_j\le_n v_{j-1}$ for $j=1,\dots,k$. We will prove the statement by induction on $k$. The case $k=1$ follows from the implication $(i)\Rightarrow(iii)$. Assume that $k>1$ and the statement holds for $k-1$. Then there is a tripotent $w'\in E$ such that $w'\sim_2 e$ and $v_{k-1}\le w'$. Note that then $u=v_k\le_n w'$. Indeed, $u\in E_2(v_{k-1})\subset E_2(w')$ and $v_{k-1}=P_2(v_{k-1})(w')$. So, by Peirce arithmetic
$$\J uu{w'}=\J uu{v_{k-1}}$$
is a tripotent. Thus, using again the implication $(i)\Rightarrow(iii)$ we get a tripotent $w$ with $w\sim_2 w'$ and $u\le w$. Then $w\sim_2e$ and the proof is complete.
\end{proof}

\subsection{Overall comparison of the relations} 

In this subsection we compare the above-defined relations and collect examples distingushing them. We start by the following proposition collecting the implications among the relations.

\begin{prop}\label{P:order implications}
Let $E$ be a JB$^*$-triple and $e,u$ be two tripotents in $E$. Then
$$\begin{array}{cccccccccc}
u\le e & \Rightarrow & u\le_r e& \Rightarrow & u \le_c e   & & &  & \\
 && \Downarrow&& \Downarrow && &&    \\
     && u \le_h e &\Rightarrow&u\le_{hc}e &\Rightarrow & u\le_n e &&  \\
      && \Downarrow&& \Downarrow&& \Downarrow &&   \\
     & & u \le_{h,t} e & \Rightarrow & u \le_{hc,t} e & \Rightarrow &u\le_{n,t} e &\Rightarrow&u\le_2 e 
\end{array}
$$
Moreover, if $E$ is a JB$^*$-algebra and $e,u\in E$ are projections,
then all the considered relations are equivalent.
\end{prop}

\begin{proof}
The two implications in the first line follow immediately from definitions.

$u\le_r e\Rightarrow u\le_h e$: Compare Proposition~\ref{P:charact ler}$(a)$ and Proposition~\ref{P:charact leh}$(iii)$.

$u\le_c e\Rightarrow u\le_{hc} e$ Compare Proposition~\ref{P:charact lec}$(a)$ and Proposition~\ref{P:hc-charact}$(a)(iii)$.

$u\le_h e\Rightarrow u\le_{hc} e$: This is trivial.

$u\le_{hc} e\Rightarrow u\le_n e$: Assume $u\le_{hc}e$. By definition and Proposition~\ref{P:charact leh}$(iii)$ there are two mutually orthogonal projections $p,q\in E_2(e)$ and a complex unit $\alpha$ such that $u=\alpha(p-q)$. Using Proposition~\ref{P:charact len}$(iii)$ we deduce that $u\le_n e$.

The downward implications from the second line to the third line are obvious. 

The first two implications on the third line follow from the definitions using the implications on the second line.

$u\le_{n,t} e\Rightarrow u\le_{2} e$: Using Proposition~\ref{P:charact len} we see that $u\le_n e\Rightarrow u\le_2 e$. The transitivity of $\le_2$ completes the argument.

Now assume that $E$ is a JB$^*$-algebra and $e,u\in E$ are projections such that $u\le_2 e$. It means, that
$$u=\J eeu=(e\circ e^*)\circ u+e\circ(e^*\circ u)-(e\circ u)\circ e^*=e\circ u.$$
But this means exactly that $u\le e.$
\end{proof}

The next proposition collects the implication between the symmetric versions of the relations.

\begin{prop}\label{P:sim implications}
Let $E$ be a JB$^*$-triple and $e,u$ be two tripotents in $E$. Then
$$\begin{array}{ccccccccc}
u= e & \Rightarrow & u\sim_r e& \Rightarrow & u \sim_c e   &&& &  \\
 && \Downarrow&& \Downarrow && &&   \\
     && u \sim_h e &\Rightarrow&u\sim_{hc}e &\Rightarrow & u\sim_n e &&  \\
      && \Downarrow&& \Downarrow&& \Updownarrow &&  \\
     && u \sim_{h,t} e & \Rightarrow & u \sim_{hc,t} e & \Rightarrow &u\sim_{n,t} e &\Leftrightarrow&u\sim_2 e  
\end{array}
$$
Moreover, if $E$ is a JB$^*$-algebra and $e,u\in E$ are projections,
then all the considered relations are equivalent.
\end{prop}

\begin{proof}
This follows from Proposition~\ref{P:order implications} and Proposition~\ref{p:simn=sim2}.
\end{proof}

The next proposition collects implications among the strict versions of the relations.

\begin{prop}\label{P:strict order implications}
Let $E$ be a JB$^*$-triple and $e,u$ be two tripotents in $E$. Then
$$\begin{array}{ccccccccc}
u< e & \Rightarrow & u<_r e& \Rightarrow & u <_c e   &&&&   \\
     && \Downarrow&& \Downarrow &&&&   \\
     && u <_h e & \Rightarrow & u<_{hc}e &\Rightarrow& u<_n e &&\\
      && \Downarrow&& \Downarrow && \Downarrow &&  \\
     && u <_{h,t} e & \Rightarrow& u<_{hc,t}e &\Rightarrow & u<_{n,t} e &\Rightarrow&u<_2 e.  
\end{array}
$$
\end{prop}

\begin{proof}
This follows from Proposition~\ref{P:order implications} using the fact that for each of the relations $\le_\kappa$ we have that $u\le_\kappa e$ and $u\sim_2 e$ implies $u\sim_\kappa e$ (see the respective characterizations above).
\end{proof}

Next we collect examples showing that no more implications are valid in general.

\begin{example2}\label{ex:distinguishing relations} 
$(a)$ Assume that $E$ is a JB$^*$-triple and $e\in E$ is a nonzero tripotent.
Then
 \begin{itemize}
    \item $-e\sim_r e$, but $e$ and $-e$ are incomparable with respect to $\le$;
    \item $ie\sim_c e$, but $e$ and $ie$ are incomparable both with respect to $\le_r$ and with respect to $\le_h$.
\end{itemize}

Further, assume that $u,v\in E$ are two orthogonal nonzero tripotents. Then:
\begin{itemize}
  \item $-u<_r u+v$, but $-u$ and $u+v$ are incomparable with respect to $\le$;
    \item $iu<_c u+v$, bud $iu$ and $u+v$ are incomparable both with respect to $\le_r$ and with respect to $\le_h$.
    \end{itemize}
    
$(b)$ Assume $E=\ce$. Then $\le_h$ coincides with $\le_r$. Since $\le_r$ is transitive, it coincides with $\le_{h,t}$. Finally, $i\sim_c 1$, but $i$ and $1$ are incomparable with respect to $\le_r$, hence also with respect to $\le_{h,t}$.

$(c)$ Assume $E=\ce\oplus_\infty\ce$. Then $(u_1,u_2)\le_h(e_1,e_2)$ if and only if $u_1\le_r e_1$ and $u_2\le_r e_2$. Since $\le_r$ is transitive, it follows that $\le_h$ is transitive as well, hence it coincides with $\le_{h,t}$. Finally, $(i,0)<_c (1,1)$, but $(i,0)$ and $(1,1)$ are incomparable with respect to $\le_{h,t}$.

$(d)$ Assume that $E$ is a JB$^*$-triple and $u,v\in E$ are two orthogonal nonzero tripotents. Then:
\begin{itemize}
 \item   $u-v\sim_h u+v$, but the tripotents $u-v$ and $u+v$ are incomparable with respect to $\le_c$;
 \item $i(u-v)\sim_{hc} u+v$, but the tripotents $i(u-v)$ and $u+v$ are incomparable both with respect to $\le_h$ and with respect to $\le_c$. 
 \item $u+iv\sim_n u+v$, but the tripotents $u+iv$ and $u+v$ are incomparable with respect to $\le_{hc}$.
\end{itemize} 

Assume moreover that $u,v,w\in E$ are three mutually orthogonal nonzero tripotents. Then:

\begin{itemize}
 \item   $u-v<_h u+v+w$, but the tripotents $u-v$ and $u+v+w$ are incomparable with respect to $\le_c$;
 \item $i(u-v)<_{hc} u+v+w$, but the tripotents $i(u-v)$ and $u+v+w$ are incomparable both with respect to $\le_h$ and with respect to $\le_c$. 
 \item $u+iv<_n u+v+w$, but the tripotents $u+iv$ and $u+v+w$ are incomparable  with respect to $\le_{hc}$.
\end{itemize} 

$(e)$ Assume $E=M_2$. Let $e,v\in E$ be as in Example~\ref{ex:leh not transitive}. Then $e\sim_{h,t} v$, but $e$ and $v$ are incomparable with respect to $\le_h$.
 
Moreover, in $M_3$ we have
$$\begin{pmatrix} v & 0 \\ 0 & 0\end{pmatrix} <_{h,t}\begin{pmatrix} e & 0 \\ 0 & 1\end{pmatrix}$$
but these tripotents are incomparable with respect to $\le_h$.

$(f)$ Assume $E=M_2$. Let $e,v\in E$ be as in Example~\ref{ex:len not transitive}$(c)$. Then $v<_{h,t}e$, but $e$ and $v$ are incomparable with respect to $\le_n$.

$(g)$ Let $E=B(\ell_2)$ and $s$ be the forward shift. Then $s<_2 1$, but $s$ and $1$ are incomparable with respect to $\le_{n,t}$.

Indeed, the initial projection of $s$ is $1$, let $p$ denote the final projection. The formula for $p$ is 
$$p(\xi_1,\xi_2,\xi_3,\dots)=(0,\xi_2,\xi_3,\dots),\quad \xi=(\xi_1,\xi_2,\xi_3,\dots)\in\ell_2.$$
Thus clearly $s<_2 1$, hence we easily get $1\not\le_{n,t}s$. We will show that
$s\not\le_{n,t}1$. We will proceed by contradiction. Assume that
$s\le_{n,t}1$. By Lemma~\ref{L:factor len} we deduce that there is a tripotent $w$ such that $s\le w\sim_2 1$. I.e., $w$ is a unitary operator and, by Proposition~\ref{P:order char} we get
$$s=\J ssw=\frac12(ss^*w+ws^*s)=\frac12(pw+w)=\frac{1+p}{2}w.$$
Since $w$ is a unitary operator, we deduce that $\frac12(1+p)$ is a partial isometry, a contradiction.
\end{example2}

The following proposition shows that in case of commutative C$^*$-algebras some of the relations coincide. The proof is easy, so it is omitted.

\begin{prop}\label{P:orders special cases} Let $E$ be a JB$^*$-triple and $e,u$ be two tripotents in $E$.
\begin{enumerate}[$(a)$]
    \item If $E=\ce$, then
    $$\begin{array}{ccccccccc}
u\le e & \Rightarrow & u\le_r e& \Rightarrow & u \le_c e   &&& &  \\
     && \Updownarrow && \Updownarrow && && \\
     && u \le_h e & \Rightarrow &u\le_{hc} e&\Leftrightarrow& u\le_n e &&\\
      && \Updownarrow&& \Updownarrow && \Updownarrow &&   \\
     && u \le_{h,t} e & \Rightarrow&u\le_{hc,t} e&\Leftrightarrow & u\le_{n,t} e &\Leftrightarrow&u\le_2 e.  
\end{array}
$$
Moreover, if $u\ne 0$, then 
\begin{enumerate}[$(i)$]
    \item $u\le e \Leftrightarrow u=e$;
    \item $u\le_r e \Leftrightarrow u=\pm e$;
    \item $u\le_c e\Leftrightarrow u\sim_c e \Leftrightarrow e\ne0$.
\end{enumerate}

\item If $E$ is an abelian C$^*$-algebra, then
$$\begin{array}{ccccccccc}
u\le e & \Rightarrow & u\le_r e& \Rightarrow & u \le_c e   &&& &  \\
     && \Downarrow&& \Downarrow && &&  \\
     && u \le_h e & \Rightarrow & u\le_{hc}e&\Rightarrow& u\le_n e && \\
      && \Updownarrow&& \Updownarrow &&  \Updownarrow && \\
     && u \le_{h,t} e & \Rightarrow & u\le_{hc,t}e&\Rightarrow & u\le_{n,t} e &\Leftrightarrow&u\le_2 e.  
\end{array}
$$
Moreover, if $E=C_0(\Omega)$ for a locally compact Hausdorff space $\Omega$, then
\begin{enumerate}[$(i)$]
    \item $u\le e \Leftrightarrow u=e\cdot\chi_U$ for a clopen set $U\subset\Omega$;
    \item $u\le_r e \Leftrightarrow u=\pm e\cdot\chi_U$ for a clopen set $U\subset\Omega$;
    \item $u\le_c e\Leftrightarrow u=\alpha e\cdot \chi_U$ for a clopen set $U\subset\Omega$ and a complex unit $\alpha$;
    \item $u\le_h e\Leftrightarrow u= e\cdot(\chi_U-\chi_V)$ for a pair of disjoint clopen sets $U,V\subset\Omega$;
    \item $u\le_{hc} e\Leftrightarrow u= \alpha e\cdot(\chi_U-\chi_V)$ for a pair of disjoint clopen sets $U,V\subset\Omega$ and a complex unit $\alpha$;
    \item $u\le_n e\Leftrightarrow \{\omega\in\Omega\setsep u(\omega)=0\}\supset\{\omega\in\Omega\setsep e(\omega)=0\}.$
\end{enumerate}
Hence, no other implications hold in general.
\end{enumerate}
\end{prop}

\subsection{Relations in different triples}

In this subsection we collect results on behaviour of the above-defined relations with respect to subtriples and direct sums.

We start by the following easy observation.

\begin{remark}\label{rem:relations in subtriples} If $B$ is a JB$^*$-subtriple  of a JB$^*$-triple $E$ and $u,e\in B$ are two tripotents, then it easily follows from the definitions or from the respective characterization that
$uRe$ in $B$ if and only if $uRe$ in $E$ whenever 
 $R$ is one of the relations from the list
$$\le,\le_r,\le_c,\le_h,\le_{hc},\le_n,\le_2.$$

If $R$ is one of the relations $\le_{h,t}$, $\le_{hc,t}$, $\le_{n,t}$, then $uRe$ in $B$ implies $uRe$ in $E$. The converse fails in general as witnessed by the following example.
\end{remark}

\begin{example2}\label{ex:lent in subtriple} 
$(a)$ Let $E=M_2(B(\ell^2))$ and $B\subset E$ be defined by
$$B=\left\{ \begin{pmatrix}
T & 0 \\ 0 & \lambda I
\end{pmatrix}\setsep T\in B(\ell^2),\lambda\in\ce\right\}.$$
Then $E$ is a von Neumann algebra and $B$ is its von Neumann subalgebra. Their common unit is
$$1=\begin{pmatrix}
I&0\\0&I
\end{pmatrix}.$$
Let $S\in B(\ell^2)$ denote the forward shift and $P$ denote the projection on the one-dimensional subspace of $\ell^2$ generated by the first canonical vector. Observe that $S^*$ is the backward shift, $S^*S=I$ and $SS^*=I-P$.

Let
$$u=\begin{pmatrix}
S&0\\0&0
\end{pmatrix}\in B\subset E.$$
Then $u$ is a tripotent in $B$ (hence in $E$). We claim that
$u\le_{n,t} 1$ in $E$ but $u\not\le_{n,t}1$ in $B$.

To see the first statement observe that
$$v=\begin{pmatrix}
S&P\\0&S^*
\end{pmatrix}$$
is a unitary element in $E$ and $u\le v$. Indeed,
$$vv^*=\begin{pmatrix}
S&P\\0&S^*
\end{pmatrix}\begin{pmatrix}
S^*&0\\P&S
\end{pmatrix}=\begin{pmatrix}
SS^*+P&PS\\S^*P&S^*S
\end{pmatrix}=\begin{pmatrix}
I&0\\0&I
\end{pmatrix}=1$$
and
$$v^*v=\begin{pmatrix}
S^*&0\\P&S
\end{pmatrix}\begin{pmatrix}
S&P\\0&S^*
\end{pmatrix}=\begin{pmatrix}
S^*S&S^*P\\PS&P+SS^*
\end{pmatrix}=\begin{pmatrix}
I&0\\0&I
\end{pmatrix}=1.$$
Moreover,
$$\J uvu = uv^*u = \begin{pmatrix}
S&0\\0&0
\end{pmatrix}\begin{pmatrix}
S^*&0\\P&S
\end{pmatrix}\begin{pmatrix}
S&0\\0&0
\end{pmatrix}=\begin{pmatrix}
SS^*S & 0\\0&0
\end{pmatrix}=u.
$$
Thus $v\sim_2 1$ and $u\le v$, so $u\le_{n,t} 1$ in $E$ by Lemma~\ref{L:factor len}.

The second statement will be proved by contradiction. Assume $u\le_{n,t}1$ in $B$. By Lemma~\ref{L:factor len} we get a tripotent $w\in B$ with $w\sim_2 1$ and $u\le w$. Since $w\in B$, we have
$$w=\begin{pmatrix}
T & 0 \\0 & \lambda I 
\end{pmatrix}$$
for some $T\in B(\ell^2)$ and $\lambda\in\ce$. The assumption $w\sim_2 1$ means that $w$ is a unitary element of $B$. Hence $\abs{\lambda}=1$ and $T$ is a unitary element of $B(\ell^2)$. Finally, the relation $u\le w$ means that
$$u=\J uwu=uw^*u= \begin{pmatrix}
S&0\\0&0
\end{pmatrix}\begin{pmatrix}
T^*&0\\0&\overline{\lambda}I
\end{pmatrix}\begin{pmatrix}
S&0\\0&0
\end{pmatrix}=\begin{pmatrix}
ST^*S & 0\\0&0
\end{pmatrix},$$
hence
$S=ST^*S$. It follows that
$$I=S^*S=S^*ST^*S=T^*S.$$
Multiplying by $T$ from the left we deduce that $T=S$, thus $S$ is unitary, which is a contradiction.

$(b)$
Let $E=M_2$ and let $B\subset E$ be formed by diagonal matrices. 
Let 
$$e=\begin{pmatrix} 1 & 0 \\ 0 & 1 \end{pmatrix},\quad 
v=\begin{pmatrix} i & 0 \\ 0 & -i \end{pmatrix}.$$
By Example~\ref{ex:leh not transitive} we know that $e\sim_{h,t} v$ in $E$. However, $e$ and $v$ are incomparable for $\le_{h,t}$ in $B$ (by Proposition~\ref{P:orders special cases}$(b)$).

$(c)$ Let $E=M_3$ and let  $B\subset E$ be formed by diagonal matrices. Let

$$e=\begin{pmatrix} 1 & 0 & 0 \\ 0 & 1 & 0\\0&0&1 \end{pmatrix},\quad 
v=\begin{pmatrix} 1&0&0\\0& i & 0 \\0& 0 & -i \end{pmatrix}.$$
Similarly as in Example~\ref{ex:leh not transitive} we show that $e\sim_{h,t} v$ in $E$. However, $e$ and $v$ are incomparable for $\le_{hc,t}$ in $B$ (by Proposition~\ref{P:orders special cases}$(b)$).
\end{example2}

\begin{prop}\label{P:direct sum}
Let $(E_j)_{j\in J}$ be a family of JB$^*$-triples and let $E=\bigoplus_{j\in J}^\infty E_j$ be their $\ell_\infty$-sum.
Let $u=(u_j)_{j\in J}$ and $e=(e_j)_{j\in J}$ be two tripotents in $E$. Then the following assertions are valid.
\begin{enumerate}[$(a)$]
    \item If $R$ is any of the above-defined relations and $uRe$, then $u_jRe_j$ for each $j\in J$.
    \item If $R$ is any of the relations from the list
    $$\le,\le_h,\sim_h,\le_n,\le_{n,t},\le_2,\sim_2,$$
    then $uRe$ if and only if  $u_jRe_j$ for each $j\in J$.
    \item If $J$ is a finite set, then $u\le_{h,t} e$ if and only if  $u_j\le_{h,t}e_j$ for each $j\in J$.
\end{enumerate}
\end{prop}

\begin{proof}
$(a)$ This is obvious.

$(b)$ This is obvious except for the case of $\le_{n,t}$. This case follows from Lemma~\ref{L:factor len}.

$(c)$ This is easy.
\end{proof}

For relations $\le_r,\le_c,\le_{hc},\le_{hc,t}$ the equivalence from assertion $(b)$ fails even if $J$ is finite. It is witnessed by the following example which is an easy consequence of Proposition~\ref{P:orders special cases}.

\begin{example2}\label{ex:direct sum} Let $E=\ce\oplus_\infty \ce$. Then:
\begin{enumerate}[$(a)$]
    \item $-1\sim_r 1$ in $\ce$, but $(-1,1)$ and $(1,1)$ are incomparable with respect to $\le_c$ (and, a fortiori, with respect to $\le_r$).
    \item $i\sim_c 1$ in $\ce$, but $(i,1)$ and $(1,1)$ are incomparable with respect to $\le_{hc,t}$ (and, a fortiori, with respect to $\le_{hc}$ and $\le_c$).
\end{enumerate}
\end{example2}

Note that it is not clear whether assertion $(c)$ of Proposition~\ref{P:direct sum} holds without assumption of finiteness of $J$.

\begin{prop}\label{P:C0(T,E)}
Let $E$ be a JB$^*$-triple and let $T$ be a Hausdorff locally compact space. Consider the JB$^*$-triple $C_0(T,E)$ (with supremum norm and pointwise triple product). Let $u,e\in C_0(T,E)$ be two tripotents. Then:
\begin{enumerate}[$(a)$]
    \item If $R$ is any of the above-defined relations and $uRe$, then $u(t)Re(t)$ for each $t\in T$.
    \item If $R$ is any of the relations from the list
    $$\le,\le_h,\sim_h,\le_n,\le_2,\sim_2,$$
    then $uRe$ if and only if  $u(t)Re(t)$ for each $t\in T$.
\end{enumerate}
\end{prop}

\begin{proof} This is obvious.
\end{proof}

Note that Example~\ref{ex:direct sum} shows that assertion $(b)$ of Proposition~\ref{P:C0(T,E)} fails for relations $\le_r,\le_c,\le_{hc},\le_{hc,t}$ -- even if $T$ is a two-point set.

On the other hand, it seems not to be clear whether this assertion  holds for $\le_{h,t}$ and $\le_{n,t}$.

Some of the relations are transitive and some are not. However, we have the following partial transitivity.

\begin{prop}\label{p:partial transitivity}
Let $E$ be a JB$^*$-triple and $R$ be any of the relations
$$\le,\le_r,\le_c,\le_h,\le_{hc},\le_n,\le_2.$$
Assume $u,v,e\in E$ are three tripotents such that $v\le e$. 

Then 
$$uRv \Leftrightarrow uRe \mbox{ and }u\le_2v.$$

If $R$ is one of the relations
$$\le_{h,t},\le_{hc,t},\le_{n,t},$$
then
$$uRv \Rightarrow uRe \mbox{ and }u\le_2v,$$
but the converse implication fails in general (for example in $E=B(\ell_2)$).
\end{prop}

\begin{proof} Note that $uRv$ implies $u\le_2v$. So, we may assume without loss of generality that $u\le_2 v$, i.e., $u\in E_2(v)$.
 Further, the assumption $v\le e$ means that $E_2(v)$ is a JB$^*$-subalgebra of $E_2(e)$ (see Proposition~\ref{P:order char}$(viii)$).
 
 Now we distinguish the individual cases:
 
 We have
 $$u\le v\Leftrightarrow u\mbox{ is a projection in }E_2(v)\Leftrightarrow u\mbox{ is a projection in }E_2(e)\Leftrightarrow u\le e.$$
 
 The cases of relations $\le_r$ and $\le_c$ then easily follow.

Further,
$$u\le_h v\Leftrightarrow u\mbox{ is self-adjoint in }E_2(v)\Leftrightarrow u\mbox{ is self-adjoint in }E_2(e)\Leftrightarrow u\le_h e.$$
 
The case of $\le_{hc}$ then easily follows.

The case of $\le_n$ follows from the equalities
$$\J uuv=u\circ_v u^{*_v}=u\circ_e u^{*_e}=\J uue.$$

The case of $\le_2$ is trivial as both $u\le_2 v$ and $u\le_2e$ hold.
 
If $R$ is one of the relations $\le_{h,t}$, $\le_{hc,t}$, $\le_{n,t}$, then $R$ is transitive and $v\le e$ implies $vRe$. Hence the implication `$\Rightarrow$' holds in these cases.

It remains to show that the implication `$\Leftarrow$' fails for these three relations. To this end we use some of the examples above.

Assume $E=M_2(B(\ell_2))$ (which is $*$-isomorphic to $B(\ell_2)$). Let $I$ be the unit in $B(\ell_2)$ and $S\in B(\ell_2)$ be the forward shift. Set
$$u=\begin{pmatrix} S & 0\\ 0&0\end{pmatrix}, 
v=\begin{pmatrix} I & 0\\ 0&0\end{pmatrix},
e=\begin{pmatrix} I & 0\\ 0&I\end{pmatrix}.$$
Then clearly $v\le e$ and $u\le_2 v$. Moreover, in Example~\ref{ex:lent in subtriple}$(a)$ it is proved that $u\le_{n,t}e$.
However, $u\not\le_{n,t}v$ by Example~\ref{ex:distinguishing relations}$(g)$.
So, this example witnesses failure of `$\Leftarrow$' for $\le_{n,t}$. 

To conclude that `$\Leftarrow$' fails also for $\le_{h,t}$ and $\le_{hc,t}$ it is enough to check that we have in fact $u\le_{h,t}e$. This folows from the results of Section~\ref{sec:vN} below. Indeed, there is a unitary element $w\in E$ such that $u\le w$. (This follows from the proof of Example~\ref{ex:lent in subtriple}$(a)$  or, alternatively, from Lemma~\ref{L:factor len}.) Since $E$ is a properly infinite von Neumann algebra and $e$ is its unit, the combination of Proposition~\ref{P:products of symmetries}$(i)$ and Proposition~\ref{P:le1 C*}$(f)$ below shows that $w\le_{h,t}e$ and hence $u\le_{h,t}e$. 
\end{proof}

It should be noted that all relations studied in this section are clearly preserved by triple homomorphisms (in one direction).

\section{Relations in JBW$^*$-triples -- auxilliary results}\label{sec:JBW*}

In the sequel we investigate the above-defined relations in JBW$^*$-triples using  known representations of JBW$^*$-triples. 
This section has an auxilliary character -- we recall the representation theorem and collect some auxilliary tools to study JBW$^*$-triples and JBW$^*$-algebras.

We  start by recalling the notion of finiteness from \cite{Finite}. Let $M$ be a JBW$^*$-triple. A tripotent $e\in M$ is \emph{finite} if any complete tripotent in $M_2(e)$ is already unitary in $M_2(e)$.
The JBW$^*$-triple $M$ itself is finite if any tripotent in $M$ is finite.

The next proposition is a new characterization of finite JBW$^*$-triples in terms of the coincidence of two of the relations studied in this note. %as witnessed by the following easy proposition.

\begin{prop}\label{P:le2=lent in finite}
Let $M$ be a JBW$^*$-triple. Then $M$ is finite if and only if the relations $\le_{n,t}$  and $\le_2$ coincide in $M$.
\end{prop}

\begin{proof}
Assume $M$ is finite. Let $u,e\in M$ be two tripotents such that $u\le_2 e$, i.e., $u\in M_2(e)$. By \cite[Lemma 3.2(b)]{Finite} there is a tripotent $v\in M$ with $v\sim_2 e$ and $u\le v$. Hence, Lemma~\ref{L:factor len} (the equivalence $(iii)\Leftrightarrow(iv)$) shows that $u\le_{n,t}e$.

Conversely, assume that the relations $\le_{n,t}$  and $\le_2$ coincide in $M$. Let $e\in M$ be a tripotent and let $u\in M_2(e)$ be a tripotent complete in $M_2(e)$. Then $u\le_2 e$, hence by the assumption we deduce $u\le_{n,t}e$. By Lemma~\ref{L:factor len} (the equivalence $(iii)\Leftrightarrow(iv)$) there is a tripotent $v\in M$ with $u\le v\sim_2 e$. By completeness of $u$ we deduce that $u=v$, therefore $u\sim_2 e$. Thus $u$ is a unitary element of $M_2(e)$, which completes the proof of finiteness.
\end{proof}

\subsection{Representation of JBW$^*$-triples}

We continue by recalling the representation of JBW$^*$-triples.
By \cite{horn1987classification,horn1988classification} any JBW$^*$-triple $M$ may be represented in the form
\begin{equation}\label{eq:representation of JBW* triples}
 \left(\bigoplus_{j\in J}^{\ell_\infty}A_j\overline{\otimes}C_j\right)\oplus^{\ell_\infty}H(W,\alpha)\oplus^{\ell_\infty}pV,    
\end{equation}
where the $A_j$'s are abelian von Neumann algebras, the $C_j$'s are Cartan factors, $W$ and $V$ are continuous von Neumann algebras, $p\in V$ is a projection, $\alpha$ is a linear involution on $W$ commuting with $*$ and $H(W,\alpha)=\{x\in W\setsep\alpha(x)=x\}$. 

We will group and analyze the individual summands similarly as in \cite{Finite}. This will be done in the subsequent sections. Here we only consider commutative von Neumann algebras and the tensor product in a special case.

 It follows from \cite[Theorem 6.4.1]{DualC(K)} (see also \cite[Theorem III.1.18]{Tak}) that any abelian von Neumann algebra may be represented as the direct sum of spaces of the form $L^\infty(\mu)$, where
  \begin{equation}\label{eq:CK=Linfty}\begin{aligned}
  \mu&\mbox{ is a Radon probability normal measure supported by a }\\ &\mbox{  hyper-Stonean compact space $\Omega$ and $C(\Omega)=L^\infty(\mu)$.}
  \end{aligned}
 \end{equation}
 
 Recall that $\mu$ is normal if it is order-continuous as a functional on $C(\Omega)$ (see, e.g., \cite[Definition 4.7.1]{DualC(K)}) and that $\Omega$ is hyper-Stonean if it is Stonean (i.e., extremally disconnected) and the union of supports of normal measures is dense in $\Omega$ (cf. \cite[Definition 5.1.1]{DualC(K)}). In case of \eqref{eq:CK=Linfty} the situation is easier -- we assume that the support of $\mu$ is whole $\Omega$, hence $\Omega$ is automatically hyper-Stonean as soon as it is extremally disconnected.
 
  The equality $C(\Omega)=L^\infty(\mu)$ means that the canonical inclusion of $C(\Omega)$ into $L^\infty(\mu)$ is a surjective isometry. This equality follows from the previous assumptions by \cite[Corollary 4.7.6]{DualC(K)}, but we include it in \eqref{eq:CK=Linfty} as it is essentially all we really use below. 
Thus, it is enough to consider the case when the $A_j$'s are the individual summands of the form $L^\infty(\mu)$
where $\mu$ satisfies \eqref{eq:CK=Linfty}.
Even in this case the description of the tensor product is not so easy. But it is simpler in case the respective $C_j$ is reflexive or even finite-dimensional.

\begin{lemma}\label{L:LinftyC=CKC}
Let $A=L^\infty(\mu)$ where $\mu$ satisfies \eqref{eq:CK=Linfty} and let $C$ be a reflexive Cartan factor. Then 
\begin{enumerate}[$(i)$]
    \item $A\overline{\otimes}C$ is canonically isomorphic to $L^\infty(\mu,C)$.
    \item Consider the canonical inclusion of $C(\Omega,C)$ into $L^\infty(\mu,C)$. It is an isometric embedding whose range is the closure of the space of simple measurable functions.    \item If $\dim C<\infty$, then $L^\infty(\mu,C)=C(\Omega,C)$, i.e., the canonical inclusion of $C(\Omega,C)$ into $L^\infty(\mu,C)$ is a surjective isometry.
\end{enumerate}
\end{lemma}

\begin{proof}
Assertion $(i)$ is proved for example in \cite[Lemma 1.2]{Finite}. It is the only assertion where it is used that $C$ is a Cartan factor.

$(ii)$ Since $\mu$ is supported by $\Omega$, the canonical inclusion is an isometric embedding. Since the range of any $f\in C(\Omega,C)$ is compact, we easily deduce that it may be uniformly approximated by simple measurable functions. Conversely, it follows from \eqref{eq:CK=Linfty} that for any measurable set $B\subset \Omega$ there is a clopen set $G$ such that the symmetric difference has zero measure.
Thus any simple measurable function is almost everywhere equal to a continuous function, which proves the converse inclusion.

$(iii)$ This follows easily from $(ii)$ as in this case simple measurable functions are dense in $L^\infty(\mu,C)$.

\end{proof}

\subsection{Some specific tools for JBW$^*$-algebras}

JBW$^*$-algebras may be viewed as JBW$^*$-triples having a unitary element, in which one of the unitary elements is fixed and plays the role of a unit. Quite often, the choice of the unit is rather natural. In this subsection we collect some tools which may simplify description of the above-defined relations in JBW$^*$-algebra.

Along this subsection, assume that $M$ is a JBW$^*$-algebra and denote by $1$ its unit. The algebraic operations are connected with the triple product by the following known identities (cf. \eqref{eq:algebra from unitary} and \eqref{eq:triple product in JB*-algebra}):
$$\begin{gathered}
a\circ b=\J a1b,\quad a^*=\J 1a1\quad      \mbox{ for }a,b\in M,\\
\J abc=(a\circ b^*)\circ c + a\circ(b^*\circ c)- (a\circ c)\circ b^*\quad \mbox{ for }a,b,c\in M.
\end{gathered}$$

\begin{lemma}\label{L:sqrt of unitary}
Let $u\in M$ be a unitary element. Then there is a unitary element $v\in M$ such that $v^2=u$.
\end{lemma}

\begin{proof}
Let $N$ denote the  JB$^*$-subalgebra of $M$ generated by $u$. Then $N$ contains both $u^*$ and $1$. Moreover, $N$ is a JC$^*$-algebra, i.e., it is a JB$^*$-subalgebra of some C$^*$-algebra $A$ (see, e.g., \cite[Lemma 3.1]{hamhalter2019mwnc}). We may assume without loss of generality that $1$ (the unit of $M$) is also the unit of $A$. Then $u$ is unitary in $A$ as well (as $u\sim_2 1$). But this means that $u$ commutes with $u^*$ in $A$. We deduce that $N$ is associative, hence $\wscl{N}$ is associative as well. It follows that $\wscl{N}$ is a commutative von Neumann algebra, hence we may find a square root of $u$ (cf. \cite[Theorem 5.2.5]{KR1}).
\end{proof}

Our next lemma gathers some conclusions which can be deduced from \cite[Theorem 4.2.28~$(vii)$ and Theorem 4.1.3$(iv)$]{Cabrera-Rodriguez-vol1} and the definition of unitary element in a unital JB$^*$-algebra, here we present an alternative argument based on the triple structure.

\begin{lemma}\label{L:shift to 1}
Let $u\in M$ be a unitary element. Let $v\in M$ be a unitary element such that $v^2=u$. For $x\in M$ set
$$\Phi(x)=\J v{x^*}v.$$
Then the following assertions are true:
\begin{enumerate}[$(i)$]
    \item $\Phi$ is a triple automorphism of $M$ such that $\Phi(1)=u$.
    \item The inverse of $\Phi$ is given by
    $$\Phi^{-1}(x)=\J{v^*}{x^*}{v^*}, \quad x\in M$$
    and satisfies $\Phi^{-1}(u)=1$.
    \item Let $v\in M$ be any tripotent and let $R$ be any of the above-defined relations. Then
    $$ v\,R\,u \Leftrightarrow \Phi^{-1}(v)\, R\, 1.$$
\end{enumerate}
\end{lemma}

\begin{proof}
The mapping $\Phi$ is clearly a linear mapping of $M$ into $M$.  Moreover,
$$\J{v^*}{\J v{x^*}v^*}{v^*}=\J{v^*}{\J{v^*}x{v^*}}{v^*}
=P_2(v^*)(x)=x$$
for $x\in M$. The same works with the roles of $v$ and $v^*$ interchanged, hence $\Phi$ is a bijection and its inverse is given by assertion $(ii)$.

Since $\norm{v}=1$, we deduce that $\norm{\Phi(x)}\le \norm{x}$ and  $\norm{\Phi^{-1}(x)}\le \norm{x}$ for each $x\in M$ (cf. \eqref{eq triple product is contractive}). It follows that $\Phi$ is a surjective isometry, so it is a triple automorphism (see \cite[Proposition 5.5]{kaup1983riemann}).

Clearly $\Phi(1)=v^2=u$, hence $\Phi^{-1}(u)=1$. This proves assertions $(i)$ and $(ii)$.

Assertion $(iii)$ follows from $(i)$ and $(ii)$ as triple automorphisms preserve all of the defined relations.
\end{proof}

So, it follows from the previous two lemmata that in order to describe the relations $vRu$ in case $u$ is unitary it is enough to describe it when $u=1$. We will use this phenomenon below several times.
If $M$ is moreover finite, we may go even further as witnessed by the following lemma.

\begin{lemma}\label{L:shift to projection}
Let $M$ be a finite JBW$^*$-algebra and let $u\in M$ be any tripotent.
\begin{enumerate}[$(i)$]
    \item There is a unitary element $\widetilde{u}\in M$ such that $u\le\widetilde{u}$.
    \item Let $\Phi$ be a triple automorphism of $M$ provided by Lemmata ~\ref{L:sqrt of unitary} and \ref{L:shift to 1} (for the unitary  $\widetilde{u}$).          Then $\Phi^{-1}(u)$ is a projection.
    \item Let $v\in M$ be any tripotent and let $R$ be any of the above-defined relations. Then
    $$ v\,R\,u \Leftrightarrow \Phi^{-1}(v)\, R\, \Phi^{-1}(u).$$
\end{enumerate}
\end{lemma}

\begin{proof}
Assertion $(i)$ is proved in \cite[Lemma 3.2$(d)$]{Finite}. If $\Phi$ is a triple automorphism provided by Lemma~\ref{L:shift to 1}, then
$$\Phi^{-1}(u)\le\Phi^{-1}(\widetilde{u})=1,$$
hence $\Phi^{-1}(u)$ is a projection. Assertion $(iii)$ follows from the fact that a triple automorphism preserves all the above-defined relations.
\end{proof}

The key point in the above lemma is that to describe the relation $v\,R\,u$ in a finite JBW$^*$-algebra it is enough to understand it in case $u$ is a projection.

\subsection{Some tools for finite-dimensional Cartan factors}

In this subsection we collect several results which will later help to understand the relations in finite-dimensional Cartan factors and also for the respective tensor products used in the representation of JBW$^*$-triples. These tools are based mainly on high homogeneity of Cartan factors. At the moment we deal with Cartan factors in an abstract ways by referring to appropriate abstract results. Applications to concrete types of Cartan factors will be given in the subsequent sections.

So, let $C$ be a fixed finite-dimensional Cartan factor. 
It's rank, i.e., the maximal cardinality of an orthogonal family of nonzero tripotents, is necessarily finite. Let $m$ denote the rank of $C$.

The following result follows from \cite[\S 5]{loos1977bounded} (alternatively, from \cite{BunceChu92} or from \cite{kaup1997real}).

\begin{lemma} 
The following assertions are valid:
\begin{enumerate}[(a)]
    \item Let $u\in C$ be a nonzero tripotent. Then $u$ may be expressed as the sum of an orthogonal family of minimal tripotents. The cardinality of such a family is uniquely determined.
    \item Any maximal orthogonal family of minimal tripotents in $C$ has cardinality $m = $rank$(C)$.
\end{enumerate}
\end{lemma}

The unique cardinality from $(a)$ is called \emph{the rank of }$u$. Any of the families from $(b)$ is called a \emph{frame} in $C$ (see \cite[\S 5]{loos1977bounded}). (Frames and rank are considered also in infinite-dimensional Cartan factors, but then the definition of frame is more complicated. Contrary to what could be expected from the setting of C$^*$-algebras, where finite rank implies finite dimension, there exist infinite dimensional Cartan factors with finite rank.)

\begin{lemma}\label{L:frames}
\begin{enumerate}[$(i)$]
    \item Let $u_1,\dots,u_m$ and $v_1,\dots,v_m$ be two frames of $C$. Then there is a triple automorphism of $C$ which maps $u_j$ to $v_j$ for $j\in\{1,\dots,m\}$.
    \item Let $u,v\in C$ be two tripotents with same rank. Then there is a triple automorphism of $C$ mapping $u$ to $v$.
\end{enumerate}
\end{lemma}

\begin{proof}
Assertion $(i)$ is proved in \cite[Theorem 5.9 and Corollary 5.12]{loos1977bounded} (see also \cite[Proposition 5.8]{kaup1997real} for a proof in the infinite dimensional case). To prove $(ii)$ note that $u=u_1+\dots+u_k$ and $v=v_1+\dots+v_k$, where $k$ is the rank of these tripotents and $u_1,\dots,u_k$ and $v_1,\dots,v_k$ are two orthogonal families of minimal tripotents. These two families may be extended to frames and then $(i)$ may be applied.
\end{proof}

Let $\operatorname{Iso}$ denote the set of all triple automorphisms of $C$ considered as a subset of the space of linear operators on $C$.
 Further, let $\U$ denote the set of all nonzero tripotents in $C$ and $\U_j$ its subset formed by all  tripotents of rank $j$ (for $j\in\{1,\dots,m\}$)
 
 The next lemma is a parametrized version of the preceding one.
 
\begin{lemma}\label{L:Cartan parametrization}
\begin{enumerate}[$(i)$]
    \item The sets $\operatorname{Iso}$, $\U$, $\U_1$,\dots $\U_{m}$ are compact.
    \item Let $j\in\{1,\dots,m\}$ and $u\in \U_{j}$. Then there is a Borel mapping $\Psi_u:\U_{j}\to \operatorname{Iso}$ such that $$\Psi_u(v)(v)= u,\mbox{ for all } v\in \U_{j}.$$
\end{enumerate}
\end{lemma}

\begin{proof}
$(i)$
It is clear that $\U$ is a closed bounded set, hence it is compact as $C$ has finite dimension. Further, by Kaup's theorem  $\operatorname{Iso}$ is precisely the set of all surjective linear isometries on $C$. Linear isometries form a closed bounded set. Since $C$ has finite dimension, any linear isometry is necessarily surjective and $\operatorname{Iso}$ is compact.

Further, fix any $j\in\{1,\dots,m\}$ and $u\in \U_{j}$. Then the mapping
$$T\mapsto T(u)$$
is a continuous map from $\operatorname{Iso}$ to $\U_j$. Moreover, its range is the whole $\U_{j}$ by Lemma~\ref{L:frames}$(ii)$,
 hence we deduce that this set is compact.

$(ii)$ Fix $u\in \U_{j}$. As mentioned in the proof of assertion $(i)$, the mapping
$$T\mapsto T(u)$$
 is a continuous map of the compact metric space $\operatorname{Iso}$ onto the compact metric space $\U_{j}$. By a consequence of the Kuratowski--Ryll-Nardzewski selection theorem
 (see, e.g., \cite[Theorem on p. 403]{Ku-RN})
 there is a Borel selection $F$ of the inverse. It is enough to set
$$\Psi_u(v)=F(v)^{-1},\quad v\in \U_j.$$
\end{proof}

We continue by a lemma stating that tripotents may be diagonalized in a Borel measurable way. 

\begin{lemma}\label{L:cartan diagonalization} Let $u\in C$ be a nonzero tripotent. Fix a decomposition
$$u=u_1+\dots+u_k,$$
where $u_1,\dots,u_k$ are mutually orthogonal minimal tripotents. Set
$$\U_u=\{v\in\U\setsep v\sim_2 u\},$$
i.e., $\U_u$ is the set of all unitary elements of $C_2(u)$.
Then there is a Borel mapping $\Theta:\U_{u}\to \operatorname{Iso}$ such that for each $v\in\U_u$ we have:
\begin{enumerate}[$(i)$]
    \item $\Theta(v)(u)=u$;
    \item $\Theta(v)(v)$ is a linear combination of $u_1,\dots,u_k$.
\end{enumerate}
\end{lemma}

\begin{proof} Note that
$$\U_u=\{v\in\U\setsep \J vvu=u \ \&\ \J uuv=v\},$$
so $\U_u$ is compact. Further, let
$$G=\{(T,v)\in \operatorname{Iso}\times \U_u\setsep T(u)=u\ \&\ T(v)\in\span\{u_1,\dots,u_k\}\}.$$
Then $G$ is compact and, moreover, for each $v\in \U_u$ the set
$$G(v)=\{T\in\operatorname{Iso}\setsep (T,v)\in G\}$$
is nonempty. Namely, since $v$ is a unitary in the finite-dimensional JB$^*$-algebra $C_2(u)$, Proposition 2.2$(b)$ in \cite{cartan6} assures the existence of mutually orthogonal minimal projections $q_1, \ldots, q_n$ in $C_2(u)$ and $\alpha_1, \ldots, \alpha_n\in \mathbb{T}$ such that $v = \sum_i \alpha_i q_i$. Clearly, $u = \sum_i q_i.$  Lemma~\ref{L:frames} implies the existence of a triple automorphism $T$ mapping $q_j$ to $u_j$ for $j \in \{1,\ldots, n\}.$ Consequently, $T(u) = u$ and $T(v) = \sum_i \alpha_i u_i$, that is, $T\in G(v).$ Thus $G$ has a Borel measurable selection by a consequence to the Kuratowski--Ryll-Nardzewski theorem (see, e.g., \cite[Theorem on p. 403]{Ku-RN}).
\end{proof}

\begin{lemma}\label{L:Cartan vector parametrization}
Let $\mu$ be a probability measure satisfying \eqref{eq:CK=Linfty}. Fix tripotents $u_1,$ $\dots,$ $u_m$ $\in$ $C$ such that for each $j$ the rank of $u_j$ equals $j$ and $m=\operatorname{rank}(C)$.

Let $\e\in L^\infty(\mu,C)=C(\Omega,C)$ be a tripotent.

\begin{enumerate}[$(i)$]
\item For each $j\in\{0,\dots,m\}$ set
$$\Omega_j=\{\omega\in\Omega\setsep \mbox{the rank of $\e(\omega)$ is }j\}.$$
Then each $\Omega_j$ is a clopen subset of $\Omega$.
    \item For each $j\in\{1,\dots,m\}$ let $\Psi_j: \U_{j}\to \operatorname{Iso}$ be a mapping provided by Lemma~\ref{L:Cartan parametrization}$(ii)$ for $u=u_j$. Given $\x\in C(\Omega,C)$ define
    $$\Psi(\x)(\omega)=\begin{cases} \Psi_j(\e(\omega))(\x(\omega)),& \omega\in \Omega_j, j\in\{1,\dots,n\},\\ \x(\omega), & \omega\in \Omega_0.
    \end{cases}$$
    Then $\Psi(\x)$ is a bounded Borel measurable mapping on $\Omega$ with values in $C$, hence it is $\mu$-almost everywhere equal to a continuous mapping.
    \item $\Psi$ is a triple automorphism of $C(\Omega,C)$ such that
    $$\Psi(\e)(\omega)=\begin{cases} u_j & \omega\in \Omega_j, j\in\{1,\dots,m\},\\
    0 &\omega\in\Omega_0.
    \end{cases}$$
\end{enumerate}
\end{lemma}

\begin{proof}
Assertion $(i)$ follows from Lemma~\ref{L:Cartan parametrization}$(i)$ and the continuity of $\e$.

$(ii)$ Clearly $\Psi(\x)$ is a function defined on $\Omega$ with values in $C$. Moreover, for each $\omega\in\Omega$ we have $\norm{\Psi(\x)(\omega)}=\norm{\x(\omega)}$ (as $\Psi_j(\e(\omega))$ is a triple automorphism  and hence an isometry of $C$). We deduce that $\Psi(\x)$ is bounded.

We continue by proving Borel measurability of $\Psi(\x)$. It is enough to prove it for its restriction to any $\Omega_j$. $\Psi(x)$ is obviously continuous  on $\Omega_0$.
Fix $j\ge 1$. Since $\e$ is continuous and the mapping $\Psi_j$ is Borel measurable, we deduce that $\omega\mapsto \Psi_j(\e(\omega))$ is a Borel measurable mapping of $\Omega_j$ into $\operatorname{Iso}$ (it is just the composition $\Psi_j\circ\e$). 
Further, the mapping $(T,x)\mapsto T(x)$ is a continuous mapping of $\operatorname{Iso}\times C$ into $C$, hence $(T,\omega)\mapsto T(\x(\omega))$ is a continuous mapping of $\operatorname{Iso}\times \Omega$ into $C$.  It follows that the mappings
$$\omega\mapsto \Psi_j(\e(\omega))(\x(\omega))$$
is Borel measurable on $\Omega_j$.

It follows that $\Psi(\x)$ is indeed a bounded Borel measurable mapping.
By Lemma~\ref{L:LinftyC=CKC} we deduce that $\Psi(\x)$ is $\mu$-almost everywhere equal to a continuous mapping.

$(iii)$ It follows from $(ii)$ that $\Psi$ maps $C(\Omega_j,C)$ into $C(\Omega_j,C)$. By the very definition it then follows that $\Psi$ is a triple homomorphism. But clearly $\Psi$ is onto and its inverse is
 $$\Psi^{-1}(\x) (\omega)=\begin{cases} \Psi_j(\e(\omega))^{-1}(\x(\omega)),& \omega\in \Omega_j, j\in\{1,\dots,n\},\\ \x(\omega), & \omega\in \Omega_0.
    \end{cases}$$
(similarly as in $(ii)$ we see that this mapping also maps $C(\Omega,C)$ into $C(\Omega,C)$). So, $\Psi$ is a triple automorphism.

Moreover, by construction we see that the required equality holds almost everywhere, in particular on a dense set. Since it is an equality of two continuous functions, the equality holds everywhere.
\end{proof}

\begin{lemma}\label{L:Cartan vector diagonalization}
Let $\mu$ be a probability measure satisfying \eqref{eq:CK=Linfty}. Let  $0=e_0,e_1,\dots,e_n\in C$ be fixed tripotents such that no two distinct out of them are $\sim_2$-equivalent. Assume that for each $j\in\{1,\dots,n\}$ we have
$$e_j=e_j^1+\dots+e_j^{k_j},$$
where $e_j^1,\dots,e_j^{k_j}$ are mutually orthogonal minimal tripotents in $C$ (hence $k_j$ is the rank of $e_j$). Let $\uu\in L^\infty(\mu,C)=C(\Omega,C)$ such that
$$\forall\omega\in\Omega,\,\exists j\in\{0,\dots,n\}\colon \uu(\omega)\sim_2 e_j.$$
\begin{enumerate}[$(i)$]
\item For each $j\in\{0,\dots,n\}$ set
$$\Omega_j=\{\omega\in\Omega\setsep \uu(\omega)\sim_2 e_j\}.$$
Then each $\Omega_j$ is a clopen subset of $\Omega$.
    \item For each $j\in\{1,\dots,n\}$ let $\Theta_j$ be a mapping provided by Lemma~\ref{L:cartan diagonalization} for $u=e_j$ and the decomposition $e_j=e_j^1+\dots+e_j^{k_j}$. Given $\x\in C(\Omega,C)$ define
    $$\Theta(\x)(\omega)=\begin{cases} \Theta_j(\uu(\omega))(\x(\omega)),& \omega\in \Omega_j, j\in\{1,\dots,n\},\\ \x(\omega), & \omega\in \Omega_0.
    \end{cases}$$
    Then $\Theta(\x)$ is a bounded Borel measurable mapping on $\Omega$ with values in $C$, hence it is $\mu$-almost everywhere equal to a continuous mapping.
    \item $\Theta$ is a triple automorphism of $C(\Omega,C)$ such that
    \begin{enumerate}[$(a)$]
        \item $\Theta(f\cdot e_j)=f\cdot e_j$ whenever $j\in\{1,\dots,n\}$ and $f\in C(\Omega)$ is a function which is zero outside $\Omega_j$.
        \item $\Theta(\uu)(\omega)$ is a linear combination of $e_j^1,\dots,e_j^{k_j}$
 whenever $\omega\in\Omega_j$, $j\in\{1,\dots,n\}$.    \end{enumerate}
   \end{enumerate}
\end{lemma}

\begin{proof}
$(i)$ Note that $\{u\in C\setsep u=\J uuu \ \&\ u\sim_2 e_j\}$ is a closed set, hence due to continuity of $\uu$ we deduce that the sets $\Omega_j$ are closed. Since they are disjoint and cover $\Omega$, they are necessarily clopen.

Assertion $(ii)$ may be proved by copying the argument from Lemma~\ref{L:Cartan vector parametrization}$(ii)$.

Assertion $(iii)$ may be proved by a slight modification of the argument from Lemma~\ref{L:Cartan vector parametrization}$(iii)$, we just use properties provided by Lemma~\ref{L:cartan diagonalization}.
\end{proof}

\section{Relations in von Neumann algebras and their right ideals}\label{sec:vN}

In this section we investigate the relations in JBW$^*$-triples of the form $M=pV$ where $V$ is a von Neumann algebra and $p\in V$ is a projection. It covers not only the summand $pV$ from \eqref{eq:representation of JBW* triples} (which corresponds to the case of continuous $V$) but also the summands of the form $A\overline{\otimes}C$ where $C$ is a Cartan factor of type $1$ (this corresponds to the case of type $I$ von Neumann algebra $V$, see \cite[p. 43]{hamhalter2019mwnc} for an explanation).

So, let us fix a von Neumann algebra $V$ and a projection $p\in V$. Set $M=pV$. It covers also the case $p=1$, i.e., if $M$ itself is a von Neumann algebra.

\subsection{General description of the relations}\label{sec:vN general}

In this subsection we collect basic characterizations of the relations in the language of C$^*$-algebras. 
We start by the following easy observation.

\begin{obs}\label{obs:shift to pipf}
Let $u,e\in M$ be two tripotents. Let $R$ be any of the above-defined relations. Then
$$u R e  \hbox{ in $M$ } \Leftrightarrow e^*u R p_i(e) \hbox{ in $V$ } \Leftrightarrow ue^* R p_f(e) \hbox{ in $M$ or in $V$}.$$
\end{obs}

\begin{proof}
Since $uRe$ implies $u\in M_2(e)$, the validity of $uRe$ depends only on the JB$^*$-triple structure of $M_2(e)$. So, it is enough to observe that the mapping $x\mapsto e^*x$ is a triple-isomorphism of $M_2(e)$ onto $V_2(p_i(e))$ and  $x\mapsto xe^*$ is a  triple-isomorphism of $M_2(e)$ onto $M_2(p_f(e))=V_2(p_f(e))$.
\end{proof}

It follows that the key step to understand the relations in this kind of JBW$^*$-triples is to characterize the validity of $uR1$ in a unital C$^*$-algebra. It is the content of the following proposition.

\begin{prop}\label{P:le1 C*}
 Let $A$ be a unital C$^*$-algebra and let $u\in A$ be a tripotent. Then we have the following
\begin{enumerate}[$(a)$]
    \item $u\le 1$ if and only if $u$ is a projection;
    \item $u\le_r 1$ if and only if $u$ or $-u$ is a projection;
    \item $u\le_c 1$ if and only if $u=\alpha p$, where $p$ is a projection and $\alpha$ is a complex unit.
    \item $u\le_h 1$ if and only if $u$ is self-adjoint;
    \item $u\sim_h 1$ if and only if $u$ is a symmetry;
    \item $u\sim_{h,t}1$ if and only if $u$ is a finite product of symmetries;
    \item $u\le_{h,t}1$ if and only if there are symmetries   $v_1,v_2,\dots,v_m\in M$ and a projection $p\in M$ such that $u=pv_1v_2\dots v_m$;
    \item $u\le_{hc}1$ if and only if $u$ is a scalar mutliple of a self-adjoint operator;
    \item $u\sim_{hc}1$ if and only if $u$ is a scalar multiple of a symmetry;
    \item $u\sim_{hc,t}1$ if and only if there are symmetries  $v_1,v_2,\dots,v_m\in A$ and a complex unit $\alpha$ such that $u=\alpha v_1v_2\dots v_m$;
    \item $u\le_{hc,t}1$ if and only if there are symmetries  $v_1,v_2,\dots,v_m\in A$, a projection $p\in M$ and a complex unit $\alpha$ such that $u=\alpha pv_1v_2\dots v_m$;
    \item $u\le_n 1$ if and only if $u$ is normal (i.e., $u^*u=uu^*$);
    \item $u\le_{n,t}1$ if and only if $u=pv$ for a projection $p\in A$ and a unitary element $v\in A$.
    \item $u\sim_2 1$ if and only if $u$ is a unitary element (i.e., $u^*u=uu^*=1$).
\end{enumerate}
\end{prop}

\begin{proof}
Assertions $(a)-(d)$ follow easily from definitions.

$(n)$ By the very definition $u\sim_2 1$ if and only if $A_2(u)=A_2(1)=A$. This exactly means that $u$ is a unitary tripotent, which is known to be equivalent to $u^*u=uu^*=1$. 

$(e)$ $u\sim_h 1$ means that $u\sim_2 1$ and $u\le_h 1$ (see Proposition~\ref{P:charact simh}). By $(d)$ and $(n)$ this takes place if and only if $u$ is a self-adjoint unitary element. But this is exactly the definition of a symmetry ($u=u^*$, $u^2=1$).

$(f)$ This follows by induction from the following observation. If 
$v,w\in A$ are two unitary elements, then $v\sim_h w$ if and only if $v^*w$ is a symmetry. Indeed, since automatically $v\sim_2 w$, we deduce
$$v\sim_h w\Leftrightarrow v\le_h w\Leftrightarrow v=wv^*w\Leftrightarrow w^*v=v^*w\Leftrightarrow (v^*w)^*=v^*w.$$ 

$(g)$ This follows from Lemma~\ref{L:factor leht} using $(f)$ and \cite[Proposition 4.6]{Finite}. (The quoted proposition is formulated for von Neumann algebras, but the same proof works also for C$^*$-algebras.)

Assertions $(h)-(k)$ follow easily from Proposition~\ref{P:hc-charact} and Proposition~\ref{P:hct-charact} using $(d)-(g)$.

$(l)$ Assume that $u\le_n 1$. By the definition it means that $\J uu1=\frac12(uu^*+u^*u)$ is a tripotent. Since this element is positive, it is a projection. Hence, the computations from Lemma~\ref{L:len implies sim2} show that $uu^*=u^*u$.

Conversely, if $uu^*=u^*u$, then $\J uu1=uu^*=p_i(u)$, so it is a projection. Hence $u\le_n 1$.

$(m)$ This follows from Lemma~\ref{L:factor len} and \cite[Proposition 4.6]{Finite}.
\end{proof}

Combining Proposition~\ref{P:le1 C*} with Observation~\ref{obs:shift to pipf} we get the following proposition.

\begin{prop}\label{P:all orders vN}
 Let $e,u\in M= p V$ be two tripotents.
 Let $r=p_f(e)$ and $q=p_i(e)$.
 Then we have the following:
\begin{enumerate}[$(a)$]
    \item $u\le_h e$ $\Leftrightarrow$ $u=ev$ for a self-adjoint $v\in V_2(q)$ $\Leftrightarrow$ $u=ve$ for a self-adjoint $v\in M_2(r)$;
    \item $u\sim_h e$ $\Leftrightarrow$ $u=ev$ for a symmetry $v\in V_2(q)$ $\Leftrightarrow$ $u=ve$ for a symmetry $v\in M_2(r)$;
    \item $u\sim_{h,t} e$ $\Leftrightarrow$ $u=ev$ where $v$ is a finite product of symmetries in $ V_2(q)$ $\Leftrightarrow$ $u=ve$ where $v$ is a finite product of symmetries in $M_2(r)$;
    \item $u\le_{h,t}e$ if and only if there are symmetries   $v_1,v_2,\dots,v_m\in V_2(q)$ and a projection $q'\le q$ such that $u=eq'v_1v_2\dots v_m$;
    \item $u\le_{hc}e$ $\Leftrightarrow$ $u=\alpha ev$ for a self-adjoint $v\in V_2(q)$ and a complex unit $\alpha$ $\Leftrightarrow$ $u=\alpha ve$ for a self-adjoint $v\in M_2(r)$ and a complex unit $\alpha$;
    \item $u\sim_{hc}e$ $\Leftrightarrow$ $u=\alpha ev$ for a symmetry $v\in V_2(q)$ and a complex unit $\alpha$ $\Leftrightarrow$ $u=\alpha ve$ for a symmetry $v\in M_2(r)$ and a complex unit $\alpha$;
    \item $u\sim_{hc,t}e$ if and only if there are symmetries  $v_1,v_2,\dots,v_m\in V_2(q)$ and a complex unit $\alpha$ such that $u=\alpha e v_1v_2\dots v_m$;
    \item $u\le_{hc,t}e$ if and only if there are symmetries  $v_1,v_2,\dots,v_m\in V_2(q)$, a projection $q'\le q$ and a complex unit $\alpha$ such that $u=\alpha eq'v_1v_2\dots v_m$;
    \item $u\le_n e$ $\Leftrightarrow$ $u=ev$ for a normal operator $v\in V_2(q)$ $\Leftrightarrow$ $u=ve$ for a normal operator $v\in M_2(r)$;
    \item $u\le_{n,t}e$ if and only if $u=eq'v$ for a projection $q'\le q$ and a unitary operator $v\in V_2(q)$.
\end{enumerate}
\end{prop}

\begin{proof}
This follows from Proposition~\ref{P:le1 C*} using Observation~\ref{obs:shift to pipf}. Let us give the proof for the first equivalence in $(a)$. The remaining cases follow in the same way.

Observation~\ref{obs:shift to pipf} shows that $u\le_h e$ if and only if $e^*u\le_h q=p_i(e)$. Since $q$ is the unit of $V_2(q)$, Proposition~\ref{P:le1 C*}$(d)$ shows that $e^*u\le_h q$ if and only if it is self-adjoint (in $V_2(q)$). Since
$u=pu=ee^*u$, the equivalence follows.
\end{proof}

We continue by pointing out the role of finiteness.

\begin{prop}\label{P:le0=lent finite vN}
If $p$ is a finite projection, then the triple $M=pV$ is a finite JBW$^*$-triple 
and hence  the relations  $\le_2$ and $\le_{n,t}$ coincide in $M$.
\end{prop}

\begin{proof}
The finiteness of $M$ follows from \cite[Proposition 4.19]{Finite}. The coincidence of $\le_2$ and $\le_{n,t}$ follows from Proposition~\ref{P:le2=lent in finite}.
\end{proof}

\subsection{Products of symmetries and the length of the chains of $\sim_h$}

By Proposition~\ref{P:le1 C*}$(f)$ the relation $\sim_{h,t}$ is closely related to products of symmetries. In this subsection we investigate this feature in more detail, it turns out to be related to the types of von Neumann algebras. There are several known results  on expressing unitary elements using products of symmetries which we collect in the following proposition.

\begin{prop}\label{P:products of symmetries} Let $V$ be a von Neumann algebra.
\begin{enumerate}[$(i)$]
    \item  \cite[Corollary]{Fillmore} If $V$ is properly infinite any unitary element in $V$ is the product of at most 4 symmetries in $V$; 
    \item \cite[Proof of Th{\'e}or\`{e}me 1$(i)\Rightarrow (ii)$ Deuxi\`{e}me cas]{Broise67} If $V$  is of type II$_1$ any unitary in $V$ is the product of at most 16 symmetries in $V$; 
    \item \cite[Theorem 3]{radjavi} Assume $V=M_n$, the algebra of $n\times n$ matrices. Then any unitary matrix in $V$ with determinant $\pm1$ is the product of at most four symmetries. Hence, any unitary matrix is a scalar multiple of a product of at most four symmetries;
    \item Assume $V$ is of type $I$. Then any unitary element in $V$ is the product of at most four symmetries and a central unitary operator.
\end{enumerate}
\end{prop}

\begin{proof}
Assertions $(i)$ and $(iii)$ are proved in the quoted papers.

Assertion $(ii)$ is proved in \cite{Broise67} in case $V$ is a factor. We present a proof of the general case which uses the ideas of \cite{Broise67}, simultaneously it is similar to the cases $(i)$ and $(iii)$.

Let $V$ be a von Neumann algebra of type $II_1$. Denote by $T$ the standard center-valued trace (see \cite[Theorem 8.2.8]{KR2}). Note that by \cite[Theorem 8.4.3]{KR2} we have
\begin{equation}\label{eq:equivalence of projections}
\forall p,q\in V\mbox{ projections}\colon p\sim q \Leftrightarrow T(p)=T(q).    
\end{equation}
Here $\sim$ is the Murray-von Neumann equivalence, i.e., $p\sim q$ in $V$ if there is a partial isometry $u\in V$ with $p_i(u)=p$ and $p_f(u)=q$.
 
By \cite[Corollary 3.14]{kadison-diagonal} we get the following (by $Z(V)$ we denote the center of $V$):
\begin{equation}\label{eq:smaller projection}
  \begin{gathered}
   \mbox{For each maximal abelian von Neumann subalgebra } W\subset V,
   \\ \mbox{each projection } p\in W \hbox{ and each }  h\in Z(V) \hbox{ with } 0\le h\le T(p),\\ \hbox{ there exists a projection } q\in W \mbox{ satisfying  }q\le p\mbox{ and }  T(q)=h. 
   \end{gathered}
\end{equation} 

We now easily deduce that
\begin{equation}\label{eq:halving}
\begin{gathered}\hbox{ For each normal element } x\in V  \hbox{ and each projection } p\in V \hbox{ with } px=xp \\
\hbox{there exists a projection } q\in V \hbox{ satisfying } q\le p, qx=xq \hbox{ and } q\sim p-q. 
       \end{gathered}
\end{equation}

Indeed, let $W$ be a maximal abelian von Neumann subalgebra of $V$ containing $p$ and $x$. By \eqref{eq:smaller projection} we get a projection $q\in W$ such that $q\le p$ and $T(q)=\frac12T(p)$. Thus by \eqref{eq:equivalence of projections}
we deduce $q\sim p-q$.

Now we are ready to make the construction itself.  Let $u\in V$ be any unitary element.
By \eqref{eq:halving} there is projection $p_0\in V$ commuting with $u$ such that $p_0\sim 1-p_0$. This enables us to express $u$ as a diagonal matrix
$$u=\begin{pmatrix}
u_1 &  0\\ 0 & u_2
\end{pmatrix},$$
where $u_1=p_0u=p_0up_0$ and $u_2=(1-p_0)u=(1-p_0)u(1-p_0)$. Hence we may express $u$ as the product of two unitary operators of the form
$$u=\begin{pmatrix}
u_1 &  0\\ 0 & 1
\end{pmatrix}\cdot \begin{pmatrix}
1 &  0\\ 0 & u_2
\end{pmatrix}=(up_0+1-p_0)(p_0+(1-p_0)u).$$
To prove that $u$ is the product of 16 symmetries, it is enough to prove that each of the two factors is the product of 8 symmetries. By the symmetry of these two cases it is enough to prove the statement for the element 
$$up_0+1-p_0=\begin{pmatrix}
u_1 &  0\\ 0 & 1
\end{pmatrix}.$$
By  \eqref{eq:smaller projection} we get a sequence $(p_n)$ of mutually orthogonal projections in $V$ such that $1-p_0=\sum_{n=1}^\infty p_n$ and $T(p_n)=\frac{1}{2^{n+1}}$ for $n\in\en$. 

The rest of the proof consists in a certain inductive construction. We first present the key induction step, then we use it to construct building blocks of four unitary elements each of them is a product of two symmetries.

\smallskip

\noindent{\tt The key induction step:}

\smallskip

Assume that $n\in\en\cup\{0\}$ is fixed and $w\in V$ is a partial isometry such that $w^*w=ww^*=p_n$ (i.e., $p_i(w)=p_f(w)=p_n$).
We will construct three partial isometries $a(w),b(w),c(w)$ with some suitable properties.

Firstly, by \eqref{eq:halving} there are two mutually orthogonal projections $p_{n,1},p_{n,2}\in V$ commuting with $w$ such that $p_{n,1}+p_{n,2}=p_n$ and $p_{n,1}\sim p_{n,2}$. Moreover, by \eqref{eq:equivalence of projections} we have $p_{n,1}\sim p_{n+1}$.

Fix partial isometries $z_1,z_2\in V$ such that $p_i(z_1)=p_{n,1}$, $p_f(z_1)=p_i(z_2)=p_{n,2}$, $p_f(z_2)=p_{n+1}$.

We will work in the von Neumann algebra $(p_{n,1}+p_{n,2}+p_{n+1})V(p_{n,1}+p_{n,2}+p_{n+1})$ -- we may represent its elements by $3\times 3$ matrices. In particular, we have
$$w=\begin{pmatrix}
w_1 & 0 & 0\\ 0 & w_2 & 0 \\ 0 & 0 & 0
\end{pmatrix},$$
where $w_1=p_{n,1}w$ and $w_2=p_{n,2}w$.

The idea is to imitate the method use for complex matrices in \cite{radjavi}. Informally speaking, we put
$$a(w)=\begin{pmatrix}
w_1 & 0 & 0\\ 0 & w_1^* & 0 \\ 0 & 0 & 0
\end{pmatrix},
b(w)=\begin{pmatrix}
1 & 0 & 0\\ 0 & w_1w_2 & 0 \\ 0 & 0 & w_2^*w_1^*
\end{pmatrix},
c(w)=\begin{pmatrix}
0 & 0 & 0\\ 0 & 0 & 0 \\ 0 & 0 & w_1w_2
\end{pmatrix},$$
observe that $a(w)b(w)=w$ and express $a(w)$ and $b(w)$ as products of self-adjoint partial isometries by
$$a(w)=\begin{pmatrix}
0& w_1 & 0\\  w_1^* &0 & 0 \\ 0 & 0 & 0
\end{pmatrix}\begin{pmatrix}
0 & 1 & 0\\ 1 & 0 & 0 \\ 0 & 0 & 0
\end{pmatrix}\mbox{ and }b(w)=\begin{pmatrix}
1 & 0 & 0\\ 0 & 0& w_1w_2  \\ 0 &  w_2^*w_1^*&0
\end{pmatrix}\begin{pmatrix}
1 & 0 & 0\\ 0 & 0 & 1 \\ 0 & 1 & 0
\end{pmatrix}.$$
These formulas, even though intuitive, are not formally correct as they tacitly use the transition partial isometries $z_1,z_2$. 
In a formally correct way 
we set
$$a(w)=\begin{pmatrix}
w_1 & 0 & 0\\ 0 & z_1 w^* z_1^* & 0 \\ 0 & 0 & 0
\end{pmatrix},
b(w)=\begin{pmatrix}
p_{n,1} & 0 & 0\\ 0 & z_1 w z_1^* w & 0 \\ 0 & 0 & z_2 w^* z_1 w^* z_1^* z_2^*
\end{pmatrix},$$ and 
$$c(w)=\begin{pmatrix}
0 & 0 & 0\\ 0 & 0 & 0 \\ 0 & 0 & z_2 z_1 w z_1^* w z_2^*
\end{pmatrix}.$$
These formulae mean
$$\begin{aligned}
a(w)&=p_{n,1}w + z_1w^*z_1^*,\\
b(w)&=p_{n,1} +  z_1wz_1^*wp_{n,2} + z_2w^*z_1w^*p_{n,2}z_1^*z_2^*=p_{n,1} +  z_1wz_1^*w + z_2w^*z_1w^*z_1^*z_2^*,\\
c(w)&=z_2z_1p_{n,2}wz_1^*wz_2^* =z_2z_1wz_1^*wz_2^*.
\end{aligned}$$
Then we have
$$\begin{gathered}
a(w)^*a(w)=a(w)a(w)^*=p_n, b(w)^*b(w)=b(w)b(w)^*=p_n+p_{n+1},
a(w)b(w)=w, \\
c(w)^*c(w)=c(w)c(w)^*=b(w)c(w)=c(w)b(w)=p_{n+1}. 
\end{gathered}$$

Moreover, we have

$$\begin{aligned} b(w)&=\begin{pmatrix}
p_{n,1} & 0 & 0\\ 0 & 0& z_1 w z_1^* w z_2^*  \\ 0 &  z_2 w^* z_1 w^* z_1^* &0
\end{pmatrix}\begin{pmatrix}
p_{n,1} & 0 & 0\\ 0 & 0 & z_2^* \\ 0 & z_2 & 0
\end{pmatrix} \\
&=b(w)(p_{n,1}+z_2+z_2^*)\cdot (p_{n,1}+z_2+z_2^*),\end{aligned}
$$
and so $a(w)$ and $b(w)$ are expressed as products of two self-adjoint partial isometries.

\smallskip

\noindent{\tt Construction of the building blocks:}

Set
$$\begin{gathered}
a_0=a(u_1), b_0=b(u_1), c_0=c(u_1),\\
a_n=a(c_{n-1}), b_n= b(c_{n-1}), c_n=c(c_{n-1}) \mbox{ for }n\in\en.
\end{gathered}$$

By an easy induction we get
\begin{equation}
             a_n^*a_n=a_n a_n^*=p_n, b_n^*b_n=b_n b_n^*=p_n+p_{n+1},  c_n^*c_n=c_n c_n^*=p_{n+1}
    \end{equation}
for $n\in\en\cup\{0\}$,
\begin{equation}
    a_0b_0=u_1 \mbox{ and }a_nb_n=c_{n-1}\mbox{ for }n\in\en,
\end{equation}
and
\begin{equation}
    b_nc_n=c_nb_n=p_{n+1}\mbox{ for }n\in\en\cup\{0\}.
\end{equation}

\smallskip

\noindent{\tt Adding the blocks and the final argument:}

\smallskip

We  set

$$\begin{gathered}
v_1=\sum_{n=0}^\infty a_{2n}+\sum_{n=0}^\infty p_{2n+1},\quad
 v_2=\sum_{n=0}^\infty b_{2n},
 \\v_3=\sum_{n=0}^\infty p_{2n}+\sum_{n=0}^\infty a_{2n+1},\quad
 v_4=p_0+\sum_{n=0}^\infty b_{2n+1}.
\end{gathered}$$

Since we add mutually orthogonal normal partial isometries, the sums are well defined. 

Moreover, $v_1,v_2,v_3,v_4$ are clearly unitary elements.

We claim that $p_0u+1-p_0=v_1v_2v_3v_4$. Indeed,
$$\begin{aligned}
p_0v_1v_2v_3v_4&= p_0 a_0 b_0 v_3v_4= p_0 u_1 v_3v_4 = p_0 u_1 (p_0 + a_1) v_4 \\ &= p_0 u_1 p_0 v_4 + p_0 u_1 a_1 v_4 = p_0 u_1  + p_0 u_1 p_0 a_1 v_4 =p_0 u_1=u_1
\end{aligned}$$
by the first step of the construction. For $n\ge 0$ we have
$$p_{2n+1} v_1 v_2 v_3 v_4 =b_{2n} a_{2n+1} b_{2n+1} =b_{2n} c_{2n} = p_{2n+1}$$
and 
$$p_{2n+2} v_1 v_2 v_3 v_4 =a_{2n+2} b_{2n+2} b_{2n+1}  =c_{2n+1} b_{2n+1} = p_{2n+2}.$$

Moreover, since each of the elements $a_n$ and $b_n$ is the product of two self-adjoint partial isometries --which can be summed thanks to the orthogonality of the corresponding summands --, we deduce that each of the four elements $v_1,\dots,v_4$ is the product of two symmetries. Hence, $p_0u+1-p_0$ is the product of 8 symmetries and the proof is completed.

Assertion $(iv)$ is a more precise formulation of the final remark in \cite{Fillmore}. Let us explain it. First, $V$ is either finite or it can be expressed as a direct sum of a finite von Neumann algebra and a properly infinite one (see \cite[Proposition 6.3.7]{KR2}). For the properly infinite summand we may use assertion $(i)$. So, assume that $V$ is finite. Then $V$ is a direct sum of von Neumann algebras of the form $L^\infty(\mu,M_n)=C(\Omega,M_n)$, where $\mu$ is a probability measure satisfying \eqref{eq:CK=Linfty} and $n\in\en$ (use \cite[Theorem V.1.27]{Tak} and Lemma~\ref{L:LinftyC=CKC}). The center of $C(\Omega,M_n)$ equals 
$$\{\f\in C(\Omega,M_n)\setsep \f(t) \mbox{ is a scalar multiple of the unit matrix for each }t\in\Omega\}.$$
Fix a unitary $\f\in C(\Omega, M_n)$. Then $\abs{\det\f(t)}=1$ for $t\in\Omega$.
There is a Borel measurable function $g_0:\Omega\to\TT$ such that $g_0(t)^n=\det\f(t)$ for $t\in\Omega$. By \eqref{eq:CK=Linfty} there is $g\in C(\Omega)$ which equals to $g_0$ $\mu$-a.e. Then $g(t)^n=\det\f(t)$ for $\mu$-a.a. $t\in\Omega$. Since both sides are continuous, the equality holds for each $t\in\Omega$.

Set $\h(t)=\overline{g(t)}\f(t)$. Then $\det\h(t)=1$ for $t\in\Omega$. Moreover, $\f$ can be expressed as $\f=\g\h$, where $\g$ belongs to the center ($\g(t)=g(t)\cdot\boldsymbol1$, where $\boldsymbol1$ is the unit matrix).

Since the C$^*$-subalgebra  of $C(\Omega,M_n)$ generated by $\h$ is abelian, \cite[Theorem 1]{pearcy1963} shows that there is a unitary element $\uu\in C(\Omega,M_n)$ such that $\uu\h\uu^*$ is diagonal. Repeating the proof of \cite[Theorem 3]{radjavi}
(with functions in place of scalars on the diagonal), we deduce that $\uu\h\uu^*$ is a product of at most four symmetries. Hence so is $\h$ and the proof is completed.
\end{proof}

It should be commented that a von Neumann algebra $V$ satisfies the so-called \emph{unitary factorization property} (i.e. each unitary in $V$ is a finite product of symmetries in $V$) if and only if the type I finite part of $V$ vanishes (cf. \cite[Proposition]{Bickchentaev2004}).

The previous proposition has some consequences for the order type relations in type $1$ Cartan factors.

\begin{prop}\label{P:coincidence in B(H,K)}
Let $H,K$ be Hilbert spaces. Then the following statements hold:
\begin{enumerate}[$(a)$]
    \item The relations $\le_{n,t}$ and $\le_{hc,t}$ coincide in $B(H,K)$. In particular, the relations $\sim_2$ and $\sim_{hc,t}$ coincide.
    \item If $H$ (or $K$) is finite-dimensional, then the relations
    $\le_2$, $\le_{n,t}$ and $\le_{hc,t}$ coincide in $B(H,K)$.
    \item If $u,v\in B(H,K)$ are two tripotents with $u\sim_{h,t}v$, 
    then there are tripotents $v_1,v_2,v_3\in B(H,K)$ such that
    $$u\sim_h v_3\sim_h v_2\sim_h v_1\sim_h v.$$
\end{enumerate}
\end{prop}

\begin{proof} Set $M=B(H,K) = p V$ with $V=B(H)$, where we can assume that $p$ is the orthogonal projection of $H$ onto $K$. 

$(a)$ Let us start by the `in particular' case.
 Assume $u\sim_2 v$. Then $u\in M_2(v)$, hence by Observation~\ref{obs:shift to pipf} we get $v^*u\sim_2 p_i(v)$,
 i.e., $v^*u$ is a unitary element of the von Neumann algebra $V_2(p_i(v))$. Then there are symmetries $w_1,\dots,w_4\in V_2(p_i(v))$ and a complex unit $\alpha$ with $v^*u=\alpha w_1w_2w_3w_4$. (Indeed, if $p_i(v)$ has finite rank, we use Proposition~\ref{P:products of symmetries}$(iii)$; if $p_i(v)$ has infinite rank, we use Proposition~\ref{P:products of symmetries}$(i)$, in this case $\alpha=1$.) It follows that
 $$1\sim_h w_4\sim_h w_3w_4\sim_h w_2w_3w_4\sim_{hc}
\alpha w_1w_2 w_3w_4=v^*u.$$
Another use of Observation~\ref{obs:shift to pipf} yields
$$v\sim_h vw_4\sim_h vw_3w_4\sim_h vw_2w_3w_4\sim_{hc}
\alpha vw_1w_2 w_3w_4=u,$$
hence $u\sim_{hc,t}v$.

Now the coincidence of $\le_{n,t}$ and $\le_{hc,t}$ follows from Lemma~\ref{L:factor len}, Lemma~\ref{L:factor leht} and Proposition~\ref{P:hc-charact}.
 
$(b)$ This follows from $(a)$ and Proposition~\ref{P:le0=lent finite vN}.

$(c)$ Assume $u\sim_{h,t}v$. Then also $u\sim_2v$, so we can proceed similarly as in $(a)$ to get $w_1,\dots,w_4$ and $\alpha$. The only difference is that in this case we can achieve $\alpha=1$. If $p_i(v)$ has infinite rank, we use Proposition~\ref{P:products of symmetries}$(i)$. If $p_i(v)$ has finite rank, then using Proposition~\ref{P:all orders vN}$(c)$ we see that $\det(v^*u)=\pm1$ (if we consider $v^*u\in V_2(p_i(v))\cong M_n$, where $n$ is the rank of $p_i(v)$), so we can use Proposition~\ref{P:products of symmetries}$(iii)$.
 \end{proof}

\begin{prop}\label{P:coincidence properly infinite} Assume that $M=pV,$ where $V$ is a von Neumann algebra and $e,u\in M$ are tripotents such that the projections $q=p_i(e)$ and $r=p_f(e)$ are properly infinite. Then the following statements hold.
\begin{enumerate}[$(i)$]
    \item If $u\sim_2 e$, then there are tripotents $v_1,v_2,v_3\in M$ such that
    $$u\sim_h v_3\sim_h v_2\sim_h v_1\sim_h e;$$
    \item $u\le_{n,t} e \Leftrightarrow u\le_{hc,t} e\Leftrightarrow u\le_{h,t} e$.
\end{enumerate}
\end{prop}

\begin{proof} 
$(i)$ Assume $u\sim_2v$. Then $uv^*$ is a unitary element of $M_2(p_f(v))$, so it is a product of four symmetries in $M_2(p_f(v))$ (by Propostion~\ref{P:products of symmetries}$(i)$). We conclude similarly as in Proposition~\ref{P:coincidence in B(H,K)}$(i)$.

$(ii)$ This follows from $(i)$, Lemma~\ref{L:factor len} and Lemma~\ref{L:factor leht}.
\end{proof}

\begin{prop}\label{p coincidence pVp type I infinitite}
Let $V$ be a von Neumann algebra and let $p\in V$ be a projection. Assume that $pVp$ is a type I von Neumann algebra.

If $u,v\in M=pV$ are two tripotents with $u\sim_{h,t}v$, 
    then there are tripotents $v_1,v_2,v_3\in M$ such that
    $$u\sim_h v_3\sim_h v_2\sim_h v_1\sim_h v.$$
\end{prop}

\begin{proof}
 If $p_i(v)$ is properly infinite, the assertion follows from Proposition~\ref{P:coincidence properly infinite}. So, it is enough to assume that $p_i(v)$ is finite. Further, using Observation~\ref{obs:shift to pipf} we may restrict to the case when $M=V$ is a finite von Neumann algebra of type I and $v=1$.
 
Such a von Neumann algebra is a direct sum of von Neumann algebras of the form $L^\infty(\mu,M_n)=C(\Omega,M_n)$ where $\mu$ is a probability measure satisfying \eqref{eq:CK=Linfty} and $n\in\en$. Hence, it is enough to prove the result for the individual summands.

So, assume that $u\sim_{h,t}v=1$. It follows that $u(t)\sim_{h,t}1$ for each  $t\in\Omega$, hence $\det u(t)=\pm1$ for $t\in\Omega$. If we now apply the construction from the proof of Proposition~\ref{P:products of symmetries}$(iii)$ to $u$ in place of $\f$, we get that $\g(t)=\pm\boldsymbol1$ for $t\in\Omega$. It follows that the product of $\g$ with a symmetry is again a symmetry. Thus $u$ is a product of four symmetries which completes the proof.
 \end{proof}

When in the proof leading to Proposition \ref{P:coincidence in B(H,K)} we replace Proposition \ref{P:products of symmetries}$(i)$ and $(iii)$ with Proposition \ref{P:products of symmetries}$(ii)$ we get the following conclusion. 

\begin{prop}\label{p pVp type II fin}
Let $V$ be a von Neumann algebra and let $p\in V$ be a projection. Assume that $pVp$ is a type $II_1$ von Neumann algebra. Then the following statements hold:\begin{enumerate}[$(i)$]\item If $u,v\in M=pV$ are two tripotents with $u\sim_{2}v$, then there are tripotents $v_1,v_2,\ldots, v_{15}\in M$ such that $$u\sim_h v_{15}\sim_h \ldots \sim_h v_2\sim_h v_1\sim_h v;$$
\item The relations $\le_2$, $\le_{n,t}$, $\le_{h,t}$ and $\le_{hc,t}$ coincide in $p V$. In particular, the relations $\sim_2$ and $\sim_{h,t}$ coincide.
\end{enumerate}
\end{prop}

Let us summarize the results of this subsection:

\begin{cor}\label{cor:pV} Let $M=pV$, where $V$ is a von Neumann algebra and $p\in V$ is a projection. Then the following assertions hold:
\begin{enumerate}[$(1)$]
    \item To describe $\sim_{h,t}$ chains of $\sim_h$ of length $16$ are enough. To describe $\le_{h,t}$ chains of $\le_h$ of length $17$ are enough. 
    \item Assume $V$ contains no direct summand of type $II$. Then to describe $\sim_{h,t}$ chains of $\sim_h$ of length $4$ are enough. To describe $\le_{h,t}$ chains of $\le_h$ of length $5$ are enough.
    \item Assume $V$ is continuous.
    Then the relations $\sim_2$ and $\sim_{h,t}$ coincide in $M$. Hence, the relations $\le_{n,t}$ and $\le_{h,t}$ coincide in $M$. If $V$ is moreover finite (i.e., of type $II_1$), the relations $\le_2$ and $\le_{h,t}$ coincide.  
    \item Assume $M=B(H,K)$ (i.e., $M$ is a Cartan factor of type $1$). Then the relations $\sim_2$ and $\sim_{hc,t}$ coincide in $M$. Hence, the relations $\le_{n,t}$ and $\le_{hc,t}$ coincide in $M$. If $\dim H<\infty$ or $\dim K<\infty$, even the relations $\le_2$ and $\le_{hc,t}$ coincide. 
\end{enumerate}
\end{cor}

\section{Symmetric and antisymmetric parts of von Neumann algebras}

In this section we address triples of the form $A\overline{\otimes}C$ where $A$ is an abelian von Neumann algebra and $C$ is a Cartan factor of type $2$ or $3$.
These spaces are thoroughly studied in \cite[Sections 5.3--5.5]{Finite}.

\subsection{Basic setting and notation}
We will assume that $A=L^\infty(\mu)$ for a probability measure $\mu$ (satisfying \eqref{eq:CK=Linfty}). Further, let $H=\ell_2(\Gamma)$ be a Hilbert space with a fixed orthonormal basis.
Then $A\overline{\otimes}B(H)$, the von Neumann tensor product, can be represented as a von Neumann sub-algebra in $B(L^2(\mu,H))$, for a description see \cite[Lemma 5.12]{Finite}.

Furhter, for any $\xi\in H$ we denote by $\overline{\xi}$ its canonical coordinatewise conjugation. If $\f\in L^2(\mu,H)$, we denote by $\overline{\f}$ the canonical pointwise conjugation.

For $x\in B(H)$ we define the transpose by
$$x^t(\xi)=\overline{x^*\overline{\xi}},\quad \xi \in H.$$
The representing matrix of $x^t$ with respect to the canonical orthonormal basis is the transpose of the representing matrix of $x$.

Similarly we may define for $T\in B(L^2(\mu,H))$ its transpose by
$$T^t\f=\overline{T^*\overline{\f}},\quad \f\in L^2(\mu,H).$$

Then 
$$B(H)_s=\{x\in B(H)\setsep x^t=x\}\mbox{ and }B(H)_a=\{x\in B(H)\setsep x^t=-x\}$$
are Cartan factors of types $3$ and $2$, respectively. They are formed by operators with symmetric (for type $3$) or antisymmetric (for type $2$) representing matrix with respect to the canonical orthonormal basis.

Moreover, the triples we address are
$$\begin{aligned}
A\overline{\otimes}B(H)_s&=(A\overline{\otimes}B(H))_s=\{T\in A\overline{\otimes}B(H)\setsep T^t=T\},\\
A\overline{\otimes}B(H)_a&=(A\overline{\otimes}B(H))_a=\{T\in A\overline{\otimes}B(H)\setsep T^t=-T\}.\end{aligned}$$

Observe that both $A\overline{\otimes}B(H)_s$ and $A\overline{\otimes}B(H)_a$ are a weak$^*$-closed subtriples of the von Neumann algebra $A\overline{\otimes}B(H)$, thus to describe relations
$$\le,\le_2,\le_n,\le_{hc},\le_h,\sim_2,\sim_{hc},\sim_h$$
we may use their description in the surrounding von Neumann algebra provided by Proposition~\ref{P:all orders vN}. A small drawback is that these characterizations are not completely internal. Anyway, we will not repeat them, we will only point out their internal forms if available. We will mostly focus on relations 
$$\le_{h,t},\le_{hc,t},\sim_{h,t},\sim_{hc,t}.$$

We may further introduce a canonical conjugation on $B(H)$ and on $B(L^2(\mu,H))$ by
$$\overline{x}(\xi)=\overline{x(\overline{\xi})},\quad \xi\in H, x\in B(H)$$
and, similarly,
$$\overline{T}\f=\overline{T\overline{\f}},\quad \f\in L^2(\mu,H), T\in B(L^2(\mu,H)).$$
The representing matrix of $\overline{x}$ with respect to the canonical orthonormal basis is the complex conjugate of the representing matrix of $x$ (entry by entry).

In the rest of this section let $M$ stand for the von Neumann algebra $A\overline{\otimes}B(H)$, $M_s$ for the triple $A\overline{\otimes}B(H)_s$ and $M_a$  for the triple $A\overline{\otimes}B(H)_a$.

\subsection{The symmetric case}

$A\overline{\otimes}B(H)_s$ is not only a subtriple, but even a weak$^*$-closed Jordan $*$-subalgebra of $A\overline{\otimes}B(H)$ containing the unit. Therefore also a large part of Proposition~\ref{P:le1 C*} may be applied.

We start by the following remark.

\begin{remark}\label{rem:AotimesB(H)s is finite}
By \cite[Proposition 5.20]{Finite} the triple $A\overline{\otimes}B(H)_s$ is a finite JBW$^*$-algebra. Hence, by Proposition~\ref{P:le2=lent in finite} the relations $\le_2$ and $\le_{n,t}$ coincide.
\end{remark}

Since $M_s$ is a finite JBW$^*$-algebra, by Lemmata~\ref{L:shift to 1} and~\ref{L:shift to projection} it is enough to analyze the relations $U\,R\,V$ ($U,V\in M_s$) only in case $V=1$ and, more generally, in case $V$ is a projection. To analyze the case $V=1$ we may use Proposition~\ref{P:le1 C*} and its small modification. We will do it in one lemma which translates some notions used in Proposition~\ref{P:le1 C*} to our case and in one proposition describing the relations defined by transitive hulls.

\begin{lemma}\label{L:AotimesB(H)s-translations} Let $U\in M_s$.
\begin{enumerate}[(i)]
    \item $U$ is self-adjoint if and only if $U=\overline{U}$.  In case $M=B(H)$ (i.e., $A=\ce$) this means that the representing matrix of $U$ is symmetric and has real entries.
    \item $U$ is a projection if and only if $\overline{U}=U=U^2$.
    \item $U$ is a symmetry if and only if $U=\overline{U}$ and $U^2=1$.
    \item $U$ is a normal element in  $M$ if and only if $\overline{U}U=U\overline{U}$ in $A\overline{\otimes}B(H)$. In case $M=B(H)$ this means that the representing matrix of $\overline{U}U$  has real entries.
    \item $U$ is a unitary element of $M_s$ if and only if $\overline{U}U=1$ (equivalently $U\overline{U}=1$) in  $A\overline{\otimes}B(H)$.
\end{enumerate}

\end{lemma}

\begin{proof}
Observe that for $U\in M_s$ we have $U^*=\overline{U}$. Now assertions $(i)-(iii)$ follow easily.

$(iv)$ The first part follows from the definition of normal elements using the previous paragraph. The special  case follows from the equality 
$\overline{U}U=\overline{U\overline{U}}$.

$(v)$ $U$ is unitary in $M_s$ if and only if $U$ is unitary in $M$, i.e., $U^*U=UU^*=1$. It means that
$\overline{U}U=U\overline{U}=1$. But the two equalities are equivalent as $\overline{U}U=\overline{U\overline{U}}$ and $\overline{1}=1$
\end{proof}

\begin{prop}\label{P:le1 AotimesB(H)s}
 Let $U\in M_s$ be a tripotent. Then we have the following
\begin{enumerate}[$(a)$]
      \item $U\sim_{h,t}1$ if and only if $U=V_1V_2\cdots V_k$, where $V_1,V_2,\dots,V_k$ are symmetries in $M$ and, moreover,
    $$V_1,V_1V_2,V_1V_2V_3,\dots,V_1V_2\cdots V_k\in M_s.$$
    \item $U\le_{h,t}1$ if and only if there are symmetries   $V_1,V_2,\dots,V_k\in M$ satisfying the assumptions of $(a)$ and a projection $P\in M$ such that 
    $$U=PV_1V_2\cdots V_m=V_1V_2\cdots V_mP^t.$$
      \item $U\sim_{hc,t}1$ if and only if there are symmetries  $V_1,V_2,\dots,V_k\in M$ satisfying the assumptions of $(a)$ and a complex unit $\alpha$ such that 
    $$U=\alpha V_1V_2\cdots V_k.$$
    \item $U\le_{hc,t}1$ if and only if there are symmetries   $V_1,V_2,\dots, V_k\in M$ satisfying the assumptions of $(a)$, a projection $P\in M$ and a complex unit $\alpha$ such that 
    $$U=\alpha PV_1V_2\cdots V_m=\alpha V_1V_2\cdots V_mP^t.$$
       \item $U\le_{n,t}1$ holds always.
    \end{enumerate}
\end{prop}

\begin{proof}
 Assertion $(a)$ can be proved following the proof of Propostition~\ref{P:le1 C*}$(f)$. 
 
 Assertion $(c)$ follows from $(a)$ using Proposition~\ref{P:hct-charact}.
 
 Assertion $(e)$ follows from Remark~\ref{rem:AotimesB(H)s is finite}.

$(b)$ By Lemma~\ref{L:factor leht} $U\le_{h,t}1$ if and only if there is some tripotent $W\in M_s$ with $U\le W$ and $W\sim_{h,t}1$.
By \cite[Proposition 4.6]{Finite} $U\le W$  if and only if there is a projection $P\in M$ such that $U=PW$. Moreover, 
$$PW=U=U^t=(PW)^t=W^t P^t=WP^t.$$

Hence the assertion follows easily from $(a)$.

Assertion $(d)$ follows from $(b)$ using Proposition~\ref{P:hct-charact}.
\end{proof}

The next proposition describes relations $U\,R\,P$ where $P$ is a projection.

\begin{prop}\label{P:leP AotimesB(H)s}
Let $P\in M_s$ be a projection and $U\in M_s$ be a tripotent. Then the following assertions hold.
\begin{enumerate}[$(a)$]
    \item If $R$ is any of the relations
    $$\le,\le_r,\le_c,\le_h,\le_{hc},\le_n,$$
    then
    $$U\,R\,P \Leftrightarrow U\,R\,1 \mbox{ and }U\le_2 P.$$
    \item $U\le_{n,t} P\Leftrightarrow U\le_2 P\Leftrightarrow \overline{U}U \le P\mbox{ in M}\Leftrightarrow U\overline{U} \le P\mbox{ in M}$.
    \item $U\sim_2P\Leftrightarrow \overline{U}U=P\Leftrightarrow U\overline{U}=P$.
    \item Let $\kappa\in\{r,c,h,hc\}$. Then
    $$U\sim_\kappa P\Leftrightarrow U\sim_2 P\mbox{ and }U\le_\kappa 1.$$
    \item $U\sim_{h,t} P$ if and only if there are $V_1,\dots,V_k\in M$ such that
    \begin{enumerate}[$(i)$]
        \item $V_j=V_j^*$ and $V_j^2=P$ for $j=1,\dots,k$;
        \item $V_1,V_1V_2,\dots,V_1\cdots V_k\in M_s$;
        \item $U=V_1V_2\cdots V_k$.
    \end{enumerate}
 \item $U\le_{h,t}P$ if and only if there are $V_1,\dots,V_k\in M$ satisfying conditions $(i)$ and $(ii)$ from assertion $(e)$ and a projection $Q\in M$ such that $Q\le P$ and
 $$U=QV_1V_2\cdots V_k=V_1V_2\cdots V_kQ^t.$$
 \end{enumerate}

\end{prop}

\begin{proof}
 $(a)$ This follows from Proposition~\ref{p:partial transitivity}.
 
 Before proceeding observe that
 \begin{equation}\label{eq:up}
  U^*=\overline{U},\  \overline{\overline{U}U}=U\overline{U}\mbox{ and }   \overline{P}=P,
 \end{equation} where the third equality follows from  Lemma~\ref{L:AotimesB(H)s-translations}$(ii)$.
 
 $(b)$ The first equivalence follows from Remark~\ref{rem:AotimesB(H)s is finite}. To show the second equivalence recall that $U\le_2 P$ if and only if $U\in M_2(P)$, which takes place if and only if both $U^*U\le P$ and $UU^*\le P$.
 It remains to use \eqref{eq:up}.

 $(c)$ $U\sim_2 P$ means that $U$ is unitary in $M_2(P)$, i.e., $U^*U=UU^*=P$.
 It remains to use \eqref{eq:up}.
 
 $(d)$ The first two cases are easy, the second two follow from Propositions~\ref{P:charact simh} and~\ref{P:hc-charact}$(b)$.
 
 $(e)$ This follows from the proof of Proposition~\ref{P:le1 C*}$(f)$ as $P$ is the unit of the C$^*$-algebra $M_2(P)$.
 
 $(f)$ This follows from $(e)$ using Lemma~\ref{L:factor leht} and \cite[Proposition 4.6]{Finite}. 
\end{proof}

 Note that assertions $(a)$ and $(d)$ from the previous proposition are abstract and hold in any triple. But we formulate them here because in combination with the respective assertions of Proposition~\ref{P:le1 C*} and Lemma~\ref{L:AotimesB(H)s-translations} they provide a concrete description.
 
 \begin{example2}\label{ex:M2s examples}
 \begin{enumerate}[$(a)$]
     \item Example~\ref{ex:leh not transitive} shows that the relations $\le_h$ and $\sim_h$ are not transitive in $M_s$, as the tripotents in that example are symmetric matrices.
     \item Let
     $$u=\begin{pmatrix} \frac{1}{\sqrt{2}} &  \frac{i}{\sqrt{2}} \\  \frac{i}{\sqrt{2}} &  \frac{1}{\sqrt{2}}\end{pmatrix}, \quad v=\begin{pmatrix} 1 & 0 \\ 0 & -1 \end{pmatrix},\quad e=\begin{pmatrix} 1 & 0 \\ 0 & 1 \end{pmatrix}.$$
     Then $u,v,e$ are symmetric unitary matrices such that
     $u\sim_h v\sim_h e$ but $u$ and $e$ are incomparable with respect to $\le_{hc}$.
     
     Indeed, $e$ is the unit and $v$ is a symmetry, so $v\sim_h u$. Futher,
     $$\J vuv=vu^*v=\begin{pmatrix} 1 & 0 \\ 0 & -1 \end{pmatrix}\begin{pmatrix} \frac{1}{\sqrt{2}} &  -\frac{i}{\sqrt{2}} \\  -\frac{i}{\sqrt{2}} &  \frac{1}{\sqrt{2}}\end{pmatrix}\begin{pmatrix} 1 & 0 \\ 0 & -1 \end{pmatrix}=\begin{pmatrix} \frac{1}{\sqrt{2}} &  \frac{i}{\sqrt{2}} \\  \frac{i}{\sqrt{2}} &  \frac{1}{\sqrt{2}}\end{pmatrix}=u,$$
     hence $u\le_h v$. Since $u\sim_2 v$, we deduce that $u\sim_h v$.
     
     Moreover, $u$ is not a scalar multiple of a self-adjoint matrix, thus $u\not\le_{hc}e$. Since $u$ is unitary and hence $u\sim_2e$, we deduce that $u$ and $e$ are incomparable with respect to $\le_{hc}$.
     
     It follows that the relations $\le_{hc}$ and $\sim_{hc}$ are not transitive on $M_s$.
     
     \item The relation $\le_{n}$ is not transitive on $M_s$. Recall that $\le_{n,t}$ coincide with $\le_2$, in particular, each tripotent $u$ satisfies $u\le_{n,t}1$. However, there are tripotents in $M_s$ which are not normal operators, for example
     $$u=\begin{pmatrix}\frac12 & \frac i2\\\frac i2& -\frac12 \end{pmatrix}.$$
 \end{enumerate}
 \end{example2}

\subsection{The finite-dimensional case -- symmetric matrices}

This subsection is devoted to the analysis of $(M_n)_s$, symmetric $n\times n$ matrices, and of the respective tensor product $A\overline{\otimes}(M_n)_s$. A key role is played by the determinant, so we start by a technical lemma on behavior of determinants. The results are important for $n\ge2$ as $(M_1)_s$ is isomorphic to $\ce$, but they work for $n=1$ as well.

\begin{lemma}\label{L:det in Mns} Let $n
\in\en$.
\begin{enumerate}[$(a)$]
    \item Let $u\in (M_n)_s$ be unitary. Then
    $$u=\alpha_1p_1+\dots+\alpha_np_n,$$
    where $p_1,\dots,p_n$ are mutually orthogonal minimal projections in $(M_n)_s$ and $\alpha_1,\dots,\alpha_n$ are complex units.
    \item Let $u\in (M_n)_s$ be unitary. Fix a decomposition from $(a)$. Then
    $$\det u=\alpha_1\cdots\alpha_n.$$
    \item Let $T$ be a Jordan $*$-automorphism of $(M_n)_s$. Then $\det T(u)=\det u$ for each unitary $u\in (M_n)_s$.
    \item Let $T$ be a  triple automorphism of $(M_n)_s$. Then $\det T(u)=\det T(1)\cdot \det u$ for each unitary $u\in (M_n)_s$.
\end{enumerate}
\end{lemma}

\begin{proof}
Assertion $(a)$ follows from the spectral decomposition (cf. \cite[Proposition 2.2 (b)]{cartan6}) using the fact that $(M_n)_s$ is a finite-dimensional JB$^*$-algebra of rank $n$.

$(b)$ The formula from $(a)$ is also a spectral decomposition in $M_n$. Since minimal projections in $(M_n)_s$ are minimal also in $M_n$, it follows that $\alpha_1,\dots,\alpha_n$ are exactly the eigenvalues of $u$, each one counted with its multiplicity. Their product is exactly $\det u$.

Assertion $(c)$ follows from $(b)$ as Jordan $*$-automorphisms are linear and map minimal projections to minimal projections.

$(d)$ Fix a decomposition of $u$ from $(a)$. Then
$$T(u)=\alpha_1T(p_1)+\dots+\alpha_nT(p_n)$$
and $T(p_j)$ are mutually orthogonal minimal tripotents with $T(p_j)\le T(1)$. Since $T(1)$ is unitary in $(M_n)_s$ and hence also in $M_n$, $x\mapsto (T(1))^{-1}x$ is a triple automorphism of $M_n$. We get
$$(T(1))^{-1}T(u)=\alpha_1(T(1))^{-1}T(p_1)+\dots+\alpha_n(T(1))^{-1}T(p_n)$$
and $(T(1))^{-1}T(p_j)$ are minimal projections in $M_n$. So,
$$\det u=\alpha_1\cdots\alpha_n=\det ((T(1))^{-1}T(u))=(\det T(1))^{-1}\cdot\det T(u),$$
which completes the proof.
\end{proof}

Note that assertions $(c)$ and $(d)$ of the preceding lemma may be alternatively proved using  a result from \cite[page 199, case $III_n$]{kaup1997real} which says that any triple automorphism on $(M_n)_s$ is of the form $T(x) = u x u^t$ for some unitary $u\in M_n$.

We continue by characterizing relation $\sim_{h,t}$ for unitary elements in $(M_n)_s$.

\begin{prop}\label{P:simht1 in symmetric matrices}
Let $n\in\en$. Let $u,e\in (M_n)_s$ be two unitary elements. Then $u\sim_{h,t} e$ if and only if $\det u=\pm\det e$. Moreover, the respective chain of $\sim_h$ has length at most $2n-1$ (at most $2n-2$ in case $\det u=\det e$).
\end{prop}

\begin{proof}
Assume $u\sim_{h,t}e$ in $(M_n)_s$. Then $u\sim_{h,t}e$ in $M_n$ as well, so $\det u=\pm\det e$ by Proposition~\ref{P:all orders vN}$(c)$ (note that a symmetry has determinant $\pm1$).

Conversely, assume $\det u=\pm\det e$. By  Lemma~\ref{L:shift to 1} there is a triple automorphism $T$ of $(M_n)_s$ with $T(e)=1$. By Lemma~\ref{L:det in Mns}$(d)$ we deduce that $
\det T(u)=\pm1$.
Thus, we may without loss of generality assume that $e=1$.

So, assume that $\det u=\pm 1$. Fix a decomposition of $u$ from Lemma~\ref{L:det in Mns}$(a)$. Then $p_1,\dots,p_n$ is a frame in $(M_n)_s$ and $p_1+\dots+p_n=1$. By applying Lemma~\ref{L:frames}$(i)$ to this frame and the canonical frame formed by diagonal matrices with exactly one $1$ on the diagonal (completed by zeros), we get a triple automorphism $S$ of $(M_n)_s$ such that $S(1)=1$ and $S(u)$ is a diagonal matrix. Then $S$ is even Jordan $*$-automorphism, hence $\det S(u)=\det u=\pm1$ (cf. Lemma \ref{L:det in Mns}$(c)$).
As a consequence we may assume without loss of generality that $u$ is a diagonal matrix.

Let us proceed by induction on $n$. 
For $n=1$ we have $(M_n)_s=\ce$, hence we even have $u\sim_r 1$, hence $u\sim_h1$. The next step is $n=2$. So, assume that $n=2$ and $u$ is a diagonal matrix with $\det u=\pm1$. It means that there is a complex unit $\alpha$ such that
$$u=\begin{pmatrix} \alpha & 0 \\ 0 &\overline{\alpha}\end{pmatrix} \mbox{\quad or\quad}u=\begin{pmatrix} \alpha & 0 \\ 0 &-\overline{\alpha}.\end{pmatrix}$$
In the first case we get
$$\begin{pmatrix} \alpha & 0 \\ 0 &\overline{\alpha}\end{pmatrix}\sim_h \begin{pmatrix} 0 & 1 \\ 1 &0\end{pmatrix}\sim_h 1.$$
To establish the second case it is enough to observe that
$$ \begin{pmatrix} \alpha & 0 \\ 0 &-\overline{\alpha}\end{pmatrix}\sim_h\begin{pmatrix} \alpha & 0 \\ 0 &\overline{\alpha}\end{pmatrix}$$

Next assume that the statement holds for $n$. Let us prove it for $n+1$. Let 
$$u=\begin{pmatrix}\alpha_1&0&0&\dots& 0 \\ 0&\alpha_2&0&\dots&0 \\ 0&0&\alpha_3&\dots&0
\\ \vdots&\vdots&\vdots&\ddots&\vdots \\ 0&0&0&\dots&\alpha_{n+1}\end{pmatrix},$$
where $\alpha_1,\dots,\alpha_{n+1}$ are complex units with $\alpha_1\cdots\alpha_{n+1}=\pm1$. By the case $n=2$ we see that
$$\begin{pmatrix}\alpha_1&0&0&\dots& 0 \\ 0&\alpha_2&0&\dots&0 \\ 0&0&\alpha_3&\dots&0
\\ \vdots&\vdots&\vdots&\ddots&\vdots \\ 0&0&0&\dots&\alpha_{n+1}\end{pmatrix}\sim_{h,t}
\begin{pmatrix}1&0&0&\dots& 0 \\ 0&\alpha_1\alpha_2&0&\dots&0 \\ 0&0&\alpha_3&\dots&0
\\ \vdots&\vdots&\vdots&\ddots&\vdots \\ 0&0&0&\dots&\alpha_{n+1}\end{pmatrix}$$
and the respective chain has length $2$. 

Now we may apply the induction hypothesis to the matrix formed by omitting the first row and the first column to show that
$$\begin{pmatrix}1&0&0&\dots& 0 \\ 0&\alpha_1\alpha_2&0&\dots&0 \\ 0&0&\alpha_3&\dots&0
\\ \vdots&\vdots&\vdots&\ddots&\vdots \\ 0&0&0&\dots&\alpha_{n+1}\end{pmatrix}\sim_{h,t}1$$
with the chain of length at most $2n-1$.
\end{proof}

We now obtain the following result on relations $\le_{h,t}$ and $\le_{hc,t}$ in $(M_n)_s$

\begin{prop}\label{P:coincidence in Mns}
Let $n\in\en$.
\begin{enumerate}[$(a)$]
    \item The relations $\le_2$ and $\le_{hc,t}$ coincide in $(M_n)_s$. In particular, the  relations $\sim_2$ and $\sim_{hc,t}$ coincide in $(M_n)_s$.
    \item To describe $\sim_{h,t}$ (or $\sim_{hc,t}$) in $(M_n)_s$ chains of $\sim_h$ (or $\sim_{hc}$) of length at most $2n-1$ are enough.
    \item To describe $\le_{h,t}$ (or $\le_{hc,t}$) in $(M_n)_s$ chains of $\le_h$ (or $\le_{hc}$) of length at most $2n$ are enough.
\end{enumerate}
\end{prop}

\begin{proof}
Let $u,e\in (M_n)_s$ be two tripotents such that $u\sim_2 e$. Let $k$ denote the rank of $e$. By Lemma~\ref{L:frames}$(ii)$ there is a triple automorphism $T$ of $(M_n)_s$ such that $T(e)$ is a diagonal matrix with $k$ ones and $n-k$ zeros on the diagonal. So, without loss of generality we may assume that already $e$ is of that form. In this case $((M_n)_s)_2(e)$ is isomorphic to $(M_k)_s$. By Proposition~\ref{P:simht1 in symmetric matrices} we deduce that $u\sim_{hc,t}e$ and that a chain of $\sim_{hc}$ of length $2k-1$ is enough. Moreover,  if even $u\sim_{h,t}e$, a chain of $\sim_{h}$ of length $2k-1$ is enough. This completes the proof of assertion $(a)$ for $\sim_2$ and of assertion $(b)$.

To prove the remaining part of $(a)$ assume $u\le_2 e$. Since $(M_n)_s$ is finite (see Remark~\ref{rem:AotimesB(H)s is finite}), there is a tripotent $v\in (M_n)_s$ with $u\le v\sim_2 e$. By the already proved part we get $v\sim_{hc,t}e$. Thus $u\le_{hc,t}e$.

Assertion $(c)$ follows from $(b)$ using Lemma~\ref{L:factor leht} (for $\le_{h,t}$) and additionally Proposition~\ref{P:hct-charact}$(a)$ (for $\le_{hc,t}$).
\end{proof}

\begin{lemma}\label{L:leht selection}
Let $n\in\en$. Then
$$A=\{u\in (M_n)_s\setsep u\mbox{ is a tripotent such that }u\le_{h,t}\1\}$$
is a compact set. Moreover, there is a Borel measurable mapping $\Phi:A\to (M_n)_s$ such that for each $u\in A$ its image $\Phi(u)$ is a tripotent such that $u\le \Phi(u)\sim_{h,t}\1$.
\end{lemma}

\begin{proof}
By Lemma~\ref{L:factor leht} $u\le_{h,t}\1$ if and only if there is a tripotent $v$ such that $u\le v\sim_{h,t}\1$. We therefore consider the set
$$B=\{(u,v)\in (M_n)_s\times (M_n)_s\setsep u\le v\sim_{h,t}\1\}.$$
Then $A$ is the projection of $B$ on the first coordinate. We observe that 
$$B=\{(u,v)\in (M_n)_s
\setsep u=\J uuu, v=\J vvv, u=\J uvu, \det v=\pm1\},$$
so $B$ is compact. We deduce that $A$ is compact as well. Moreover, $\Phi$ may be found as a Borel measurable selection of the mapping
$$u\mapsto \{v\setsep (u,v)\in B\}$$
which exists by a consequence of Kuratowski-Ryll-Nardzewski theorem (see, e.g., \cite[Theorem on p. 403]{Ku-RN}).
\end{proof}

\begin{prop}\label{P:AotimesMn}
Let $n\in \en$ and let $M=L^\infty(\mu,M_n)$, where $\mu$ satisfies \eqref{eq:CK=Linfty}. In this case we have $M=C(\Omega,M_n)$ (cf. Lemma~\ref{L:LinftyC=CKC}).
\begin{enumerate}[$(i)$]
    \item Let $\uu\in M_s (=L^\infty(\mu,(M_n)_s)=C(\Omega,(M_n)_s)$ be a unitary element. Then there is $f\in C(\Omega,\TT)$ such that $\uu\sim_{h,t} f\cdot\1$.
    \item Let $\uu,\e\in M_s$ be two tripotents. Then $\uu\sim_{h,t}\e$ (in $M_s$) if and only if $\uu(\omega)\sim_{h,t}\e(\omega)$ (in $(M_n)_s$) for each $\omega\in\Omega$. Moreover, chains of $\sim_h$ of length at most $2n-1$ are enough.
    \item Let $\uu,\e\in M_s$ be two tripotents. Then $\uu\le_{h,t}\e$ (in $M_s$) if and only if $\uu(\omega)\le_{h,t}\e(\omega)$ (in $(M_n)_s$) for each $\omega\in\Omega$. Moreover, chains of $\le_h$ of length at most $2n$ are enough.
\end{enumerate}
\end{prop}

\begin{proof} Let $p_1,\dots,p_n\in M_n$ be the canonical diagonal projections with exactly one $1$ on the diagonal. Further, set $e_k=p_1+\dots+p_k$ for $k\in\{1,\dots,n\}$. (In particular, then $e_n=\1$, the unit matrix.)

$(i)$ Assume $\uu\in M_s$ is unitary. Then $\uu(\omega)$ is unitary in $(M_n)_s$ (i.e., $\uu(\omega)\sim_2\1$ in $(M_n)_s$) for each $\omega\in \Omega$. We apply Lemma~\ref{L:Cartan vector diagonalization} to $\uu$ and $\1=p_1+\dots+p_n$ and get the respective mapping $\Theta$. Then
$\Theta(\1)=\1$ and $\Theta(\uu)(\omega)$ is diagonal matrix for each $\omega\in\Omega$.
Set  
$$g(\omega)=\det\uu(\omega)=\det\Theta(\uu)(\omega),\quad\omega\in\Omega.$$
Note that the second equality follows from Lemma~\ref{L:det in Mns}, Lemma~\ref{L:Cartan vector diagonalization}$(ii)$ and $(iii)$ and the fact that $\Theta$ is unital. Then $g$
is a continuous function, so by the assumption \eqref{eq:CK=Linfty} there is $f\in C(\Omega,\TT)$ with $f^n=g$ (cf. the proof of Proposition~\ref{P:products of symmetries}$(iv)$).
Next, by applying the procedure from the proof of Lemma~\ref{P:simht1 in symmetric matrices} to diagonal matrices whose entries are continuous functions we deduce that $\overline{f}\cdot\Theta(\uu)\sim_{h,t}\1$. It follows that $\Theta(\uu)\sim_{h,t}f\cdot\1$, thus $\uu\sim_{h,t}f\cdot\1$ (as $\Theta$ is a triple automorphism and $\Theta(f\cdot\1)=f\cdot\1$ by Lemma~\ref{L:Cartan vector diagonalization}$(iii)$).

$(ii)$ The `only if' part is obvious. To prove the converse assume that 
$\uu(\omega)\sim_{h,t}\e(\omega)$ for each $\omega\in\Omega$. Apply Lemma~\ref{L:Cartan vector parametrization} to  $\e$ and $e_1,\dots,e_n$ in place of $u_1,\dots,u_m$. We get a mapping $\Psi$.

Then $\Psi(\e)$ attains only values $0,e_1,\dots,e_n$ and, by the properties of $\Psi$ described in the just quoted Lemma, $\Psi(\uu)(\omega)\sim_{h,t}\Psi(\e)(\omega)$ for $\omega\in\Omega$. Hence we may apply Lemma~\ref{L:Cartan vector diagonalization}
to $\Psi(\uu)$ in place of $\uu$, $e_k=p_1+\dots+p_k$ and $e_k^j = p_j$. We get a mapping $\Theta$.

Then $\Theta(\Psi(\e))=\Psi(\e)$, $\Theta(\Psi(\uu))(\omega)\sim_{h,t}\Psi(\e)(\omega)$ for $\omega\in\Omega$ and the values of $\Theta(\Psi(\uu))$ are diagonal matrices.

By applying the procedure from the proof of Lemma~\ref{P:simht1 in symmetric matrices} to diagonal matrices whose entries are continuous functions we deduce that $\Theta(\Psi(\uu))\sim_{h,t}\Psi(\uu)$ using a chain of $\sim_h$ of length at most $2n-1$. (We proceed separately on each of the clopen sets $\Omega_k$.)

Now we deduce that the same holds for $\uu$ and $\e$.

$(iii)$ The `only if' part is obvious. The statement of the length chains follows from $(ii)$ and Lemma~\ref{L:factor leht}. So, it remains to prove the `if' part.

To this end assume that $\uu(\omega)\le_{h,t}\e(\omega)$ for each $\omega\in\Omega$. Up to applying Lemma~\ref{L:Cartan vector parametrization} as in the proof of $(ii)$ we may assume that the values of $\e$ are only $0,e_1,\dots,e_n$. By application of Lemma~\ref{L:leht selection} on the clopen sets $\Omega_1,\dots,\Omega_n$ we find a tripotent $\vv\in L^\infty(\Omega,(M_n)_s)=C(\Omega,(M_n)_s)$
such that for each $\omega\in\Omega$ we have $\uu(\omega)\le\vv(\omega)\sim_{h,t}\e(\omega)$. Clearly $\uu\le \vv$ and by $(ii)$ we get $\vv\sim_{h,t}\e$. Thus $\uu\le_{h,t}\e$ by Lemma~\ref{L:factor leht}.
 \end{proof}

\begin{question}
\begin{enumerate}[$(1)$]
    \item How long chains of $\sim_h$ are necessary to describe $\sim_{h,t}$ in the JBW$^*$-triple $(M_n)_s$ (or $L^\infty(\mu,(M_n)_s)$)? Is the bound $2n-1$ optimal? Is there a uniform bound independent on $n$?
    \item How long chains of $\sim_h$ are necessary to describe $\sim_{h,t}$ in $B(H)_s$ for an infinite-dimensional $H$? Is there some bound?
    \item Let $H$ be infinite-dimensional and let $u\in B(H)_s$ be unitary. Is $u\sim_{h,t}1$ in $B(H)_s$?
    \item Do the relations $\le_2$ and $\le_{hc,t}$ coincide in $B(H)_s$ for an infinite-dimensional $H$? 
    \end{enumerate}
\end{question}

\subsection{The antisymmetric case}

The case of $M_a=A\overline{\otimes}B(H)_a$ is quite different. It is a subtriple of $M=A\overline{\otimes}B(H)$, but not a JB$^*$-subalgebra. It is closed under the involution, but not under the Jordan product  (in fact, $x\circ y\in M_s$ whenever $x,y\in M_a$) and, moreover, it does not contain the unit of $M$.

 But sometimes $M_a$ admits a structure of a JBW$^*$-algebra. It depends on the dimension of $H$. To avoid trivial cases we assume $\dim H\ge3$ as it is usual. Basic properties are summarized in the following remark.

\begin{remark}\label{rem:AotimesB(H)a}
\begin{enumerate}[$(1)$]
    \item Assume $\dim H<\infty$. By \cite[Proposition 5.26$(a)$]{Finite} $A\overline{\otimes}B(H)_a$ is isomorphic to $L^\infty(\mu,B(H)_a)$ and it is a finite JBW$^*$-triple. Hence, by Proposition~\ref{P:le2=lent in finite} the relations $\le_2$ and $\le_{n,t}$ coincide in $M_a$.
    \item If $\dim H$ is finite and odd, then $A\overline{\otimes}B(H)_a$ contains no unitary element. This is known, an easy proof is given in  \cite[Proposition 5.26$(c)$]{Finite}.
    \item If $\dim H$ is finite and even, then $A\overline{\otimes}B(H)_a$ contains unitary elements, so it admits a structure of JBW$^*$-algebra. This is known, an easy proof is given in  \cite[Proposition 5.26$(b)$]{Finite}.
    \item If $\dim H$ is infinite, then $A\overline{\otimes}B(H)_a$ contains unitary elements, hence it admits a structure of JBW$^*$-algebra. Moreover, this triple is not finite, there are complete non-unitary tripotents. These facts are proved in \cite[Proposition 5.27]{Finite}. It follows that the relations $\le_2$ and $\le_{n,t}$ do not coincide in this case.
\end{enumerate}
\end{remark}

As remarked above, the unit of $M$ does not belong to $M_a$, so -- even if it admits a structure of a JBW$^*$-algebra -- there is no natural unit to apply the reductions from Lemmata~\ref{L:shift to 1} or~\ref{L:shift to projection}.  Moreover, neither a diagonalization may be used, as diagonal operators are not antisymmetric. However, at least some reductions are possible.

\begin{lemma}\label{L:C2-shift}
Let $U\in M_a$ be a tripotent. Then there is a unitary element $V\in M_s$ such that $VUV$ is self-adjoint in $M$.
In this case $T\mapsto VTV$ is a triple automorphism of $M$ commuting with the transpose. In particular, it is a triple automorphism of $M_a$.

If $U$ is even a unitary, then $VUV$ is a symmetry in $M$
\end{lemma}

\begin{proof} By \cite[Lemma 5.22]{Finite} we have $U=W-W^t$ for a tripotent $W$ such that $W\perp W^t$. It follows that $W+W^t$ is tripotent in $M_s$.
Since $M_s$ is finite, there is a unitary element $\widetilde{U}\in M_s$ with $W+W^t\le \widetilde{U}$.
By Lemma~\ref{L:sqrt of unitary} there is a unitary element $V\in M_s$ commuting with $\widetilde{U}$ such that $V^2=\widetilde{U}^*$.

Then $T\mapsto VTV$ is a  triple automorphism of $M$ (cf. Lemma~\ref{L:shift to 1}). Moreover, it clearly commutes with the transpose and hence it maps $M_a$ onto $M_a$.

Note that $V\widetilde{U}V=1$, the unit of $M$. Hence $VWV$ and $VW^tV$ are mutually orthogonal projections. Then
$$VUV=VWV-VW^tV$$
is self-adjoint.

In case $U$ is unitary, $VUV$ is a self-adjoint unitary, i.e., a symmetry in $M$.
\end{proof}

Note that, in case $U$ is unitary, the element $S=VUV$ given by the previous lemma is an `antisymmetric symmetry'. This may sound strange, but there is no contradiction -- the word `antisymmetric' means that $S^t=-S$ while the word `symmetry' means that $S^*=S$ and $S^2=1$. 

The previous  lemma says that if $R$ is any of the above relations, to understand when  $URE$ it is enough to assume that $E$ is self-adjoint in $M$ (or even a symmetry if $E$ is unitary).

\begin{example2}\label{ex:c2dim3}
Assume that $\dim H=3$.
\begin{enumerate}[$(a)$]
    \item The rank of $B(H)_a$ is $1$, i.e., any nonzero tripotent in $B(H)_a$ is simultaneously complete and minimal (i.e., its Peirce-$2$ subspace is one-dimensional). Hence, the relations are characterized in the same way as in Proposition~\ref{P:orders special cases}$(a)$ except that in $(iii)$ the second equivalence should be omitted.
    \item If $M=L^\infty(\mu,B(H))$, then the characterizations from Proposition~\ref{P:orders special cases}$(b)$ hold except for condition $(vi)$ which is replaced by
    \begin{enumerate}
        \item[$(vi')$] $u\le_n e\Leftrightarrow \exists h\in L^\infty(\mu)\colon u=h\cdot e$. 
    \end{enumerate}
\end{enumerate}
\end{example2}

If $\dim H\ge 4$, then the structure of the relations is richer, in a sense at least as rich as in $B(K)$ where $K$ is a Hilbert space whose dimension is the (integer part) half of $\dim H$. This is precised in the following proposition. 

\begin{prop}\label{P:E=W-Wt}
Let $E\in M_a$ be a tripotent. Then $E=W-W^t$ where $W\in M$ is a tripotent and $W\perp W^t$.

Let $V\in M$ be any tripotent such that $V\le_2 W$. Then  $V\perp V^t$, $U=V-V^t$ is a tripotent in $M_a$ satisfying $U\le_2 E$.

Further, in this case, if $R$ is any of the above-defined relations, then 
$$V\;R\;W \mbox{ in }M \Rightarrow U\;R\;E\mbox{ in } M_a.$$
Moreover, if $R\in\{\le,\le_r,\sim_r,\le_c,\sim_c,\le_h,\sim_h,\le_{hc},\sim_{hc},\le_n,\le_2,\sim_2\}$, then
$$V\;R\;W \mbox{ in }M \Leftrightarrow U\;R\;E\mbox{ in } M_a.$$
\end{prop}

\begin{proof}
The existence of $W$ follows from  \cite[Lemma 5.22]{Finite}.
By Lemma~\ref{L:C2-shift} (and its proof) we may assume that $E$ is self-adjoint in $M$ and $W$ is a projection in $M$. (This is not essential but it simplifies the arguments.)

Assume $V\le_2W$, i.e., $\J WWV=V$. Since the transpose defines a triple automorphism on $M$, $V^t$ is clearly a tripotent and, moreover,
$$\J {W^t}{W^t}{V^t}=(\J WWV)^t=V^t,$$
thus $V^t\le_2 W^t$. Now it is clear that $V^t\perp V$, $U=V-V^t$ is a tripotent in $M_a$ and $U\le_2 E$. 

We continue by proving the equivalences for the named relations.

\begin{description}
\item[$\le$] We have
$$V\le W \Leftrightarrow \J VVW=V \Leftrightarrow
(\J VVW)^t=V^t \Leftrightarrow
\J {V^t}{V^t}{W^t}=V^t \Leftrightarrow V^t\le W^t,$$
hence the implication $\Rightarrow$ is in this case obvious. To prove the converse assume $U\le E$. Then
$$\begin{aligned}
V-V^t&=\J{V-V^t}{V-V^t}{W-W^t}=\J VVW-\J{V^t}{V^t}{W^t},
\end{aligned}$$
where we used that $V\perp V^t$, $V\perp W^t$ and $V^t\perp W$. Moreover, it follows by the Peirce calculus that $\J VVW\in M_2(W)$ and $\J {V^t}{V^t}{W^t}\in M_2(W^t)\subset M_0(W)$. It follows that $V=\J VVW$, i.e., $V\le W$.
\item[$\le_r,\le_c$] These cases follow from the case `$\le$' together with the linearity of the transpose.
\item[$\sim_r,\sim_c$] If $V=\alpha W$, then $V^t=\alpha W^t$, hence $U=\alpha E$. Conversely, if $U=\alpha E$, i.e., $V-V^t=\alpha W-\alpha W^t$, thus $V=\alpha W$.
\item[$\le_h$] Recall that
$$V\le_h W\Leftrightarrow \J WVW=V,$$
hence the proof is completely analogous to the proof of the case `$\le$'.
\item[$\le_{hc}$] This case follows from the case `$\le_h$' together with linearity of the transpose.
\item[$\le_n$] Assume $V\le_n W$. Then $\J VVW$ is a tripotent in $M_2(W)$. Hence
$$\J{V^t}{V^t}{W^t}=(\J VVW)^t$$
is a tripotent in $M_2(W^t)\subset M_0(W)$. Hence $\J{V^t}{V^t}{W^t}\perp \J VVW$, thus
$\J VVW-\J {V^t}{V^t}{W^t}$ is a tripotent. Since
$$\J{V-V^t}{V-V^t}{W-W^t}=\J VVW-\J{V^t}{V^t}{W^t},$$
we deduce $E=V-V^t\le_n W-W^t=E$.

Conversely, assume that $U\le_n E$. Then
$$\J{V-V^t}{V-V^t}{W-W^t}=\J VVW-\J{V^t}{V^t}{W^t}$$
is a tripotent. Since $\J VVW\in M_2(W)$ and $\J{V^t}{V^t}{W^t}\in M_0(W)$, it follows easily that both $\J VVW$ and $\J{V^t}{V^t}{W^t}$ are tripotents. Hence $V\le_n W$.

\item[$\le_2,\sim_2$] The equivalence for $\le_2$ is trivial. Observe that, assuming $V\le_2 W$ (which we do assume from the beginning), we have
$$V\sim_2W\Leftrightarrow\J VVW=W,$$
hence the proof is completely analogous to the proof of the case `$\le$'.

\item[$\sim_h,\sim_{hc}$] These cases follow by combining the cases `$\le_h,\le_{hc}$' with the case `$\sim_2$'.
\end{description}

The remaining relations are transitive hulls of the respective relations, so the remaining implications follow from the ones already proved.
\end{proof}

\begin{example2}
\begin{enumerate}[$(a)$]
    \item We may use Example~\ref{ex:leh not transitive} and Proposition~\ref{P:E=W-Wt} to show that the relations $\le_h$ and $\sim_h$ are not transitive on $(M_4)_a$. It is enough to consider matrices
    $$e=\begin{pmatrix}
    0&0&1&0 \\ 0&0&0&1 \\-1&0&0&0\\0&-1&0&0
    \end{pmatrix}, u=\begin{pmatrix}
    0&0&0&-1 \\ 0&0&-1&0 \\0&1&0&0\\1&0&0&0
    \end{pmatrix},
    v=\begin{pmatrix}
    0&0&i&0 \\ 0&0&0&-i \\-i&0&0&0\\0&i&0&0
    \end{pmatrix}.$$
    \item Similarly we may use Example~\ref{ex:lehc not transitive} to show that $\le_{hc}$ and $\sim_{hc}$ are not transitive in $(M_4)_a$.
    \item In the same way we use Example~\ref{ex:len not transitive}$(c)$ to show that $\le_n$ is not transitive in $(M_4)_a$.
\end{enumerate}
\end{example2}

\subsection{The case of finite even dimension}

Let us now focus on $(M_{2n})_a$ for $n\ge2$. (The analysis below is valid for $n=1$, i.e., for $(M_2)_a$ as well, but this case is trivial as $(M_2)_a$ is isomorphic to $\ce$.)

We know that $(M_{2n})_a$ admits unitary elements, but not a canonical unit. However, it is isomorphic to the classical JB$^*$-algebra $H_n(\Ha_C)$ of hermitian $n\times n$ matrices of biquaternions (which was studied for example in \cite{cartan6}). Biquaternions are quaternions with complex coefficients (see, e.g., \cite[Section 3]{cartan6} or \cite[Section 6]{Finite}). We will use the following matrix representation (cf. \cite[(5) in Section 3.2]{cartan6} or \cite[Lemma 6.7(ii)]{Finite}).

$\Ha_C$ is the C$^*$-algebra $M_2$ of $2\times 2$ complex matrices with usual multiplication and involution $^*$ equipped moreover with a linear involution $^\inv$ defined by
$$\begin{pmatrix}
a&b\\c&d
\end{pmatrix}^\inv =\begin{pmatrix}
d&-b\\-c&a
\end{pmatrix}.$$
Then $^\inv$ is a linear involution (i.e., $(xy)^\inv=y^\inv x^\inv$ for $x,y\in \Ha_C$) commuting with $^*$ (i.e., $(x^*)^\inv=(x^\inv)^*$).
Hence, $M_n(\Ha_C)$ is the C$^*$-algebra of $n\times n$ matrices with entries in $\Ha_C$, which is canonically isomorphic to $M_{2n}$ (if $\x\in M_n(\Ha_C)$, we denote by $\widehat{\x}$ the corresponding element of $M_{2n}$). It is further equipped with the linear involution $^\inv$ -- if $\x=(x_{ij})\in M_n(\Ha_C)$, then $\x^\inv$ is the $n\times n$ matrix with $x_{ji}^\inv$ on place $ij$. Then
$$H_n(\Ha_C)=\{\x\in M_n(\Ha_C)\setsep \x^\inv=\x\}$$
is a JB$^*$-subalgebra of $M_n(\Ha_C)$.

\begin{lemma}\label{L:M2na=HnHC}
Let $n\in\en$, $n\ge 2$. Then
$$\uu=\begin{pmatrix}
0&1&\dots&0&0\\
-1&0& \dots &0&0 \\
\vdots&\vdots&\ddots&\vdots&\vdots \\
0&0&\dots&0&1\\
0&0&\dots&-1&0
\end{pmatrix}$$
is a unitary element of $(M_{2n})_a$ with $\det\uu=1$. 

Hence the operator $T:\x\mapsto \x\uu$ is a triple automorphism of $M_{2n}$.

Moreover, $T$ maps $(M_{2n})_a$ onto $H_n(\Ha_C)$ (more precisely, onto the image of $H_n(\Ha_C)$ under the mapping $\x\mapsto\widehat{\x}$).
\end{lemma}

\begin{proof}
Clearly $\uu$ is an antisymmetric matrix, hence $\uu\in (M_{2n})_a$. It is also clear that $\uu$ is a unitary matrix and $\det\uu=1$. It can be easily checked that $T$ is a triple automorphism of $M_{2n}$. The only thing to be checked is that $T((M_{2n})_a)=H_n(\Ha_C)$.

To this end we will represent elements of $M_{2n}$ as elements of $M_n(M_2)=M_n(\Ha_C)$. Then
$$\uu=\begin{pmatrix}
u&0&\dots&0\\
0&u&\dots&0\\
\vdots&\vdots&\ddots&\vdots\\
0&0&\dots&u
\end{pmatrix},\mbox{ where }u=\begin{pmatrix}
0&1\\-1&0
\end{pmatrix}.$$
Hence, if $\x=(x_{ij})_{1\le i,j\le n}\in M_n(M_2)$, then
$$T(\x)=\x\uu=(x_{ij}u)_{1\le i,j\le n}.$$
Assume that $\x\in (M_{2n})_a$. Then  $x_{ji}=-x_{ij}^t$ for $i,j\in\{1,\dots,n\}$. Assume that $x_{ij}=\begin{pmatrix}
a&b\\c&d
\end{pmatrix}$. Then
$$\begin{aligned}
(x_{ij}u)^\inv&=\begin{pmatrix}
-b&a\\-d&c
\end{pmatrix}^\inv=\begin{pmatrix}
c&-a\\d&-b
\end{pmatrix},\\
x_{ji}u&=-x_{ij}^tu=-\begin{pmatrix}
a&c\\b&d
\end{pmatrix}\begin{pmatrix}
0&1\\-1&0
\end{pmatrix}=-\begin{pmatrix}
-c&a\\-d&b
\end{pmatrix}=\begin{pmatrix}
c&-a\\d&-b
\end{pmatrix}, \end{aligned}$$
hence $x_{ji}u=(x_{ij}u)^\inv$. So, $T(\x)\in H_n(\Ha_C)$. 

Conversely, assume $T(\x)\in H_n(\Ha_C)$, i.e., $(x_{ij}u)^\inv=x_{ji}u$ for $i,j\in\{1,\dots,n\}$. Assume again that  $x_{ij}=\begin{pmatrix}
a&b\\c&d
\end{pmatrix}$. Then
$$x_{ji}=x_{ji}uu^*=(x_{ij}u)^\inv u^*=\begin{pmatrix}
c&-a\\d&-b
\end{pmatrix}\begin{pmatrix}
0&-1\\1&0
\end{pmatrix}=\begin{pmatrix}
-a&-c\\-b&-d
\end{pmatrix}=-x_{ij}^t,
$$
so $\x\in (M_{2n})_a$.

This completes the proof.
\end{proof}

\begin{lemma}\label{L:simht in HnHC}
Let $n\in\en$, $n\ge2$.
\begin{enumerate}[$(a)$]
    \item Let $\e\in H_n(\Ha_C)$ be a unitary element. Then 
    $$\e=\alpha_1 \p_1+\dots+\alpha_n \p_n,$$
    where $\p_1,\dots,\p_n$ are mutually orthogonal minimal projections in $H_n(\Ha_C)$ and $\alpha_1,\dots,\alpha_n$ are complex units. Moreover,
    $$\alpha_1\cdots\alpha_n=\dt_n\e,$$
    where $\dt_n$ is the determinant defined in \cite[Section 5]{cartan6}.
    \item Let $\e,\uu\in H_n(\Ha_C)$ be two unitary elements.
    Then 
    $$\uu=\alpha_1 \e_1+\dots+\alpha_n \e_n,$$
    where $\e_1,\dots,\e_n$ are mutually orthogonal minimal tripotents in $H_n(\Ha_C)$ satisfying $\e_j\le\e$ for each $j$ and $\alpha_1,\dots,\alpha_n$ are complex units. Moreover,
    $$\alpha_1\cdots\alpha_n=\dt_{n,\e}\uu,$$
    where $\dt_{n,\e}$ is the quantity defined in \cite[Section 5]{cartan6}.
   \item  Let $\e,\uu\in H_n(\Ha_C)$ be two unitary elements. Then $\uu\sim_{h,t}\e$ if and only if $\dt_n\uu=\pm\dt_n\e$. Moreover, chains of $\sim_h$ of length $2n-1$ are enough.   
\end{enumerate}
\end{lemma}

\begin{proof}
$(a)$ The existence of such a decomposition follows from the spectral theorem (cf. \cite[Lemma 2.2]{cartan6}) using the fact that $H_n(\Ha_C)$ has rank $n$. The same formula provides also the spectral decomposition of $\e$ in $M_n(\Ha_C)=M_{2n}$.

We claim that each $\p_j$ has rank $2$ in $M_{2n}$. Indeed, by \cite[Lemma 5.22]{Finite} any tripotent in $(M_{2n})_a$ has even rank in $M_{2n}$. Using the automorphism $T$ from Lemma~\ref{L:M2na=HnHC} we deduce that the same is true for tripotents in $H_n(\Ha_C)$. Since $M_{2n}$ has rank $2n$, necessarily the rank of each $p_j$ is $2$. The last identity now follows from \cite[Theorem 5.1$(ix)$]{cartan6}.

$(b)$ Let $S:M_{n}(\Ha_C)\to M_{n}(\Ha_C)$ be an operator provided by \cite[Lemma 5.2]{cartan6} (denoted there by $T$). Then $S\e=\1$ and $S\uu$ is a unitary element in $H_n(\Ha_C)$. Let 
$$S\uu=\alpha_1\p_1+\dots+\alpha_n\p_n$$ 
be the decomposition of $S\uu$ provided by $(a)$. Then
$$\uu=\alpha_1 S^{-1}(\p_1)+\dots+\alpha_n S^{-1}(\p_n).$$
is the required decomposition. The equality now follows from \cite[Proposition 5.3]{cartan6} and $(a)$.

$(c)$ Assume first $\uu\sim_h\e$. By Proposition~\ref{P:charact simh} we have $\uu=\vv_1-\vv_2$, where $\vv_1,\vv_2$ are two orthogonal tripotents with $\vv_1,\vv_2\le \e$. By decomposing $\vv_1$ and $\vv_2$ to minimal tripotents we see that $\dt_{n,\e}\uu=\pm1$. By \cite[Proposition 5.3$(ii)$]{cartan6} we deduce that $\dt_n\uu=\pm\dt_n\e$.
By induction we now see that $\uu\sim_{h,t}\e$ implies $\dt_n\uu=\pm\dt_n\e$.

Conversely, assume $\dt_n\uu=\pm\dt_n\e$. By \cite[Proposition 5.3$(ii)$]{cartan6} it means that $\dt_{n,\e}\uu=\pm1$.  So, fix a decomposition of $\uu$ as in $(b)$. Then $\alpha_1\cdots\alpha_n=\pm1$. Let $\p_1,\dots,\p_n$ be the canonical diagonal projections in $H_n(\Ha_C)$ having on the diagonal exactly once the unit matrix of order two. 

Then $\e_1,\dots,\e_n$ and $\p_1,\dots,\p_n$ are two frames in $H_n(\Ha_C)$, so
 Lemma~\ref{L:frames} provides a triple automorphism $S$ of $H_{n}(\Ha_C)$ such that $S(\e_j)=\p_j$ for each $j$. Then
 $S(\e)=\1$ and
$$S(\uu)=\alpha_1 \p_1+\dots+\alpha_n\p_n$$
so $S(\uu)$ is a diagonal matrix in $H_n(\Ha_C)$.
Set
$$E=\left\{\x=(x_{ij})\in H_n(\Ha_C)\setsep \forall i,j\in\{1,\dots,n\}\colon  \begin{array}{cc}
     x_{ij} \mbox{ is a complex multiple}  \\
     \mbox{of the unit matrix}
\end{array}\right\}$$
Then $E$ is a JB$^*$-subalgebra of $H_n(\Ha_C)$ canonically isomorphic to $(M_n)_s$. We have $S(\uu)\in E$ and the determinant of the respective $n\times n$ matrix is $\pm 1$. By Proposition~\ref{P:simht1 in symmetric matrices} we deduce that $S(\uu)\sim_{h,t}\1$ in $E\cong (M_n)_s$ and the respective chain of $\sim_h$ has length at most $2n-1$. Since $S$ is a triple automorphism of $H_n(\Ha_C)$, we deduce
$\uu\sim_{h,t}\e$ in $H_n(\Ha_C)$. The necessary length of chains of $\sim_h$ remains to be bounded by $2n-1$.
This completes the proof.
\end{proof}

\begin{prop}\label{P:simht in M2na}
Let $n\in\en$, $n\ge 2$. Let $\vv,\e\in (M_{2n})_a$ be two unitary elements. Then $\vv\sim_{h,t}\e$ if and only if $\det\vv=\det\e$. Moreover, chains of $\sim_h$ of length $2n-1$ are enough.
\end{prop}

\begin{proof}
Let $\uu$ and $T$ be as in Lemma~\ref{L:M2na=HnHC}. Then
\begin{multline*}
\vv\sim_{h,t}\e \mbox{ in }(M_{2n})_a \Leftrightarrow T(\vv)\sim_{h,t}T(\e) \mbox{ in }H_n(\Ha_C)
\\ \Leftrightarrow
\dt_n T(\vv)=\pm \dt_n T(\e) \Leftrightarrow
(\dt_n T(\vv))^2=(\dt_n T(\e))^2.\end{multline*}
The first equivalence follows from the fact that $T$ is a triple isomorphism of $(M_{2n})_a$ and $H_n(\Ha_C)$ (this follows from Lemma~\ref{L:M2na=HnHC}). The second equivalence follows from Lemma~\ref{L:simht in HnHC}$(c)$ an the third one is obvious.

We further have
$$(\dt_n T(\vv))^2=\det\widehat{T(\vv)}=\det (\vv\uu)=\det\vv\cdot\det\uu=\det\vv.$$
Indeed, the first equality follows from \cite[Theorem 5.1$(vii)$]{cartan6}, the second one from the definition of $T$. The third one is a consequence of the classical theorem on determinant of a product and the last one is valid as $\det\uu=1$. 

Similarly we get $(\dt_n T(\e))^2=\det\e$. 

This completes the proof of the equivalence $\vv\sim_{h,t}\e\Leftrightarrow\det\vv=\det\e$. The bound on the length of chains follows from Lemma~\ref{L:simht in HnHC}$(c)$.
\end{proof}

\begin{lemma}\label{L:leht selection a}
Let $n\in\en$, $n\ge2$. Let $\e\in (M_{2n})_a$ be a fixed unitary element. Then
$$A=\{\uu\in (M_{2n})_a\setsep \uu\mbox{ is a tripotent such that }\uu\le_{h,t}\e\}$$
is a compact set. Moreover, there is a Borel measurable mapping $\Phi:A\to (M_{2n})_a$ such that for each $\uu\in A$ its image $\Phi(\uu)$ is a tripotent such that $\uu\le \Phi(\uu)\sim_{h,t}\e$.
\end{lemma}

\begin{proof}
The proof may be done by a slight modification of the proof of Lemma~\ref{L:leht selection}.
\end{proof}

\subsection{The case of a general finite dimension}

Next we are going to apply the results from the previous subsection to analyze the relations in $(M_n)_a$ and $L^\infty(\mu,(M_n)_a)$ for general $n\ge 4$.

\begin{lemma}\label{L:Mna rank}
Let $n\in\en$, $n\ge 4$. 
\begin{enumerate}[$(i)$]
    \item The rank of $(M_n)_a$ equals $\lfloor \frac n2\rfloor$, the integer part of $\frac n2$.
    \item Let $\uu\in (M_n)_a$ be a tripotent of rank $k$. Then its Peirce-$2$ subspace is triple-isomorphic to $(M_{2k})_a$ and hence to $H_k(\Ha_C)$.
\end{enumerate}
\end{lemma}

\begin{proof} For $k\le\lfloor\frac n2\rfloor$ let
$$\e_k=\begin{pmatrix}
0&1&0&0&\dots&0&0&\dots &0\\
-1&0&0&0&\dots&0&0&\dots &0\\
0&0&0&1&\dots&0&0&\dots &0\\
0&0&-1&0&\dots&0&0&\dots &0\\
\vdots&\vdots&\vdots&\vdots&\ddots&\vdots&\vdots&\vdots&\vdots\\
0&0&0&0&\dots&0&1&\dots &0\\
0&0&0&0&\dots&-1&0&\dots &0\\
\vdots&\vdots&\vdots&\vdots&\vdots&\vdots&\vdots&\ddots&\vdots\\
0&0&0&0&\dots&0&0&\dots &0
\end{pmatrix},$$
where the number of nonzero rows (or columns) is exactly $2k$. Then $\e_k$ is a tripotent of order $k$
(it is complete if $k=\lfloor\frac n2\rfloor$ and unitary if additionally $n$ is even). Clearly its Peirce-$2$ subspace is isomorphic to $(M_{2k})_a$, hence to $H_k(\Ha_C)$ (by Lemma~\ref{L:M2na=HnHC}).

Now $(i)$ is obvious and $(ii)$ follows by Lemma~\ref{L:frames}.
(Note that the statement $(i)$ is a well known fact, which can be found in \cite{loos1977bounded} in the finite-dimensional case or in \cite[Table 1 in page 210]{kaup1997real}.)
\end{proof}

\begin{prop}\label{P:coincidence in Mna}
Let $n\in\en$, $n\ge 4$.
\begin{enumerate}[$(a)$]
    \item The relations $\le_2$ and $\le_{hc,t}$ coincide in $(M_n)_a$. In particular, the  relations $\sim_2$ and $\sim_{hc,t}$ coincide in $(M_n)_a$.
    \item To describe $\sim_{h,t}$ (or $\sim_{hc,t}$) in $(M_n)_a$ chains of $\sim_h$ (or $\sim_{hc}$) of length at most $2\lfloor \frac n2\rfloor-1$  are enough .
    \item To describe $\le_{h,t}$ (or $\le_{hc,t}$) in $(M_n)_s$ chains of $\le_h$ (or $\le_{hc}$) of length at most $2\lfloor \frac n2\rfloor$ are enough.
\end{enumerate}
\end{prop}

\begin{proof} Let $\uu,\e\in (M_n)_a$ be two tripotents such that $\uu\sim_2 \e$. If $n$ is even and $\e$ is unitary, Propositon~\ref{P:simht in M2na} yields that $\uu\sim_{h,t} \alpha\cdot \e$ where $\alpha$ is a complex unit such that $\alpha^n=\det\uu\cdot\overline{\det\e}$ and the length of the respective chain is at most $n-1$. In general Lemma~\ref{L:Mna rank} says that the Peirce-$2$ subspace of $\e$ is isomorphic to $(M_{2k})_a$, hence by the unitary case
$\uu\sim_{hc,t}\e$ and the length of the respective chain is $2k-1$. This completes the proof of assertion $(a)$ for $\sim_2$ and of assertion $(b)$.

To prove the remaining part of $(a)$ assume $\uu\le_2 \e$. Since $(M_n)_a$ is finite (see Remark~\ref{rem:AotimesB(H)a}$(1)$), there is a tripotent $v\in (M_n)_s$ with $u\le v\sim_2 e$. By the already proved part we get $v\sim_{hc,t}e$. Thus $u\le_{hc,t}e$ (cf. Lemma~\ref{L:factor leht} and  Proposition~\ref{P:hct-charact}).

Assertion $(c)$ follows from $(b)$ using Lemma~\ref{L:factor leht} (for $\le_{h,t}$) and additionally Proposition~\ref{P:hct-charact}$(a)$ (for $\le_{hc,t}$).
\end{proof}

\begin{prop}
Let $n\in \en$, $n\ge 4$ and let $M=L^\infty(\mu,M_n)$, where $\mu$ satisfies \eqref{eq:CK=Linfty}. In this case we have $M=C(\Omega,M_n)$ (cf. Lemma~\ref{L:LinftyC=CKC}).
\begin{enumerate}[$(i)$]
      \item Let $U,V\in M_a$ be two tripotents. Then $U\sim_{h,t}V$ (in $M_a$) if and only if $U(\omega)\sim_{h,t}V(\omega)$ (in $(M_n)_a$) for each $\omega\in\Omega$. Moreover, chains of $\sim_h$ of length at most $2\lfloor\frac n2\rfloor-1$ are enough.
    \item Let $U,V\in M_a$ be two tripotents. Then $U\le_{h,t}V$ (in $M_a$) if and only if $U(\omega)\le_{h,t}V(\omega)$ (in $(M_n)_a$) for each $\omega\in\Omega$. Moreover, chains of $\le_h$ of length at most $2\lfloor\frac n2\rfloor$ are enough.
\end{enumerate}
\end{prop}

\begin{proof} For $1\le j\le\lfloor\frac n2\rfloor$ let $\e_j\in (M_n)_a$ be the tripotent from the proof of Lemma~\ref{L:Mna rank} and $\p_j=\e_j-\e_{j-1}$ (we set $\e_0=0$). Then $\p_j$ form a frame in $(M_n)_a$.

$(i)$ The `only if' part is obvious. To prove the converse assume that 
$U(\omega)\sim_{h,t}V(\omega)$ for each $\omega\in\Omega$. Apply Lemma~\ref{L:Cartan vector parametrization} to $V$ in place of $\e$ and $\e_1,\dots,\e_{\lfloor \frac n2\rfloor}$ in place of $u_1,\dots,u_m$. We get a mapping $\Psi$.

Then $\Psi(V)$ attains only values $0,\e_1,\dots,\e_{\lfloor \frac n2\rfloor}$ and $\Psi(U)(\omega)\sim_{h,t}\Psi(V)(\omega)$ for $\omega\in\Omega$. Hence we may apply Lemma~\ref{L:Cartan vector diagonalization}
to $\Psi(U)$ in place of $\uu$ and $\e_k=\p_1+\dots+\p_k$. We get a mapping $\Theta$.

Then $\Theta(\Psi(V))=\Psi(V)$, $\Theta(\Psi(U))(\omega)\sim_{h,t}\Psi(V)(\omega)$ for $\omega\in\Omega$ and the values of $\Theta(\Psi(U))$ are linear combinations of the $\p_j$'s.

By Lemma~\ref{L:M2na=HnHC} (use $\e_{\lfloor\frac n2\rfloor}$ in place of $\uu$) we may transfer the situation to $H_{\lfloor\frac n2\rfloor}(\Ha_C)$ and then by applying the procedure from the proofs of Lemma~\ref{L:simht in HnHC} and Lemma~\ref{P:simht1 in symmetric matrices} to diagonal matrices whose entries are continuous functions we deduce that $\Theta(\Psi(U))\sim_{h,t}\Psi(V)$ using a chain of $\sim_h$ of length at most $2{\lfloor\frac n2\rfloor}-1$. (We proceed separately on each of the clopen sets $\Omega_k$.)

Now we deduce that the same holds for $U$ and $V$.

$(ii)$ The `only if' part is obvious. The statement of the length chains follows from $(i)$ and Lemma~\ref{L:factor leht}. So, it remains to prove the `if' part.

To this end assume that $U(\omega)\le_{h,t}V(\omega)$ for each $\omega\in\Omega$. Up to applying Lemma~\ref{L:Cartan vector parametrization} as in the proof of $(i)$ we may assume that the values of $\e$ are only $0,\e_1,\dots,\e_{\lfloor\frac n2\rfloor}$. By application of Lemma~\ref{L:leht selection a} on the clopen sets $\Omega_1,\dots,\Omega_{\lfloor\frac n2\rfloor}$ we find a tripotent $W\in L^\infty(\Omega,(M_n)_a)=C(\Omega,(M_n)_a)$
such that for each $\omega\in\Omega$ we have $U(\omega)\le W(\omega)\sim_{h,t} V(\omega)$. Clearly $U\le W$ and by $(i)$ we get $W\sim_{h,t}V$. Thus $U\le_{h,t} V$ by Lemma~\ref{L:factor leht}.
 \end{proof}

\begin{question}
\begin{enumerate}
    \item Let $U,E\in M_a$ be two tripotents such that $U\le_2 E$. Are there decompositions $U=V-V^t$ and $E=W-W^t$ such that $V,W$ are tripotents in $M$, $V\perp V^t$,  $W\perp W^t$ and $V\le_2 W$?
    
    What happens in case
    $$E=\begin{pmatrix}
    0&1&0&0\\
    -1&0&0&0\\
    0&0&0&1\\
    0&0&-1&0
    \end{pmatrix},
    U=\begin{pmatrix}
    0&0&\frac{1}{\sqrt2} &  \frac{1}{\sqrt2} \\
    0&0&\frac{1}{\sqrt2} &  -\frac{1}{\sqrt2} \\
    -\frac{1}{\sqrt2} &  -\frac{1}{\sqrt2}&0&0 \\
    -\frac{1}{\sqrt2} &  \frac{1}{\sqrt2}&0&0
    \end{pmatrix} \ ?$$
    
    \item How long chains of $\sim_h$ are necessary to define $\sim_{h,t}$ in $M_a$? Is the above bound for finite-dimensional $H$ optimal? Is there a uniform bound for a general $H$? 
    
    \item Do the relations $\sim_2$ and $\sim_{hc,t}$ coincide in $B(H)_a$ if $\dim H=\infty$?
\end{enumerate}
\end{question}

\section{Spin factors and exceptional Cartan factors}\label{sec:spinetc}

In this section we will deal with the summands of the form $A\overline{\otimes}C$, where $A$ is an abelian von Neumann algebra and $C$ is a Cartan factor of type $4$, $5$ or $6$. Cartan factors of type $4$ are called {\em spin factors} and they are
defined by introducing an alternative structure on a Hilbert space as we recall below. They are also JC$^*$-triples, i.e., subtriples of $B(H)$ for a Hilbert space $H$, but we will not use this fact as the definition introduces a nice enough structure to work with. Cartan factors of type $5$ and $6$ are exceptional, they are defined as certain matrices of complex octonions which form a non-associative algebra that may be viewed as the eight-dimensional spin factor with an additional structures.

\subsection{Spin factors}\label{sec:spin}
Let us start by recalling the definitions and fixing the notation.
Throughout this subsection $H$ will denote the Hilbert space  $\ell^2(\Gamma)$ for a set $\Gamma$ of cardinality at least $3$, equipped with the canonical (coordinatewise) conjugation. The hilbertian norm on $H$ will be denoted by $\norm{\cdot}_2$, the orthogonality induced by the inner product will be denoted by $\perp_2$. The Hilbert space $H$ can be regarded as a type 1 Cartan factor with its Hilbertian norm.  

We will consider  another structure of JB$^*$-triple on $H$ -- a  triple product and a norm defined by
$$\begin{aligned}
\J xyz&=\ip xy z+ \ip zy x - \ip{x}{\overline{z}}\overline{y},\\
\norm{x}^2&=\ip xx +\sqrt{\ip xx^2-\abs{\ip{x}{\overline{x}}}^2}.
\end{aligned}
$$
The resulting JB$^*$-triple, which is known as a \emph{type 4  Cartan factor} or \emph{spin factor}, will be denoted by $C$. By $\perp$ we denote the relation of orthogonality of tripotents.

The spaces $H$ and $C$ are isomorphic as Banach spaces,
since clearly
$$\norm{x}_2\le\norm{x}\le\sqrt2\norm{x}_2\mbox{ for }x\in C.$$
In particular, $C$ is reflexive, hence it is a JBW$^*$-triple.

We will consider $A=L^\infty(\mu)$ for a probability measure satisfying \eqref{eq:CK=Linfty}. Then $A\overline{\otimes}C=L^\infty(\mu,C)$ due to the reflexivity of $C$ (cf. Lemma~\ref{L:LinftyC=CKC}$(i)$). 

\begin{remark}\label{rem:spin is finite}
By \cite[Corollary 6.4]{Finite} we know that $A\overline{\otimes}C$ is a finite JBW$^*$-triple. Hence, by Proposition~\ref{P:le2=lent in finite} the relations $\le_2$ and $\le_{n,t}$ coincide. We will show that much more is true.
\end{remark}

To simplify notation in the sequel we set
$$H_r=\{x\in H\setsep x\mbox{ has real coordinates}\}=\{x\in H\setsep \overline{x}=x\}.$$
Then $H_r$ is a real-linear subspace of $H$, it is a real Hilbert space. Moreover, 
$$\norm{x}=\norm{x}_2\mbox{ for }x\in H_r,$$
i.e., $H_r$ is also (isometrically) a real-linear subspace of $C$. This subspace will play a key role. 

We continue by a description of tripotents in $C$. The proof is known and easy (see, e.g., \cite[Lemma 6.1]{Finite} or \cite[Section 3]{KP2019} or \cite[\S 3.1.4]{FriedmanBook2005}).

\begin{lemma}\label{L:tripotents in spin factor}
The rank of $C$ equals $2$. Moreover, nonzero tripotents in $C$ are either unitary or minimal. They may be characterized as follows.
\begin{enumerate}[$(a)$]
    \item $u\in C$ is unitary if and only if $u=\alpha z$, where $\alpha$ is a complex unit and $z\in H_r$ satisfies $\norm{z}_2=1$.
    \item For $u\in C$ the following assertions are equivalent:
    \begin{enumerate}[$(i)$]
        \item $u$ is a minimal tripotent;
        \item $u\perp_2 \overline{u}$ and $\norm{u}_2=\frac1{\sqrt{2}}$;
        \item $u=a+ib$, where $a,b\in H_r$, $a\perp_2 b$ and $\norm{a}_2=\norm{b}_2=\frac12$.
    \end{enumerate}
     In this case
    $C_2(u)=\span \{u\}= \mathbb{C} u$, $C_0(u)=\span\{\overline{u}\}= \mathbb{C} \overline{u}$ and $C_1(u)=\{u,\overline u\}^{\perp_2}$ and the Peirce projections are the respective orthogonal projections.
\end{enumerate}
\end{lemma}

We continue with characterizations of the above-defined relations for tripotents in $C$. For the sake of completeness we include also the following characterizations of $\le$ and $\le_2$ which follows by combining  \cite[Proposition 6.3]{Finite} and Remark~\ref{rem:spin is finite}.

\begin{prop}\label{P:spin-le-le2}
Let $u,e\in C$ be two nonzero tripotents.
\begin{enumerate}[$(i)$]
    \item $u\le e$ $\Leftrightarrow$ either $u=e$ or $u$ is minimal and $e=u+\alpha\overline{u}$ for a complex unit $\alpha$;
        \item $u\le_2 e$ $\Leftrightarrow$ $u\le_{n,t}e$ $\Leftrightarrow$ either $e$ is unitary or $u=\alpha e$ for a complex unit $\alpha$.
\end{enumerate}
\end{prop}

To describe the other relations it will be suitable to distinguish three cases. It is the content of the following three propositions.

The first one deals with minimal tripotents. It follows easily from Proposition~\ref{P:orders special cases}$(a)$ and it is not specific for spin factors.

\begin{prop}\label{P:spin nonunitary tripotents} Let $u,e\in C$ be two nonzero tripotents such that $e$ is minimal. Then we have the following.
\begin{enumerate}[$(i)$]
     \item $u\le_{h,t} e$ $\Leftrightarrow$  $u\le_h e$ $\Leftrightarrow$ $u\le_r e$ $\Leftrightarrow$ $u=\pm e$;
    \item $u\le_2 e$ $\Leftrightarrow$ $u\le_{n,t} e$ $\Leftrightarrow$ $u \le_n e$ $\Leftrightarrow$ $u\le_{hc} e$ 
     $\Leftrightarrow$ $u\le_c e$ $\Leftrightarrow$ $u=\alpha e$ for a complex unit $\alpha$.
\end{enumerate}
\end{prop}

\begin{prop}\label{P:spin min-unitary}
 Let $u,e\in C$ be two nonzero tripotents such that $e$ is unitary and $u$ is not. Then we have the following.
 \begin{enumerate}[$(i)$]
     \item $u\le_h e$ $\Leftrightarrow$ $u\le_r e$ $\Leftrightarrow$ $e=\pm u+\alpha\overline{u}$ for a complex unit $\alpha$;
      \item $u\le_n e$ $\Leftrightarrow$ $u\le_{hc} e$ $\Leftrightarrow$ $u\le_c e$ $\Leftrightarrow$  $e\in\span\{u,\overline{u}\}$.
      \end{enumerate}
\end{prop}

\begin{proof}
$(ii)$  We have $e=\gamma z$, where $\gamma$ is a complex unit and $z\in H_r$ is a norm-one vector. Consider the Peirce decomposition of $z$ with respect to $u$, i.e.,
 $$z=\alpha u+\beta \overline{u}+x,$$
 where $\alpha,\beta\in\ce$ and $x\in\{u,\overline{u}\}^{\perp_2}$. We have
 $$z=\overline{z}=\overline{\alpha}\  \overline{u}+\overline{\beta} u+\overline{x},$$
 hence $\beta=\overline{\alpha}$ and $\overline{x}=x$, so $x\in H_r$. We have
 $$\J uue=\gamma\J uuz=\gamma(\alpha u+\frac12 x).$$
 
 Assume $u\le_n e$. Then $\J uue$ is a tripotent. We know from Lemma~\ref{L:len implies sim2} that $\J uue\sim_2 u$, hence $\J uue$ is a scalar multiple of $u$ by Proposition~\ref{P:spin-le-le2}$(ii)$.
 It follows that $x=0$. Thus $e\in\span\{u,\overline{u}\}$.
 
  Conversely, if $e\in\span\{u,\overline{u}\}$, then by the above we have $e=\gamma (\alpha u+\overline{\alpha} \ \overline{u})$. As $u\perp_2 \overline{u}$, $\ip ee=1$, $\ip uu=\ip{\overline{u}}{\overline{u}}=\frac12$, we deduce that $\alpha$ is a complex unit. Hence $\gamma \alpha u\le e$ by Proposition~\ref{P:spin-le-le2}$(i)$, so $u\le_c e$.
  
  The remaining implications are obvious.
  
  $(i)$  If $e=\pm u+\alpha\overline{u}$, by Proposition~\ref{P:spin-le-le2}$(i)$ we deduce that $u\le e$ or $-u\le e$, so $u\le_r e$.

 Conversely, assume $u\le_h e$. Since this implies $u\le_n e$, by the already proved $(ii)$ we have $e=\alpha u+\beta \overline{u}$ for some $\alpha,\beta\in \ce$. Further, we have
    $$\begin{aligned}u &=\J eue=2\ip eu e-\ip e{\overline{e}}\overline{u}\\&=2\ip{\alpha u+\beta \overline u}u (\alpha u+\beta \overline{u})-\ip{\alpha u+\beta \overline{u}}{\overline{\alpha}\  \overline{u}+\overline{\beta} u}\overline{u}
 \\&=2\alpha \ip uu (\alpha u+\beta \overline{u})-\alpha \beta (\ip uu+\ip{\overline{u}}{\overline{u}})\overline{u}
 =\alpha^2 u,\end{aligned}$$
 hence necessarily $\alpha =\pm1$. Finally, by Lemma~\ref{L:tripotents in spin factor} we have
 $$1=\norm{e}_2^2=\abs{\alpha}^2\norm{u}_2^2+\abs{\beta}^2\norm{\overline{u}}_2^2=\frac12(1+\abs{\beta}^2),$$
 hence 
 $\beta$ must be a complex unit.
 \end{proof}

\begin{prop}\label{P:spin-unitaries}
Let $u,e\in C$ be two unitary tripotents.
\begin{enumerate}[$(i)$]
    \item  $u\sim_h e$ if and only if
  either $u=\pm e$ or $\ip ue=0$ and $\ip{u}{\overline{u}}=-\ip{e}{\overline{e}}$. (The last condition is fulfilled if and only if
  $u=\alpha x$ and $e=\pm i\alpha y$, where  $x,y\in H_r$, $x\perp_2 y$ and $\alpha$ is a complex unit.) 
 
   \item  $u\sim_{hc} e$ if and only if
  either $u=\alpha e$ for a complex unit $\alpha$ or $u\perp_2 e$. 
  \item There is a unitary $v\in C$ such that
  $$ e\sim_{hc} v\sim_h u.$$
  In particular $e\sim_{hc,t}u$ (and chains of $\sim_{hc}$ of length two are enough).
   \item  $u\sim_{h,t} e$ if and only if $\ip{u}{\overline{u}}=\pm\ip{e}{\overline{e}}$. (This takes place if and only if there are $x,y\in H_r$ and a complex unit $\alpha$ such that $u=\alpha x$ and $e\in\{\pm \alpha y,\pm i\alpha y\}$.) Moreover, chains of $\sim_h$ of length three are enough.
\end{enumerate}
\end{prop}

\begin{proof}
$(i)$ Since $u,e$ are unitaries, we have $u=\alpha x$ and $e=\beta y$ for some complex units $\alpha,\beta$ and norm-one vectors $x,y\in H_r$. Recall that $u\sim_h e$ if and only if $u\le_h e$, which takes place if and only if $u=\J eue$.   The last equality means that 
$$\alpha x=\J {\beta y}{\alpha x}{\beta y}=\beta^2\overline{\alpha} \J yxy=\beta^2\overline{\alpha} (2\ip yx y-\ip yy x)=\beta^2\overline{\alpha} (2\ip yx y-x).$$
If $\ip yx=0$, this equality is equivalent to $\alpha=-\beta^2\overline{\alpha}$, i.e., $\alpha^2=-\beta^2$. 
This means that $\beta=\pm i\alpha$, hence the second case takes place. (Note that $\alpha^2=\ip{u}{\overline{u}}$ and
$\beta^2=\ip{e}{\overline{e}}$.)

If $\ip yx\ne0$, necessarily $y$ is a multiple of $x$. Since both $x,y\in H_r$, we get $y=\pm x$. In both cases the equality reduces to
$\alpha=\beta^2\overline{\alpha}$, thus $\alpha^2=\beta^2$, i.e., $\beta=\pm\alpha$. So, we deduce that this possibility is equivalent to $u=\pm e$.
 
$(ii)$ Recall that $u\sim_{hc} e$ if and only if there is a complex unit $\beta$ such that $u\sim_h \beta e$. Thus the equivalence follows easily from $(i)$
 
$(iii)$ Without loss of generality $u\in H_r$. By the assumptions there is a complex unit $\alpha$ such that $\alpha e\in H_r$ has real coordinates. Find $v\in H_r\cap\{u,\alpha e\}^{\perp_2}$ such that $\norm{v}_2=1$. (This is possible as $\dim H_r\ge 3$.) Then $(i)$ yields
$$\alpha e\sim_h iv\sim_h u,$$
which completes the proof. 

$(iv)$  Since $u,e$ are unitaries, we have $u=\alpha x$ and $e=\beta y$ for some complex units $\alpha,\beta$ and  $x,y\in H_r$. Observe that $\ip{u}{\overline{u}}=\alpha^2$ and $\ip{e}{\overline{e}}=\beta^2$. Hence the condition in brackets is clearly equivalent to $\ip{u}{\overline{u}}=\pm\ip{e}{\overline{e}}$.

Now let us prove the remaining equivalence.

The `only if' part follows easily from $(i)$. Let us prove the `if part'. 

Fix $x,y\in H_r$ such that $\norm{x}_2=\norm{y}_2=1$ and a complex unit $\alpha$ satisfying the hypotheses. Let $z_1\in H_r\cap\{x,y\}^{\perp_2}$ such that $\norm{z_1}_2=1$. Further, let $z_2\in H_r\cap\{z_1,y\}^{\perp_2}$ such that $\norm{z_2}_2=1$. The vectors $z_1,z_2$ exist since $\dim H_r\ge 3$.

Then it follows from $(i)$ that
$$\begin{gathered}\alpha x\sim_h i\alpha z_1\sim_h\pm\alpha y,\\
\alpha x\sim_h i\alpha z_1\sim_h \alpha z_2\sim_h\pm i\alpha y.
\end{gathered}$$
Now the assertion easily follows.
\end{proof}

\begin{cor}\label{cor:relations in spin}
\begin{enumerate}[$(i)$]
    \item The relations $\le_2$, $\le_{n,t}$ and $\le_{hc,t}$  coincide in $C$. 
    \item To describe the relation $\le_{hc,t}$ in $C$ chains of $\le_{hc}$ of length three are enough.
    \item To describe the relation $\le_{h,t}$ in $C$ chains of $\le_{h}$ of length four are enough.
    \item If $e\in C$ is unitary and $u\in C$ is minimal, then $u\le_{h,t} e$.
\end{enumerate}
\end{cor}

\begin{proof}
$(i)$ Assume that $u,e\in C$ are tripotents such that $u\le_2 e$.

If $u=0$, clearly $u\le_{hc}e$. So, assume $u\ne0$.

If $e$ is minimal, then Proposition~\ref{P:spin nonunitary tripotents}$(ii)$ implies that $u\le_{hc}e$.

Assume $e$ is unitary. Since $C$ is finite, there is $v\in C$ unitary with $u\le v$. By Proposition~\ref{P:spin-unitaries}$(iii)$ there is a unitary $w\in C$ such that
$$u\le v\sim_{hc} w\sim_h e.$$
Hence $u\le_{hc,t}e$ and, moreover, the chain of $\le_{hc}$ of length three is enough (cf. Proposition \ref{P:order implications}).

$(ii)$ This follows from the proof of $(i)$.

$(iii)$  Assume that $u,e\in C$ are tripotents such that $u\le_{h,t} e$.

If $u=0$, clearly $u\le_{h}e$. So, assume $u\ne0$.

If $e$ is minimal, then we deduce from Proposition~\ref{P:spin nonunitary tripotents}$(i)$ that $u\le_{h}e$.

So, assume $e$ is unitary. By Lemma~\ref{L:factor leht} we get a tripotent $v\in C$ such that 
$$u\le v\sim_{h,t} e.$$
Now we may conclude by Proposition~\ref{P:spin-unitaries}$(iv)$.

$(iv)$ Assume $e$ is unitary. Then $e=\alpha x$ for a complex unit $\alpha$ and a unit vector $x\in H_r$. Hence, $\ip{e}{\overline{e}}=\alpha^2$.

Further, set
$$v=u+\alpha^2\overline{u}.$$
Then $v$ is a unitary element such that $u\le v$ (cf. Proposition~\ref{P:spin-le-le2}$(i)$).

Then
$$\ip{v}{\overline{v}}=\ip{u+\alpha^2\overline{u}}{\overline{u}+\overline{\alpha}^2u}
=\alpha^2\ip{\overline{u}}{\overline{u}}+\alpha^2\ip{u}{u}=\alpha^2,$$
where we used equalities $\ip{u}{\overline{u}}=0$ and $\ip uu=\ip{\overline{u}}{\overline{u}}=\frac12$ provided by Lemma~\ref{L:tripotents in spin factor}$(b)$.

Using Proposition~\ref{P:spin-unitaries}$(iv)$ we see that $v\sim_{h,t}e$, hence $u\le_{h,t}e$.

\end{proof}

\begin{example2}\label{ex:spin}
Assume that $\dim H=3$. 
\begin{enumerate}[$(a)$]
    \item $(\frac12,\frac i2,0)$, $(\frac i2,-\frac 12,0)$ are two minimal tripotents such that
    $(\frac i2,-\frac 12,0)\sim_c(\frac12,\frac i2,0)$, but they are incomparable with respect to $\le_h$. Thus the cases $(i)$ and $(ii)$ from Proposition~\ref{P:spin nonunitary tripotents} are different.
    \item $e=(1,0,0)$ is a unitary tripotent. 
    
    $u_1=(-\frac12,\frac i2,0)$ is a minimal tripotent such that $u_1\le_r e$ but $u_1\not\le e$.
    
    $u_2=(\frac i2,\frac 12,0)$ is a minimal tripotent such that $u_2\le_n e$ but $u_2\not\le_h e$.
    
    $u_3=(0,\frac i2,\frac 12)$ is a minimal tripotent such that $u_3\not\le_n e$.
    
    It follows, in particular,  that the cases $(i)$ and $(ii)$ in Proposition~\ref{P:spin min-unitary} are different and $\le_2$ does not coincide with $\le_n$. Hence $\le_n$ is not transitive in $C$.
    
    \item Set $u=(1,0,0)$ and $e=(0,1,0)$. It follows from Proposition~\ref{P:spin-unitaries}$(iv)$ that $u\sim_{h,t}e$. However, Proposition~\ref{P:spin-unitaries}$(i)$ shows that $u$ and $e$ are incomparable with respect to $\le_h$. In particular, the relations $\le_h$ and $\sim_h$ are not transitive in $C$.
    
    \item Set $u=(1,0,0)$ and $e=(\frac1{\sqrt{2}},\frac{1}{\sqrt2},0)$. It follows from Proposition~\ref{P:spin-unitaries}$(iv)$ that $u\sim_{h,t}e$. However, Proposition~\ref{P:spin-unitaries}$(ii)$ shows that $u$ and $e$ are incomparable with respect to $\le_{hc}$. In particular, the relations $\le_{hc}$ and $\sim_{hc}$ are not transitive in $C$.
    
\end{enumerate}
\end{example2}

Now let us focus on triples of the form $L^\infty(\mu,C)$ where $\mu$ is a probability measure satisfying \eqref{eq:CK=Linfty}.
Recall that in such a case we have $L^\infty(\mu)=C(\Omega)$. If $\dim C<\infty$, then $L^\infty(\mu,C)=C(\Omega,C)$ as well (by Lemma~\ref{L:LinftyC=CKC}). However, $C$ may have infinite dimension and then we do not have this equality. The following lemma will help us to overcome this small inconvenience.

\begin{lemma}\label{L:spin partition}
Let $\uu\in L^\infty(\mu,C)$ be a tripotent. Then there is a unique decomposition 
$$\Omega=U_{\uu}\cup M_{\uu}\cup Z_{\uu}$$
of $\Omega$ into three clopen sets, such that
\begin{enumerate}[$(i)$]
    \item $\uu(\omega)$ is unitary $\mu$-almost everyhere on $U_{\uu}$;
    \item $\uu(\omega)$ is a minimal tripotent $\mu$-almost everyhere on $M_{\uu}$;
    \item $\uu(\omega)=0$  $\mu$-almost everyhere on $Z_{\uu}$.
\end{enumerate}
\end{lemma}

\begin{proof}
Set
$$\begin{aligned}
U^0_{\uu}&=\{\omega\in \Omega\setsep \norm{\uu(\omega)}_2=1\},\\
M^0_{\uu}&=\{\omega\in \Omega\setsep \norm{\uu(\omega)}_2=\frac1{\sqrt{2}}\},\\
Z^0_{\uu}&=\{\omega\in \Omega\setsep \uu(\omega)=0\}.
\end{aligned}$$
These sets are disjoint, Borel measurable and cover $\Omega$ up to a set of $\mu$-measure zero. Let $U_{\uu}$ be the clopen set which differs from $U^0_{\uu}$ only by a set of $\mu$-measure zero. The existence of $U_{\uu}$ follows from \eqref{eq:CK=Linfty}. The uniqueness is clear -- if we have two such clopen sets, their symmetric difference is a clopen set of zero measure, hence empty. 

Similarly we define $M_{\uu}$ and $Z_{\uu}$. The resulting three clopen sets are pairwise disjoint as the intersection of any two of them is a clopen set of measure zero. Further, they cover $\Omega$ as their union is a clopen set with full measure.
\end{proof}

We continue by characterizing the relations in $L^\infty(\mu,C)$. The first step is the following proposition which collects descriptions of those relations which may be easily characterized pointwise (almost everywhere). We also provide characterizations using the structure of $C$.

\begin{prop}\label{P:Aotimesspin relations} Let $M=L^\infty(\mu,C)$. Assume that $\uu,\e\in M$ are two tripotents. Let $M_{\uu},U_{\uu},M_{\e},U_{\e}$ be the sets provided by Lemma~\ref{L:spin partition}. Then we have the following:
\begin{enumerate}[$(a)$]
    \item $\uu\le_{n,t}\e$ $\Leftrightarrow$ $\uu\le_2\e$ $\Leftrightarrow$ $\uu(\omega)\le_2\e(\omega)$ $\mu$-a.e. $\Leftrightarrow$ $U_{\uu}\subset U_{\e}$, $M_{\uu}\subset U_{\e}\cup M_{\e}$ and there is $f\in C(M_{\uu}\cap M_{\e},\TT)$ such that $\uu(\omega)=f(\omega)\e(\omega)$ $\mu$-a.e. on $M_{\uu}\cap M_{\e}$.
    \item $\uu\sim_2\e$ $\Leftrightarrow$ $\uu(\omega)\sim_2\e(\omega)$ $\mu$-a.e. $\Leftrightarrow$ $U_{\uu}= U_{\e}$, $M_{\uu}=M_{\e}$ and there is $f\in C(M_{\uu},\TT)$ such that $\uu(\omega)=f(\omega)\e(\omega)$ $\mu$-a.e. on $M_{\uu}$.
    \item $\uu\le\e$ $\Leftrightarrow$ $\uu(\omega)\le\e(\omega)$ $\mu$-a.e. $\Leftrightarrow$ $U_{\uu}\subset U_{\e}$, $M_{\uu}\subset U_{\e}\cup M_{\e}$, $\uu(\omega)=\e(\omega)$ $\mu$-a.e. on $U_{\uu}\cup (M_{\uu}\cap M_{\e})$ and there is $f\in C(M_{\uu}\cap U_{\e},\TT)$ such that $\e(\omega)=\uu(\omega)+f(\omega)\overline{\uu(\omega)}$ $\mu$-a.e. on $M_{\uu}\cap U_{\e}$.
    \item $\uu\le_h\e$ $\Leftrightarrow$ $\uu(\omega)\le_h\e(\omega)$ $\mu$-a.e. $\Leftrightarrow$ 
    \begin{itemize}
        \item $U_{\uu}\subset U_{\e}$ and $M_{\uu}\subset U_{\e}\cup M_{\e}$;
        \item there is a clopen subset $A\subset M_{\uu}\cap M_{\e}$ such that $$\uu(\omega)=(\chi_A(\omega)-\chi_{M_{\uu}\cap M_{\e}\setminus A}(\omega))\e(\omega)\mbox{  $\mu$-a.e. on } M_{\uu}\cap M_{\e};$$
        \item there are a clopen subset $B\subset M_{\uu}\cap U_{\e}$ and a function $f\in C(M_{\uu}\cap U_{\e},\TT)$ such that
        $$\e(\omega)=(\chi_B(\omega)-\chi_{M_{\uu}\cap U_{\e}\setminus B}(\omega))\uu(\omega)+f(\omega)\overline{\uu(\omega)}\mbox{ $\mu$-a.e. on }  M_{\uu}\cap U_{\e};$$
        \item There are disjoint clopen subsets $D,E\subset U _{\uu}$ such that $\uu(\omega)=\e(\omega)$ $\mu$-a.e. on $D$, $\uu(\omega)=-\e(\omega)$ $\mu$-a.e. on $E$ and $\ip{\uu(\omega)}{\e(\omega)}=0$ and $\ip{\uu(\omega)}{\overline{\uu(\omega)}}=-\ip{\e(\omega)}{\overline{\e(\omega)}}$ $\mu$-a.e. on $U_{\uu}\setminus (D\cup E)$.
    \end{itemize}
     \item $\uu\sim_h\e$ $\Leftrightarrow$ $\uu(\omega)\sim_h\e(\omega)$ $\mu$-a.e.  $\Leftrightarrow$ 
    \begin{itemize}
        \item $U_{\uu}=U_{\e}$ and $M_{\uu}=M_{\e}$;
        \item there is a clopen subset $A\subset M_{\uu}$ such that $$\uu(\omega)=(\chi_A(\omega)-\chi_{M_{\uu}\setminus A}(\omega))\e(\omega)\mbox{ $\mu$-a.e. on }  M_{\uu};$$
         \item There are disjoint clopen subsets $D,E\subset M_{\uu}$ such that $\uu(\omega)=\e(\omega)$ $\mu$-a.e. on $D$, $\uu(\omega)=-\e(\omega)$ $\mu$-a.e. on $E$ and $\ip{\uu(\omega)}{\e(\omega)}=0$ and $\ip{\uu(\omega)}{\overline{\uu(\omega)}}=-\ip{\e(\omega)}{\overline{\e(\omega)}}$ $\mu$-a.e. on $U_{\uu}\setminus (D\cup E)$.
    \end{itemize}
     \item $\uu\le_n\e$ $\Leftrightarrow$ $\uu(\omega)\le_n\e(\omega)$ $\mu$-a.e.  $\Leftrightarrow$ 
    \begin{itemize}
        \item $U_{\uu}\subset U_{\e}$ and $M_{\uu}\subset U_{\e}\cup M_{\e}$;
        \item there is $f\in C(M_{\uu}\cap M_{\e},\TT)$ such that $\uu(\omega)=f(\omega)\e(\omega)$ $\mu$-a.e. on $M_{\uu}\cap M_{\e}$;
        \item there are $g,h\in C(M_{\uu}\cap U_{\e},\TT)$ such that $\e(\omega)=g(\omega)\uu(\omega)+h(\omega)\overline{\uu(\omega)}$ $\mu$-a.e. on $M_{\uu}\cap U_{\e}$.
    \end{itemize}
    
\end{enumerate}
\end{prop}

\begin{proof}
The first equivalence in assertion $(a)$ follows from Remark~\ref{rem:spin is finite}. The second equivalence in assertion $(a)$ and the first equivalences in assertions $(b)$--$(f)$ follow from an obvious analogue of Proposition~\ref{P:C0(T,E)}$(b)$.

The remaining equivalences in assertions $(a)$--$(f)$ follow essentially by combinining Propositions~\ref{P:spin-le-le2}, \ref{P:spin nonunitary tripotents},
\ref{P:spin min-unitary} and~\ref{P:spin-unitaries}. More precisely, the quoted propositions show the equivalence with a formally weaker condition -- without requiring continuity of the respective functions and clopeness of the respective sets. So, it is enough to observe that the functions are continuous after modifying on a set of measure zero and those sets are clopen after taking a symmetric difference with a set of measure zero. It is so, because the following explicit formulae provide Borel measurable functions and Borel sets and we then use assumption \eqref{eq:CK=Linfty}:

\begin{enumerate}[$(a)$]
    \item $f(\omega)=2\ip{\uu(\omega)}{\e(\omega)}$ for $\omega\in M_{\uu}\cap M_{\e}$;
    \item $f(\omega)=2\ip{\uu(\omega)}{\e(\omega)}$ for $\omega\in M_{\uu}$;
    \item $f(\omega)=2\ip{\e(\omega)}{\overline{\uu(\omega)}}$ for $\omega\in M_{\uu}\cap U_{\e}$;
    \item We have
    $$\begin{aligned}
    A&=\{\omega\in M_{\uu}\cap M_{\e}\setsep \uu(\omega)=\e(\omega)\},\\
    (M_{\uu}\cap M_{\e})\setminus A&=\{\omega\in M_{\uu}\cap M_{\e}\setsep \uu(\omega)=-\e(\omega)\},
        \end{aligned}$$
   and
    $$\begin{aligned}
    f(\omega)&=2\ip{\e(\omega)}{\overline{\uu(\omega)}},\quad \omega\in M_{\uu}\cap U_{\e},
    \\
    B&=\left\{\omega\in M_{\uu}\cap U_{\e}\setsep \ip{\e(\omega)}{\uu(\omega)}=\frac12\right\},\\
    (M_{\uu}\cap U_{\e})\setminus B&=\left\{\omega\in M_{\uu}\cap U_{\e}\setsep \ip{\e(\omega)}{\uu(\omega)}=-\frac12\right\}.
    \end{aligned}$$
    
    The definitions of the clopen sets $E$ and $D$ are clear.
    
    \item The formulas are analogous as in $(d)$.
    
    \item The formulas are:
    $$\begin{aligned}
    f(\omega)&=2\ip{\uu(\omega)}{\e(\omega)},
    \quad \mbox{ for } \omega\in M_{\uu}\cap M_{\e},\\
    g(\omega)&=2\ip{\e(\omega)}{\uu(\omega)},
    \quad \mbox{ for } \omega\in M_{\uu}\cap U_{\e},\\
    h(\omega)&=2\ip{\e(\omega)}{\overline{\uu(\omega)}},
    \quad \mbox{ for } \omega\in M_{\uu}\cap U_{\e}.
    \end{aligned}$$
\end{enumerate}
\end{proof}

We continue by looking at relations $\sim_{h,t}$ and $\le_{h,t}$. If $\dim C<\infty$ we may proceed similarly as for symmetric and antisymmetric matrices using Lemmata~\ref{L:Cartan parametrization} and~\ref{L:cartan diagonalization}. But $C$ may be infinite-dimensional or even non-separable, so we cannot use the Kuratowski-Ryll-Nardzewski selection theorem. Fortunately, the structure of $C$ permits to provide explicit formulas for certain mappings.
 
We remark that Lemma~\ref{L:frames} holds also for spin factors. It follows from Lemma~\ref{L:tripotents in spin factor} and Proposition~\ref{P:spin-le-le2}$(i)$ that $C$ has rank $2$ and any frame is of the form $u,\alpha\overline{u}$ where $u$ is a minimal tripotent and $\alpha$ is a complex unit. If $u,\alpha\overline{u}$ and $v,\beta\overline{v}$ are two such frames, a routine computation shows that there is a surjective isometry
$T:H_r\to H_r$ and a complex unit $\gamma$ such that the operator
$$x+i y\mapsto \gamma(T(x)+iT(y))$$
is a triple automorphism of $C$ mapping $u$ to $v$ and $\alpha\overline{u}$ to $\beta\overline{v}$. 

However, we do not wish to work with mappings with values in the nonseparable space of operators on $C$. We rather give a direct proof of a parametrized versions of Proposition~\ref{P:spin-unitaries}$(iv)$ and Corollary~\ref{cor:relations in spin}. This is done in the following proposition.

\begin{prop}\label{P:spin vector leht} Let $M=L^\infty(\mu,C)$. Assume that $\uu,\e\in M$ are two tripotents. Let $M_{\uu},U_{\uu},M_{\e},U_{\e}$ be the sets provided by Lemma~\ref{L:spin partition}. Then we have the following:

\begin{enumerate}[$(a)$]
  \item $\uu\sim_{h,t}\e$ $\Leftrightarrow$ $\uu(\omega)\sim_{h,t}\e(\omega)$ $\mu$-a.e. $\Leftrightarrow$ 
    \begin{itemize}
        \item $U_{\uu}=U_{\e}$ and $M_{\uu}=M_{\e}$;
        \item There is a clopen subset $A\subset M_{\uu}$ such that $$\uu(\omega)=(\chi_A(\omega)-\chi_{M_{\uu}\setminus A}(\omega))\e(\omega)\mbox{ $\mu$-a.e. on } M_{\uu};$$
         \item There is a clopen subset $B\subset U_{\uu}$ such that $$\ip{\uu(\omega)}{\overline{\uu(\omega)}}=(\chi_B(\omega)-\chi_{U_{\uu}\setminus B}(\omega))\ip{\e(\omega)}{\overline{\e(\omega)}} \mbox{ $\mu$-a.e. on } U_{\uu}.$$
    \end{itemize}
    Moreover, chains of $\sim_h$ of length three are enough.
    \item $\uu\le_{h,t}\e$ $\Leftrightarrow$ $\uu(\omega)\le_{h,t}\e(\omega)$ $\mu$-a.e.  $\Leftrightarrow$ 
    \begin{itemize}
        \item $U_{\uu}\subset U_{\e}$ and $M_{\uu}\subset M_{\e}\cup U_{\e}$;
        \item There is a clopen subset $A\subset M_{\uu}\cap M_{\e}$ such that $$\uu(\omega)=(\chi_A(\omega)-\chi_{M_{\uu}\cap M_{\e}\setminus A}(\omega))\e(\omega)\mbox{ $\mu$-a.e. on }  M_{\uu}\cap M_{\e};$$
         \item There is a clopen subset $B\subset U_{\uu}$ such that $$\ip{\uu(\omega)}{\overline{\uu(\omega)}}=(\chi_B(\omega)-\chi_{U_{\uu}\setminus B}(\omega))\ip{\e(\omega)}{\overline{\e(\omega)}} \mbox{ $\mu$-a.e. on }U_{\uu}.$$
    \end{itemize} 
    
    Moreover, chains of $\le_h$ of length four are enough.
        
\end{enumerate}
\end{prop}

\begin{proof} $(a)$ The implication `$\Rightarrow$' from the first equivalence is obvious (by an analogue of Proposition~\ref{P:C0(T,E)}). The implication `$\Rightarrow$' from the second equivalence follows by combining Propositions~\ref{P:spin nonunitary tripotents}$(i)$ and~\ref{P:spin-unitaries}$(iv)$ if we additionally observe that the sets $A,B$ may be clopen. But this follows again using \eqref{eq:CK=Linfty} from the formulas
$$\begin{aligned}
A&=\{\omega\in M_{\uu}\setsep \uu(\omega)=\e(\omega)\},\\
    M_{\uu}\setminus A&=\{\omega\in M_{\uu}\setsep \uu(\omega)=-\e(\omega)\},
    \\B&=\left\{\omega\in U_{\uu}\setsep \ip{\uu(\omega)}{\overline{\uu(\omega)}}=\ip{\e(\omega)}{\overline{\e(\omega)}}\right\},\\
   U_{\uu}\setminus B&=\left\{\omega\in U_{\uu}\setsep \ip{\uu(\omega)}{\overline{\uu(\omega)}}=-\ip{\e(\omega)}{\overline{\e(\omega)}}\right\}.
        \end{aligned}$$

It remains to prove that the third condition implies $\uu\sim_{h,t}\e$. So, assume that the third condition is fulfilled. We will show that $\uu\sim_{h,t}\e$ using a parametrized version of the proof of Proposition~\ref{P:spin-unitaries}$(iv)$.

We will define several Borel measurable functions on $U_{\uu}$. Firstly, the function
$$\alpha_{\uu}(\omega)=\ip{\uu(\omega)}{\overline{\uu(\omega)}},\quad \omega\in U_{\uu},$$
is Borel measurable and its values are complex units ($\mu$-a.e.).
It follows that there is a Borel-measurable function
$\beta_{\uu}:U_{\uu}\to\TT$ such that
$$\alpha_{\uu}(\omega)=\beta_{\uu}(\omega)^2 \mbox{ $\mu$-a.e. on } U_{\uu}.$$

Further, the function $x_{\uu}: U_{\uu}\to C$ defined by
$$x_{\uu}(\omega)=\overline{\beta_{\uu}(\omega)}\uu(\omega), \quad \omega\in U_{\uu},$$
is Borel measurable as well. We observe that it has values in $H_r$  ($\mu$-a.e.) as
$$\ip{x_{\uu}(\omega)}{\overline{x_{\uu}(\omega)}}=
\ip{\overline{\beta_{\uu}(\omega)}\uu(\omega)}{\beta_{\uu}(\omega)\overline{\uu(\omega)}}=\overline{\beta_{\uu}(\omega)}^2\alpha_{\uu}(\omega)=1 \ \mu\mbox{-a.e.}$$

Similarly, using $\e$ instead of $\uu$, we define functions $\alpha_{\e}$, $\beta_{\e}$ and $x_{\e}$.

Fix $z_1,z_2,z_3\in H_r$ three mutually orthogonal unit vectors. 

Next we will define several functions on some cartesian powers of the unit sphere of $H_r$ as follows:

$$\begin{aligned}
\psi_1(x)& =\begin{cases} \frac{z_1-\ip {z_1}x x}{\norm{z_1-\ip{z_1}{x}}_2}, & z_1\ne \ip {z_1}x x,\\ 
\frac{z_2-\ip {z_2}x x}{\norm{z_2-\ip{z_2}{x}}_2},
& z_1=\ip {z_1}x x,
\end{cases}\mbox{\quad for }x\in S_{H_r};
\\
\phi(x_1,x_2)&=\frac{x_2-\ip {x_2}{x_1}x_1} {\norm{x_2-\ip {x_2}{x_1}x_1}_2},\quad (x_1,x_2)\in S_{H_r}^2, x_1\ne\pm x_2,
\\
\psi_2(x_1,x_2)&=\begin{cases} 
\psi_1(x_1), & x_1=\pm x_2,\\
 & \\
 \frac{z_1-\ip {z_1}{x_1} x_1-\ip {z_1}{\phi(x_1,x_2)} \phi(x_1,x_2)}{\norm{z_1-\ip {z_1}{x_1} x_1-\ip {z_1}{\phi(x_1,x_2)} \phi(x_1,x_2)}_2}, &\small \begin{cases}
  x_1\ne\pm x_2, \hbox{ and }\\
 z_1\ne \ip {z_1}{x_1} x_1 \\ \quad
 +\ip {z_1}{\phi(x_1,x_2)}\phi(x_1,x_2),
 \end{cases} \\
  & \\
 \frac{z_2-\ip {z_2}{x_1} x_1-\ip {z_2}{\phi(x_1,x_2)} \phi(x_1,x_2)}{\norm{z_2-\ip {z_2}{x_1} x_1-\ip {z_2}{\phi(x_1,x_2)} \phi(x_1,x_2)}_2}, & \small
 \begin{cases} x_1\ne\pm x_2,\\
 z_1= \ip {z_1}{x_1} x_1\\
\quad+ \ip {z_1}{\phi(x_1,x_2)}\phi(x_1,x_2), \\
 z_2\ne \ip {z_2}{x_1} x_1\\\quad + \ip {z_2}{\phi(x_1,x_2)}\phi(x_1,x_2),
 \end{cases}
 %& 
 \\
 \frac{z_3-\ip {z_3}{x_1} x_1-\ip {z_3}{\phi(x_1,x_2)} \phi(x_1,x_2)}{\norm{z_3-\ip {z_3}{x_1} x_1-\ip {z_3}{\phi(x_1,x_2)} \phi(x_1,x_2)}_2}, & \small
 \begin{cases}
 x_1\ne\pm x_2,\\
 z_1= \ip {z_1}{x_1} x_1 \\\quad +\ip {z_1}{\phi(x_1,x_2)}\phi(x_1,x_2), \\ 
 z_2= \ip {z_2}{x_1} x_1 \\\quad +\ip {z_2}{\phi(x_1,x_2)}\phi(x_1,x_2).
 \end{cases}
\end{cases}
\end{aligned}$$
Then $\psi_2:S_{H_r}^2\to S_{H_r}$ is a Borel measurable mapping such that $\psi_2(x_1,x_2)$ lies in $\{x_1,x_2\}^{\perp_2}$ for any $x_1,x_2\in S_{H_r}$.
 
Now we proceed to a parametrized version of the procedure from Proposition~\ref{P:spin-unitaries}$(iv)$.

We define Borel measurable mappings $\vv_1$ and $\vv_2$ as follows:
$$\begin{aligned}
\vv_1(\omega)&=\begin{cases}
\uu(\omega), &\omega\in M_{\uu},\\
i\beta_{\uu}(\omega)\psi_2(x_{\uu}(\omega),x_{\e}(\omega)), & \omega\in U_{\uu},
\end{cases}\\ 
\vv_2(\omega)&=\begin{cases}
\uu(\omega), &\omega\in M_{\uu},\\
\vv_1(\omega), & \omega\in B,\\
\beta_{\uu}(\omega)\psi_2(\vv_1(\omega),x_{\e}(\omega)), & \omega\in U_{\uu}\setminus B.
\end{cases}
\end{aligned}$$
Then $\vv_1,\vv_2$ are indeed Borel measurable mappings (with separable ranges) whose values are tripotents in $C$. Moreover, it follows from Proposition~\ref{P:spin-unitaries} that for each $\omega\in\Omega$ we have
$$\uu(\omega)\sim_h\vv_1(\omega)\sim_h\vv_2(\omega)\sim_h\e(\omega).$$
 Thus $\uu\sim_{h,t}\e$ and there is a chain of $\sim_h$ of length three witnessing it. This completes the proof.

$(b)$  The implication `$\Rightarrow$' from the first equivalence is again obvious (by an analogue of Proposition~\ref{P:C0(T,E)}). The implication `$\Rightarrow$' from the second equivalence follows by combining Propositions~\ref{P:spin nonunitary tripotents}$(i)$ and~\ref{P:spin-unitaries}$(iv)$ if we additionally observe that the sets $A,B$ may be clopen which may be done by similar formulae as in the proof of $(a)$. 

Assume the third condition holds.
Set
$$\vv(\omega)=\begin{cases} \uu(\omega), & \omega\in U_{\uu}\cup (M_{\e}\cap M_{\uu}) \\
\e(\omega), & \omega\in\Omega\setminus (M_{\uu}\cup U_{\uu}),\\
\uu(\omega)+\ip{\e(\omega)}{\overline{\e(\omega)}}\overline{\uu}(\omega),& \omega\in M_{\uu}\cap U_{\e}.
\end{cases}$$
Clearly, $\vv$ is a Borel measurable mapping with separable range. Moreover, its values are tripotents and $\uu(\omega)\le\vv(\omega)\sim_{h,t}\e(\omega)$ for $\omega\in\Omega$ (on $M_{\uu}\cap U_{\uu}$ we use the proof of Corollary~\ref{cor:relations in spin}$(iv)$). Thus $\uu\le\vv$. Moreover, assertion $(a)$ implies $\vv\sim_{h,t}\e$.  Thus $\uu\le_{h,t}\e$ and, moreover, 
 using again assertion $(a)$ we deduce that chains of $\le_{h}$ of length $4$ are enough.
\end{proof}

\subsection{Type 5 Cartan factor}

In this subsection we investigate the above-defined relations in JBW$^*$-triples of the form $A\overline{\otimes}C_5$, where $A$ is an abelian von Neumann algebra and $C_5$ is the Cartan factor of type 5.

We start by the following remark.

\begin{remark}\label{rem:c5 is finite}
By \cite[Proposition 6.10]{Finite} we know that $A\overline{\otimes}C_5$ is a finite JBW$^*$-triple. Hence, by Proposition~\ref{P:le2=lent in finite} the relations $\le_2$ and $\le_{n,t}$ coincide.
\end{remark}

Recall that $C_5$ may be represented as the space of $1\times 2$ matrices whose entries are complex Cayley numbers. Further, the algebra of complex Cayley numbers is the eight-dimensional spin factor with an additional algebraic structure. Thus $\dim C_5=16$. We will not use details of the algebraic structure, which is described for example in \cite[Section 6.4]{Finite}. We will use basic facts on tripotents in $C_5$  collected in the following lemma which follows from \cite[Proposition 6.11 and the subsequent remarks]{Finite}.

\begin{lemma}\label{L:tripotents in C5}
$C_5$ contains no unitary elements. Any nonzero tripotent in $C_5$ is either complete or minimal (in particular, $C_5$ is of rank $2$). 

Moreover, if $e\in C_5$ is a complete tripotent, then both $(C_5)_2(e)$ and $(C_5)_1(e)$ are triple-isomorphic to the eight-dimensional spin factor.
\end{lemma}

It follows that we may apply the results from the previous subsection. Let us summarize the consequences for tripotents in $C_5$.

\begin{prop}\label{P:c5-relations} Let $u,e\in C_5$ be two nonzero tripotents.
\begin{enumerate}[$(a)$]
    \item If $e$ is minimal, then the equivalences from Proposition~\ref{P:spin nonunitary tripotents} are valid.
    \item Assume that both $e$ and $u$ are complete. Then
    \begin{enumerate}[$(i)$]
        \item $e\sim_2 u\Leftrightarrow e\sim_{hc,t} u$. Moreover, chains of $\sim_{hc}$ of length two are enough to describe $\sim_{hc,t}$.
        \item To describe $\sim_{h,t}$ the chains of $\sim_h$ of length three are enough.
    \end{enumerate}
    \item Assume that $e$ is complete and $u$ is minimal. Then the following assertions hold:
    \begin{enumerate}[$(i)$]
        \item $u\le_h e\Leftrightarrow u\le_r e$;
        \item $u\le_n e\Leftrightarrow u\le_{hc}e\Leftrightarrow u\le_c e$;
        \item $u\le_2 e\Leftrightarrow u\le_{n,t} e\Leftrightarrow u\le_{hc,t}e\Leftrightarrow u\le_{h,t}e$.
        
        Moreover, to describe $\le_{h,t}$ the chains of $\le_h$ of length four are enough.
    \end{enumerate}
\end{enumerate}
\end{prop}

\begin{proof}
Assertion $(a)$ is obvious as Proposition~\ref{P:spin nonunitary tripotents} is not specific for spin factors but holds for minimal tripotents in any JBW$^*$-triple.

$(b)$ This follows by combinining Lemma~\ref{L:tripotents in C5} with Proposition~\ref{P:spin-unitaries}.

$(c)$ By combining Lemma~\ref{L:tripotents in C5} with Proposition~\ref{P:spin min-unitary} we get assertions $(i)$ and $(ii)$. Assertion $(iii)$ then follows using moreover Corollary~\ref{cor:relations in spin}$(iii)$, $(iv)$.
\end{proof}

\begin{cor}\label{cor:relations in C5}
\begin{enumerate}[$(i)$]
    \item The relations $\le_2$, $\le_{n,t}$ and $\le_{hc,t}$  coincide in $C_5$. 
    \item To describe the relation $\le_{hc,t}$ in $C_5$ chains of $\le_{hc}$ of length three are enough.
    \item To describe the relation $\le_{h,t}$ in $C_5$ chains of $\le_{h}$ of length four are enough.
 \end{enumerate}
\end{cor}

\begin{remark}
As $C_5$ contains the eight-dimensional spin factor as a subtriple, Example~\ref{ex:spin} may be applied for $C_5$ as well.
\end{remark}

Now we are going to look at the triples of the form $L^\infty(\mu,C_5)$ where $\mu$ is a probability measure satisfying \eqref{eq:CK=Linfty}.

\begin{prop}\label{P:relations is Aotimesc5}
Let $M=L^\infty(\mu,C_5)=C(\Omega,C_5)$ (where $\mu$ satisfies \eqref{eq:CK=Linfty}). Let $\uu,\e$ be two tripotents in $M$. 
\begin{enumerate}[$(i)$]
    \item If 
$$R\in\{\le,\le_h,\le_n,\le_{h,t},\le_{n,t},\le_2,\sim_h,\sim_{h,t},\sim_2\},$$
then
$$\uu R\e \Leftrightarrow \forall\omega\in\Omega\colon \uu(\omega) R\e(\omega).$$
\item $\uu\le_{n,t}\e \Leftrightarrow \uu\le_2\e$;
\item To describe $\sim_{h,t}$ chains of $\sim_h$ of length three are enough.
\item To describe $\le_{h,t}$ chains of $\le_h$ of length three are enough.
\end{enumerate}
\end{prop}

\begin{proof}
Assertion $(i)$ for $R\in \{\le,\le_h,\sim_h,\le_n,\le_2,\sim_2\}$ follows from Proposition~\ref{P:C0(T,E)}$(b)$.

By Remark~\ref{rem:c5 is finite} we deduce assertion $(ii)$ and hence the validity of $(i)$ for $R=\le_{n,t}$.

It remains to prove $(iii),(iv)$ and the validity of $(i)$ for $R\in\{\le_{h,t},\sim_{h,t}\}$. The `only if' part of $(i)$ is clear from previous results (cf. Proposition~\ref{P:C0(T,E)}$(a)$).

Assume $\uu(\omega)\le_{h,t}\e(\omega)$ for $\omega\in\omega$. 
Fix $u_1,u_2\in C_5$ two minimal orthogonal tripotents (forming hence a frame in $C_5$) and set $e=u_1+u_2$. Then $e$ is a complete tripotent. 

Apply Lemma~\ref{L:Cartan vector parametrization} to $\e$ and $u_1,e$ and let $\Psi$ be the resulting mapping. Then $\Psi$ is a triple automorphism of $M$, $\Psi(\e)$ attains only values $0,u_1,e$ and $\Psi(\uu)(\omega)\le_{h,t}\Psi(\e)(\omega)$ for each $\omega\in\Omega$.

Hence, both $\Psi(\uu)$ and $\Psi(\e)$ have values in $(C_5)_2(e)$, i.e., $\Psi(\uu),\Psi(\e)$ belong to  $C(\Omega,(C_5)_2(e))$. Since $(C_5)_2(e)$ is triple isomorphic to the eight-dimensional spin factor (by Lemma~\ref{L:tripotents in C5}), we may conclude by Proposition~\ref{P:spin vector leht}.
\end{proof}
 
\subsection{Type 6 Cartan factor}\label{subsec:c6}

In this subsection we investigate the above-defined relations in JBW$^*$-triples of the form $A\overline{\otimes}C_6$, where $A$ is an abelian von Neumann algebra and $C_6$ is the Cartan factor of type 6.

We start by the following remark.

\begin{remark}\label{rem:c6 is finite}
By \cite[Proposition 6.8]{Finite} we know that $A\overline{\otimes}C_6$ is a finite JBW$^*$-algebra. Hence, by Proposition~\ref{P:le2=lent in finite} the relations $\le_2$ and $\le_{n,t}$ coincide.
\end{remark}

Recall that $C_6$ may be represented as the space of `hermitian' $3\times 3$ matrices whose entries are complex Cayley numbers. This structure  is described for example in \cite[Section 6.3]{Finite}. In the following lemma we collect basic facts on tripotents in $C_6$ which follow from \cite[Remark 6.9]{Finite}.

\begin{lemma}\label{L:tripotents in C6}
In $C_6$ there are three types of nonzero tripotents -- unitary tripotents, minimal tripotens and rank-two tripotents.  

Moreover, if $u\in C_6$ is a rank-two tripotent, then $(C_6)_2(u)$ is triple-isomorphic to the ten-dimensional spin factor.
\end{lemma}

Hence we get the following proposition which is a complete analogue of Proposition~\ref{P:c5-relations}.

\begin{prop}\label{P:c6-relations-easy} Let $u,e\in C_6$ be two nonzero tripotents.
\begin{enumerate}[$(a)$]
    \item If $e$ is minimal, then the equivalences from Proposition~\ref{P:spin nonunitary tripotents} are valid.
    \item Assume that both $e$ and $u$ have rank two. Then
    \begin{enumerate}[$(i)$]
        \item $e\sim_2 u\Leftrightarrow e\sim_{hc,t} u$. Moreover, chains of $\sim_{hc}$ of length two are enough.
        \item To describe $\sim_{h,t}$ the chains of $\sim_h$ of length three are enough.
    \end{enumerate}
    \item Assume that $e$ has rank two and $u$ is minimal. Then the following assertions hold:
    \begin{enumerate}[$(i)$]
        \item $u\le_h e\Leftrightarrow u\le_r e$;
        \item $u\le_n e\Leftarrow u\le_{hc}e\Leftrightarrow u\le_c e$;
        \item $u\le_2 e\Leftrightarrow u\le_{n,t} e\Leftarrow u\le_{hc,t}e\Leftrightarrow u\le_{h,t}e$.
        
        Moreover, to describe $\le_{h,t}$ the chains of $\le_h$ of length four are enough.
    \end{enumerate}
\end{enumerate}
\end{prop}
 
 To describe properties of the relations between unitary elements in $C_6$ we will use the notion of determinant from \cite{cartan6}. If $u\in C_6$ is unitary, the spectral decomposition theorem in this finite-dimensional JBW$^*$-algebra implies that $u=\alpha_1 p_1+\alpha_2 p_2+\alpha_3 p_3$ where $p_1,p_2,p_3$ are mutually orthogonal minimal projections and $\alpha_1,\alpha_2,\alpha_3$ are complex units (cf. \cite[Theorem 4.1]{cartan6}). Following \cite[Section 4]{cartan6} we set in this case $\dt u=\alpha_1\alpha_2\alpha_3$ and call this quantity the \emph{determinant} of $u$.
 
 \begin{prop}\label{P:C6 unitaries}
 Let $u,e\in C_6$ be two unitary elements.
 \begin{enumerate}[$(i)$]
      \item $u\sim_{h,t} e\Leftrightarrow \dt u=\pm\dt e$. Moreover, the respective chain of $\sim_h$ has length at most $5$.
     \item Always $u\sim_{hc,t} e$.  Moreover, the respective chain of $\sim_{hc}$ has length at most $5$.
     \end{enumerate}
 \end{prop}

 \begin{proof}
 $(i)$ It follows from \cite[Corollary 4.3]{cartan6} that $\dt u=\pm\dt e$ whenever $u\sim_h e$. Hence, an obvious inductive argument proves the implication $\Rightarrow$.
 
 To prove the converse assume that $\dt u=\pm\dt e$. There is a triple automorphism $T$ of $C_6$ such that $T(e)=1$ (by Lemma~\ref{L:shift to 1}). By \cite[Corollary 4.4]{cartan6} we deduce that $\dt T(u)=\pm1$. Hence, we may and shall assume that $e=1$.
 
 Further, recall that $C_6$ is represented as the JB$^*$-algebra of hermitian $3\times 3$ matrices of complex octonions (see, e.g., \cite[Section 3]{cartan6}). There is a Jordan $*$-automorphism $S$ of $C_6$ such that $S(u)$ is a diagonal matrix (cf. \cite[Theorem 4.1(iii)]{cartan6}). Since Jordan $*$-isomorphisms clearly preserve the value of determinant, we may and shall additionally assume that $u$ is a diagonal matrix. 
 
 So, $e=1$ and $u$ is a diagonal matrix with $\dt u=\pm 1$. Note that $\dt u$ is the product of the numbers on the diagonal, so it is equal to the usual determinant of this complex matrix. Next observe that $(M_3)_s$, the JB$^*$-algebra of symmetric $3\times 3$ complex matrices, canonically embeds into $C_6$. Hence, we may conclude by applying Proposition~\ref{P:simht1 in symmetric matrices} for $n=3$. We conclude that $u\sim_{h,t}e$ and the respective chaing of $\sim_h$ has length at most $2\cdot3-1=5$.
 
 $(ii)$ This follows easily from $(i)$.
  \end{proof}

By combining Propositions~\ref{P:c6-relations-easy} and~\ref{P:C6 unitaries} we get the following result.

\begin{prop}\label{P:relations in C6}
The following holds in $C_6$:
\begin{enumerate}[$(i)$]
    \item The relations $\le_2,\le_{n,t},\le_{hc,t}$ coincide.
    \item To describe $\le_{h,t}$, the chains of $\le_h$ of length $6$ are enough.
    \item To describe $\le_{hc,t}$, the chains of $\le_{hc}$ of length $6$ are enough.
    \end{enumerate}
\end{prop}

\begin{remark}
As $C_6$ contains the ten-dimensional spin factor as a subtriple, Example~\ref{ex:spin} may be applied for $C_6$ as well.
\end{remark}

We next address the case of $L^\infty(\mu,C_6)$.

\begin{lemma}\label{L:leht selection C_6}
Let $n\in\en$, $n\ge2$. Let $e\in C_6$ be a fixed unitary element. Then
$$A=\{u\in (M_{2n})_a\setsep u\mbox{ is a tripotent such that }\uu\le_{h,t}\e\}$$
is a compact set. Moreover, there is a Borel measurable mapping $\Phi:A\to C_6$ such that for each $u\in A$ its image $\Phi(u)$ is a tripotent such that $u\le \Phi(u)\sim_{h,t}e$.
\end{lemma}

\begin{proof}
The proof may be done by a slight modification of the proof of Lemma~\ref{L:leht selection}.
\end{proof}

\begin{prop}
Let $\mu$ be a probability measure satisfying \eqref{eq:CK=Linfty} and $M=L^\infty(\mu,C_6)=C(\Omega,C_6)$.  Let $\uu,\e$ be two tripotents in $M$. 
\begin{enumerate}[$(i)$]
    \item If 
$$R\in\{\le,\le_h,\le_n,\le_{h,t},\le_{n,t},\le_2,\sim_h,\sim_{h,t},\sim_2\},$$
then
$$\uu R\e \Leftrightarrow \forall\omega\in\Omega\colon \uu(\omega) R\e(\omega).$$
\item $\uu\le_{n,t}\e \Leftrightarrow \uu\le_2\e$;
\item To describe $\sim_{h,t}$ chains of $\sim_h$ of length five are enough.
\item To describe $\le_{h,t}$ chains of $\le_h$ of length six are enough.
\end{enumerate}
\end{prop}

\begin{proof} The proof is completely analogous to that of Proposition~\ref{P:relations is Aotimesc5}. The first part may be copied.

To prove the rest assume first that $\uu(\omega)\sim_{h,t}\e(\omega)$ for $\omega\in\Omega$. Let $p_1,p_3,p_3$ be the canonical diagonal projections in $C_6$ with exactly one $1$ on the diagonal. Then $p_1+p_2+p_3=\1$, the unit of $C_6$. Further, set $e_1=p_1, e_2=p_1+p_2, e_3=p_1+p_2+p_3$ and apply Lemma~\ref{L:Cartan vector parametrization} to $\e$ and $e_1,e_2,e_3$. We get a mapping $\Psi$ and clopen sets $\Omega_0,\dots,\Omega_3$.

Then $\Psi$ is an automorphism of $M$, $\Psi(\e)$ has values only $0,e_1,e_2,e_3$ and $$\hbox{ $\Psi(\uu)(\omega)\sim_{h,t}\Psi(\e)(\omega)$ for $\omega\in\Omega$.}$$ Next we apply Lemma~\ref{L:Cartan vector diagonalization} to $\Psi(\uu)$ and $e_k=p_1+\dots+p_k$, $k=1,2,3$. We thus obtain a mapping $\Theta$. 

Then $\Theta$ is an automorphism of $M$, $\Theta(\Psi(\e))=\Psi(\e)$, the values of $\Theta(\Psi(\uu))$ are diagonal matrices and 
$\Theta(\Psi(\uu)(\omega))\sim_{h,t}\Psi(\e)(\omega)$ for $\omega\in\Omega$.

For $\omega\in\Omega_0\cup\Omega_1$ we have already  $\Theta(\Psi(\uu)(\omega))\sim_{h}\Psi(\e)(\omega)$. On $\Omega_2$ we may apply Proposition~\ref{P:spin vector leht} (since $(C_6)_2(e_2)$ is isomorphic to the ten-dimensional spin factor). Finally, for $\omega\in\Omega_3$ we have $\dt_3\Theta(\Psi(\uu))=\pm1$ by Proposition~\ref{P:C6 unitaries}$(i)$, so we may apply Proposition~\ref{P:AotimesMn}$(ii)$. If put together these things we get a proof of $(i)$ for $\sim_{h,t}$ and of $(iii)$.

Finally, assume that $\uu(\omega)\le_{h,t}\e(\omega)$ for $\omega\in\Omega$. We apply Lemma~\ref{L:Cartan vector parametrization} as in the previous case to get $\Psi$ and $\Omega_0,\dots,\Omega_3$. On $\Omega_0\cup\Omega_1\cup\Omega_2$ we may apply Proposition~\ref{P:spin vector leht}, while on $\Omega_3$ we use Lemma~\ref{L:leht selection C_6}.
In this way we get a proof of $(i)$ for $\le_{h,t}$ and of $(iv)$.
\end{proof}

 \section{Triples of the form $H(W,\alpha)$}
 
 The last type of JBW$^*$-triples to be analyzed are triples of the form $H(W,\alpha)$. Let us recall their definition and some properties.
 
 Let $W$ be a continuous von Neumann algebra. Assume that $\alpha:W\to W$ is a linear involution commuting with the standard involution $^*$. Set
 $$M=H(W,\alpha)=\{x\in W\setsep \alpha(x)=x\}.$$
 We will moreover assume that the involution $\alpha$ is central, i.e. $\alpha(x)=x$ for each $x$ from the center of $W$. This assumption may be done by \cite[Remark 5.2]{Finite}. 
 
 Since $M$ is a subtriple of $W$, the relations $\le,\le_r,\le_c,\le_h,\le_n,\le_2$ can be described in the same way as in a von Neumann algebra and Remark~\ref{rem:relations in subtriples} applies.
 
 \begin{remark}\label{rem:H(W,alpha) is finite}
 By \cite[Proposition 5.8]{Finite} we know that $H(W,\alpha)$ is a finite JBW$^*$-algebra. Hence, by Proposition~\ref{P:le2=lent in finite} the relations $\le_2$ and $\le_{n,t}$ coincide.
 \end{remark}
 
 Further, by Lemmata~\ref{L:shift to 1} and~\ref{L:shift to projection} to describe the relation $uRe$  it is  enough to understand it in case $e=1$ or, more generaly, if $e$ is a projection. In fact, the key thing is just the case $e=1$ as witnessed by the following obvious lemma.
 
 \begin{lemma}
 Let $p\in M=H(W,\alpha)$ be a projection. Then the following assertions hold:
 \begin{enumerate}[$(a)$]
     \item $p$ is a projection in $W$, satisfies $\alpha(p)=p$ and $pWp$ is a continuous von Neumann algebra.
     \item $pWp$ is invariant for $\alpha$.
     \item $M_2(p)$ is canonically Jordan $*$-isomorphic to $H(pWp,\alpha|_{pWp})$.
 \end{enumerate}
 \end{lemma}

\begin{prop}\label{P:HWalpha contains M2s} $M=H(W,\alpha)$ contains a subtriple isomoprhic to $(M_2)_s$. Therefore Example~\ref{ex:M2s examples} may be applied to deduce that the relations $\le_h,\sim_h,\le_{hc},\sim_{hc},\le_n$ are not transitive in $M$.
\end{prop}

\begin{proof}
Since $W$ is continuous, there is a projection $p\in W$ such that $p\sim 1-p$. By \cite[Lemma 5.7]{Finite}  there is a partial isometry $e_1\in M$ such that $p_i(e_1)=p$ and $p_f(e_1)=\alpha(p)$.

 Further, by \cite[Lemma 5.5]{Finite} $1-p\sim 1-\alpha(p)$, hence $p\sim 1-\alpha(p)$. Hence there is a partial isometry $u\in W$ such that $p_i(u)=p$ and $p_f(u)=1-\alpha(p)$. 
Then $\alpha(u)$ is also a partial isometry and $p_i(\alpha(u))=1-p$ and $p_f(\alpha(u))=\alpha(p)$.
Then $v=u+\alpha(u)$ is a unitary element in $M$.

Set $e_2=ue_1^*\alpha(u)$. Then $e_2\in M$ and it is a partial isometry with $p_i(e_2)=1-p$ and $p_f(e_2)=1-\alpha(p)$.
Hence, $e=e_1+e_2$ is a unitary element in $M$. 

Now, we claim that $E=\span\{e_1,e_2,v\}$ is a subtriple of $M$ isomorphic to $(M_2)_s$, the isomorphism being
$$\begin{pmatrix}
a &b \\ b&c
\end{pmatrix} \mapsto a e_1+c e_2+bv.$$
To this end let us work in $W$ (and in $M$) equipped with the operations $\circ_e$ and $^{*_e}$:

In this setting clearly $e_1$ and $e_2$ are mutually orthogonal projections. Moreover,
$$\begin{aligned}
u^{*_e}&=eu^*e=(e_1+e_2)u^*(e_1+e_2)=e_1u^*e_2=e_1u^*ue_1^*\alpha(u)=e_1pe_1^*\alpha(u)\\&=e_1e_1^*\alpha(u)=\alpha(p)\alpha(u)=\alpha(u),\end{aligned}$$
hence $v$ is $^{*_e}$-selfadjoint unitary element in $M$.

It remains to compute
$$\begin{aligned}
e_1\circ_e v&=\J {e_1}{e}{v}=\J{e_1}{e_1}v=\frac12(e_1e_1^*v+ve_1^*e_1)
\\&=\frac12(\alpha(p)(u+\alpha(u))+(u+\alpha(u))p)=\frac12(\alpha(u)+u)=\frac12v,\end{aligned}$$
and
$$e_2\circ_e v=e\circ_e v-e_1\circ_e v=\frac12v.$$
This completes the proof that $E$ is isomorphic to $(M_2)_s$.
Hence, Example~\ref{ex:M2s examples} may be applied.
\end{proof}

It remains to analyze relations $\sim_{h,t}$ and $\sim_{hc,t}$. Since $W$ is continuous, these two relations coincide with $\sim_2$ in $W$ and, moreover, chains of $\sim_h$ of length $16$ are enough to describe $\sim_{h,t}$ in $W$. However, it is not clear whether a similar thing holds also in $H(W,\alpha)$.

Similarly as in Proposition~\ref{P:le1 AotimesB(H)s}$(a)$ we get the following easy result.
 
\begin{lemma}
Let $u\in H(W,\alpha)$ be a unitary element. Then $u\sim_{h,t}1$ if and only if 
$$u=v_1v_2\dots v_n,$$
where $v_1,\dots,v_n\in W$ are symmetries and
$$v_1,v_1v_2,v_1v_2v_3,\dots,v_1v_2\dots v_n\in H(W,\alpha).$$
\end{lemma}

\begin{question}\ 
\begin{enumerate}[$(1)$]
    \item Do the relations $\sim_{h,t}$ and $\sim_2$ coincide in $H(W,\alpha)$?
    \item Is there a  bound of the length of chains of $\sim_h$ needed to describe $\sim_{h,t}$ in $H(W,\alpha)$?
\end{enumerate}
\end{question}

\section{Final overview and open problems}

We defined several natural relations on tripotents and analyzed them firstly in general JB$^*$-triples and then in the individual summands from the representation of JBW$^*$-triples  recalled in \eqref{eq:representation of JBW* triples}. In this last section we briefly review the main results, common points, differences and open problems.

The first level consists of 
relations $\le_r,\le_c,\le_h,\le_{hc},\le_n$. Preorders $\le_r$ and $\le_c$ play an auxilliary role, $\le_h$ and $\le_n$ are inspired by the phenomenon of being self-adjoint (hermitian) and normal, respectively. The relation $\le_{hc}$ is an intermediate one capturing the phenomenon of being a complex multiple of a self-adjoint element.

These relations have natural descriptions in von Neumann algebras given in Section~\ref{sec:vN general}. It does not matter whether we consider these relations in a larger or a smaller triple (cf. Remark~\ref{rem:relations in subtriples}), these descriptions remain to be valid in subtriples of von Neumann algebras (the summands of the form $A\overline{\otimes}C$ where $C$ is a Cartan factor of type $1,2,3$, $H(W,\alpha)$ and $pV$). Cartan factors of type $4$, spin factors, may be also found as subtriples of a von Neumann algebra, but instead of that we used the underlying structure of a Hilbert space to describe the relations  (see Section~\ref{sec:spin}). For the Cartan factor of type $5$ we used the fact that the Peirce-$2$ subspace of any complete tripotent is isomorphic to the eight-dimensional spin factor. It is probably not easy to characterize the relations using directly the structure of $C_5$, but we may  use the results on spin factors. For $C_6$, the Cartan factor of type $6$, the situation is even more complicated - the Peirce-$2$ subspace of a rank-two tripotent is isomorphic to the ten-dimensional spin factor, hence we may again use the results on spin factors. However, $C_6$ admits unitary elements and for them such a simple reduction is not possible. However, we may use Lemma~\ref{L:shift to 1} together with \cite[Corollary 10.3]{cartan6} to reduce it to $H_3(\Ha_C)$ (hence to $(M_6)_a$, cf. Lemma~\ref{L:M2na=HnHC}).

The situation is easy in the simplest JBW$^*$-triple $\ce$ and, more generally, in rank-one Cartan factors. We summarize it in the following theorem.

\begin{thm}\label{T:rank one}
Let $M$ be a rank-one Cartan factor, i.e., either $M=H=B(\ce,H)$ for a Hilbert space $H$ or $M=(M_3)_a$. Let $u,e\in M$ be two tripotents. Then
$$u\le e\Rightarrow u\le_r e\Leftrightarrow u\le_h e \Rightarrow u\le_c\Leftrightarrow u\le_{hc} e\Leftrightarrow u\le_n e\Leftrightarrow u\le_2 e$$
and the remaining implications are not valid.
\end{thm}

The validity of the mentioned implications is now clear (cf.\ Proposition~\ref{P:orders special cases}$(a)$), counterexamples to the remaining implications are given in Example~\ref{ex:distinguishing relations}$(a)$.

\smallskip

The implications which are in general  valid for the first-level relations are summarized in the next theorem.

\begin{thm}
Let $M$ be a JBW$^*$-triple and let $u,e\in M$ be two tripotents. Then:
$$\begin{array}{ccccccccc}
u\le e & \Rightarrow & u\le_r e& \Rightarrow & u \le_c e   &&& &  \\
 && \Downarrow&& \Downarrow && &&   \\
     && u \le_h e &\Rightarrow&u\le_{hc}e &\Rightarrow & u\le_n e &\Rightarrow&u\le_2 e
\end{array}
$$
Moreover,  if $M$ is not a rank-one Cartan factor, none of the implication may be reversed, except possibly for the last one.
\end{thm}

The validity of the implications follows from Proposition~\ref{P:order implications}. If $M$ is not a rank-one Cartan factor, it contains at least two mutually orthogonal nonzero tripotents, hence the respective counterxamples may be found in Example~\ref{ex:distinguishing relations}$(d)$.

The equivalence $u\le_n e \Leftrightarrow u\le_2 e$ holds in abelian von Neumann algebras (any element is normal) and in fact characterizes an interesting class of JBW$^*$-triples.

\begin{thm}
Let $M$ be a JBW$^*$-triple. Then the following assertions are equivalent.
\begin{enumerate}[$(1)$]
    \item $M$ is triple isomorphic to $$\bigoplus_{j\in J}^{\ell_\infty}A_j\overline{\otimes} C_j,$$
where $A_j$'s are abelian von Neumann algebras and $C_j$'s are rank one Cartan factors.
\item $M$ does not contain a subtriple isomorphic to $(M_2)_s$.
\item If $u,e\in M$ are tripotents, then $u\le_n e\Leftrightarrow u\le_2 e$.
\item The relation $\le_n$ is transitive in $M$.
\item The relation $\le_h$ is transitive in $M$.
\item The relation $\sim_h$ is transitive in $M$.
\item The relation $\le_{hc}$ is transitive in $M$.
\item The relation $\sim_{hc}$ is transitive in $M$.
\end{enumerate}
\end{thm}

\begin{proof}
$(1)\Rightarrow (3)\&(4)\&(5)$: We may assume that each $A_j$ is of the form $L^\infty(\mu_j)=C(\Omega_j)$ where $\mu_j$ is a probability measure satisfying $\eqref{eq:CK=Linfty}$. Since $C_j$ is necessarily reflexive, $A_j\overline{\otimes}C_j=L^\infty(\mu_j,C_j)$ (see Lemma~\ref{L:LinftyC=CKC}). Now we can easily conclude using Theorem~\ref{T:rank one}, Proposition~\ref{P:direct sum} and an obvious analogue of Proposition~\ref{P:C0(T,E)}.

$(3)\Rightarrow(4)$: This follows from the transitivity of $\le_2$.

$(5)\Rightarrow(7)$: This follows from Proposition~\ref{P:hct-charact}.

The implications $(5)\Rightarrow(6)$ and $(7)\Rightarrow(8)$ are trivial.

$(4)\vee(6)\vee(8)\Rightarrow(2)$: This follows from Example~\ref{ex:M2s examples} and Remark \ref{rem:relations in subtriples}.

$(2)\Rightarrow(1)$: Consider the representation of $M$ in the form \eqref{eq:representation of JBW* triples}. We observe that any summand different from those given in $(1)$ contains a subtriple isomorphic to $(M_2)_s$:

A Cartan factor of type 1 is of the form $B(H,K)$. If both $H$ and $K$ have dimension at least two, it clearly contains even $M_2$.

 A Cartan factors of type $2$ is of the form $B(H)_a$. If it has rank at least two, $\dim H\ge4$, so it contains a subtriple isomorphic to $(M_4)_a$, hence to $H_2(\Ha_C)$ (see Lemma~\ref{L:M2na=HnHC}). Finally, this subtriple clearly contains a further subtriple isomorphic to $(M_2)_s$.

For Cartan factors of type $3$ the statement is trivial. 

Any spin factor has, by definition, dimension at least $3$, hence we can use \cite[Lemma 3.4$(i)$]{KP2019}.

Further, the Cartan factor of type $5$ contains as a subtriple the eight-dimensional spin factor and the
Cartan factor of type $6$ contains as a subtriple the ten-dimensional spin factor, so we conclude by the previous case. 

For the summand $H(W,\alpha)$ we may use Proposition~\ref{P:HWalpha contains M2s}.  

The summand $pV$ clearly contains even $M_2$ (note that $V$ is continuous).
\end{proof}

In case the relations $\le_h,\le_{hc},\le_n$ are not transitive, we make a further step and consider their transitive hulls. 
Let us point out  that the resulting relations $\le_{h,t},\le_{hc,t},\le_{n,t}$ do depend on the surrounding triple and need not be preserved when passing to a subtriple (see Example~\ref{ex:lent in subtriple}).

When studying the relations $\le_{h,t},\le_{hc,t},\le_{n,t}$, there are two basic sets of questions -- their coincidence with other relations and the smallest possible length of a chain of $\le_h,\le_{hc},\le_n$ needed to describe them.

For $\le_{n,t}$ the situation is quite simple -- a chain of length two is enough (see Lemma~\ref{L:factor len}) and $\le_{n,t}$ coincides with $\le_2$ exactly in those JBW$^*$-triples which are finite (see Proposition~\ref{P:le2=lent in finite}). Finite JBW$^*$-triples are characterized in \cite{Finite}.

For $\le_{h,t}$ and $\le_{hc,t}$ the situation is more complicated. Firstly, if $V$ is a continuous von Neumann algebra and $p\in V$ is a projection, then $\sim_2$ coincides with $\sim_{h,t}$ in $pV$ (cf. Corollary~\ref{cor:pV}$(3)$). It follows that in this case $\le_{h,t}$ coincides with $\le_{n,t}$, and even with $\le_2$ if $V$ is moreover finite (cf. Corollary~\ref{cor:pV}$(3)$). For type I JBW$^*$-triples $\sim_{h,t}$ does not coincide with $\sim_2$ as such a triple contains abelian tripotents (cf. Proposition \ref{P:orders special cases}). It is not known whether $\sim_{h,t}$ coincides with $\sim_2$ in the remaining continuous summand -- $H(W,\alpha)$ -- and we have no idea how to attack this question.

For Cartan factors which are either of type $1$ or of finite rank the relation $\sim_{hc,t}$ coincides with $\sim_2$ (see Proposition~\ref{P:coincidence in B(H,K)}$(a)$, Proposition~\ref{P:coincidence in Mns}$(a)$, Proposition~\ref{P:coincidence in Mna}$(a)$, Corollary~\ref{cor:relations in spin}$(i)$, Corollary~\ref{cor:relations in C5}$(i)$ and Proposition~\ref{P:relations in C6}$(i)$). In Cartan factors of finite rank  $\le_{hc,t}$ coincides even with $\le_2$ (as such JBW$^*$-triples are finite).
In finite rank Cartan factors having a unitary element the relation $\sim_{h,t}$ may be characterized using various notions of determinant (see Proposition~\ref{P:products of symmetries}$(iii)$, Proposition~\ref{P:simht1 in symmetric matrices}, Proposition~\ref{P:simht in M2na}, Proposition~\ref{P:spin-unitaries}$(iv)$ and Proposition~\ref{P:C6 unitaries}$(i)$),
hence $\sim_{h,t}$ does not coincide with $\sim_{hc,t}$. For the remaining Cartan factors of infinite rank, i.e., $B(H)_a$ and $B(H)_s$ for an infinite-dimensional Hilbert space $H$, we do not know whether $\sim_2$ and $\sim_{hc,t}$ coincide.

If we leave Cartan factors and look just to type I JBW$^*$-triples, in particular, to triples of the form $A\overline{\otimes}C$ where $C$ is a Cartan factor, there is no chance to have a coincidence of $\sim_2$ and $\sim_{hc,t}$. There are some results indicating that it would be natural to replace the complex multiple by a `multiple by a central element' -- cf. Proposition~\ref{P:products of symmetries}$(iv)$ or Proposition~\ref{P:AotimesMn}$(i)$. But we have not investigated this direction in more detail.

The last question we are going to comment concerns length of chains of $\sim_h$ needed to describe $\sim_{h,t}$. This is related to some old results on expressing unitary elements as products of symmetries collected in Proposition~\ref{P:products of symmetries}. For triples of the form $pV$ this length is bounded by $16$ (by $4$ if $V$ has no direct summand of type $II$), see Corollary~\ref{cor:pV} (assertion $(1)$ and $(2)$). For Cartan factors of type $4$ or $5$ chains of length $3$ are enough (see Proposition~\ref{P:spin-unitaries}$(iv)$ and Proposition~\ref{P:c5-relations}$(b)(ii)$), for Cartan factors of type $6$ chains of length $5$ are sufficient (see Proposition~\ref{P:c6-relations-easy}$(b)(ii)$ and Propositon~\ref{P:C6 unitaries}$(i)$). Moreover, the chains may be found in a measurable way, so the bounds remain to be valid in tensor products.  For $(M_n)_s$ and $(M_n)_a$ we have bounds depending on $n$ (see Proposition~\ref{P:coincidence in Mns}$(b)$ and Proposition~\ref{P:coincidence in Mna}$(b)$) and we do not know whether this dependence is necessary. It also transfers to the respective tensor products. It is completely unclear whether there is a bound in $B(H)_s$ and $B(H)_a$ for infinite-dimensional $H$ or in $H(W,\alpha)$.

\smallskip\smallskip

\textbf{Acknowledgements.} The first author supported by the project OPVVV CAAS 
 CZ.02.1.01/0.0/0.0/16\_019/0000778.    The third author supported by MCIN / AEI / 10. 13039 / 501100011033 / FEDER ``Una manera de hacer Europa'' project no.  PGC2018-093332-B-I00, Junta de Andaluc\'{\i}a grants FQM375, A-FQM-242-UGR18 and PY20$\underline{\ }$ 00255, and by the IMAG--Mar{\'i}a de Maeztu grant CEX2020-001105-M / AEI / 10.13039 / 501100011033.

\smallskip\smallskip

\textbf{Declaration of competing interests:} none.

\bibliography{orders}\bibliographystyle{acm}

\def\cprime{$'$} \def\cprime{$'$}
\begin{thebibliography}{10}

\bibitem{BattagliaOrder1991}
{\sc Battaglia, M.}
\newblock Order theoretic type decomposition of {${\rm JBW}^*$}-triples.
\newblock {\em Quart. J. Math. Oxford Ser. (2) 42}, 166 (1991), 129--147.

\bibitem{Bickchentaev2004}
{\sc Bikchentaev, A.~M., and Sherstnev, A.~N.}
\newblock Projective convex combinations in {$C^*$}-algebras with the unitary
  factorization property.
\newblock {\em Mat. Zametki 76}, 4 (2004), 625--628.

\bibitem{braun1978holomorphic}
{\sc Braun, R., Kaup, W., and Upmeier, H.}
\newblock A holomorphic characterization of {J}ordan {$C\sp*$}-algebras.
\newblock {\em Math. Z. 161}, 3 (1978), 277--290.

\bibitem{Broise67}
{\sc Broise, M.}
\newblock Commutateurs dans le groupe unitaire d'un facteur.
\newblock {\em J. Math. Pures Appl. (9) 46\/} (1967), 299--312.

\bibitem{BunceChu92}
{\sc Bunce, L.~J., and Chu, C.-H.}
\newblock Compact operations, multipliers and {R}adon-{N}ikod\'{y}m property in
  {${\rm JB}^*$}-triples.
\newblock {\em Pacific J. Math. 153}, 2 (1992), 249--265.

\bibitem{Cabrera-Rodriguez-vol1}
{\sc Cabrera~Garc\'{\i}a, M., and Rodr\'{\i}guez~Palacios, A.}
\newblock {\em Non-associative normed algebras. {V}ol. 1}, vol.~154 of {\em
  Encyclopedia of Mathematics and its Applications}.
\newblock Cambridge University Press, Cambridge, 2014.
\newblock The Vidav-Palmer and Gelfand-Naimark theorems.

\bibitem{Cabrera-Rodriguez-vol2}
{\sc Cabrera~Garc\'{i}a, M., and Rodr\'{i}guez~Palacios, A.}
\newblock {\em Non-associative normed algebras. {V}ol. 2}, vol.~167 of {\em
  Encyclopedia of Mathematics and its Applications}.
\newblock Cambridge University Press, Cambridge, 2018.
\newblock Representation theory and the Zel'manov approach.

\bibitem{DualC(K)}
{\sc Dales, H.~G., Dashiell, Jr., F.~K., Lau, A. T.-M., and Strauss, D.}
\newblock {\em Banach spaces of continuous functions as dual spaces}.
\newblock CMS Books in Mathematics/Ouvrages de Math\'{e}matiques de la SMC.
  Springer, Cham, 2016.

\bibitem{Fillmore}
{\sc Fillmore, P.~A.}
\newblock On products of symmetries.
\newblock {\em Canad. J. Math. 18\/} (1966), 897--900.

\bibitem{FriedmanBook2005}
{\sc Friedman, Y.}
\newblock {\em Physical applications of homogeneous balls}, vol.~40 of {\em
  Progress in Mathematical Physics}.
\newblock Birkh\"{a}user Boston, Inc., Boston, MA, 2005.
\newblock With the assistance of Tzvi Scarr.

\bibitem{Friedman-Russo}
{\sc Friedman, Y., and Russo, B.}
\newblock Structure of the predual of a {$JBW^\ast$}-triple.
\newblock {\em J. Reine Angew. Math. 356\/} (1985), 67--89.

\bibitem{Friedman-Russo-GN}
{\sc Friedman, Y., and Russo, B.}
\newblock The {G}el\cprime fand-{N}a\u{\i}mark theorem for {${\rm
  JB}^\ast$}-triples.
\newblock {\em Duke Math. J. 53}, 1 (1986), 139--148.

\bibitem{cartan6}
{\sc Hamhalter, J., Kalenda, O. F.~K., and Peralta, A.~M.}
\newblock Determinants in {J}ordan matrix algebras.
\newblock arXiv:2110.10458.

\bibitem{Finite}
{\sc Hamhalter, J., Kalenda, O. F.~K., and Peralta, A.~M.}
\newblock Finite tripotents and finite {${\rm JBW}^*$}-triples.
\newblock {\em J. Math. Anal. Appl. 490}, 1 (2020), article no. 124217, 65pp.

\bibitem{hamhalter2019mwnc}
{\sc Hamhalter, J., Kalenda, O. F.~K., Peralta, A.~M., and Pfitzner, H.}
\newblock Measures of weak non-compactness in preduals of von {N}eumann
  algebras and {$\rm JBW^\ast$}-triples.
\newblock {\em J. Funct. Anal. 278}, 1 (2020), article no. 108300, 69pp.

\bibitem{HKPP-BF}
{\sc Hamhalter, J., Kalenda, O. F.~K., Peralta, A.~M., and Pfitzner, H.}
\newblock Grothendieck's inequalities for {$\rm JB^*$}-triples: proof of the
  {B}arton-{F}riedman conjecture.
\newblock {\em Trans. Amer. Math. Soc. 374}, 2 (2021), 1327--1350.

\bibitem{hanche1984jordan}
{\sc Hanche-Olsen, H., and St{\o}rmer, E.}
\newblock {\em Jordan operator algebras}, vol.~21.
\newblock Pitman Advanced Publishing Program, 1984.

\bibitem{harris1974bounded}
{\sc Harris, L.~A.}
\newblock Bounded symmetric homogeneous domains in infinite dimensional spaces.
\newblock In {\em Proceedings on {I}nfinite {D}imensional {H}olomorphy
  ({I}nternat. {C}onf., {U}niv. {K}entucky, {L}exington, {K}y., 1973)\/}
  (1974), pp.~13--40. Lecture Notes in Math., Vol. 364.

\bibitem{horn1987ideal}
{\sc Horn, G.}
\newblock Characterization of the predual and ideal structure of a {${\rm
  JBW}^*$}-triple.
\newblock {\em Math. Scand. 61}, 1 (1987), 117--133.

\bibitem{horn1987classification}
{\sc Horn, G.}
\newblock Classification of {JBW{$^*$}}-triples of type {${\rm I}$}.
\newblock {\em Math. Z. 196}, 2 (1987), 271--291.

\bibitem{horn1988classification}
{\sc Horn, G., and Neher, E.}
\newblock Classification of continuous {$JBW^*$}-triples.
\newblock {\em Trans. Amer. Math. Soc. 306}, 2 (1988), 553--578.

\bibitem{isidro1995real}
{\sc Isidro, J.~M., Kaup, W., and Rodr\'{\i}guez-Palacios, A.}
\newblock On real forms of {${\rm JB}^*$}-triples.
\newblock {\em Manuscripta Math. 86}, 3 (1995), 311--335.

\bibitem{KadisonAnn51}
{\sc Kadison, R.~V.}
\newblock Isometries of operator algebras.
\newblock {\em Ann. of Math. (2) 54\/} (1951), 325--338.

\bibitem{kadison-diagonal}
{\sc Kadison, R.~V.}
\newblock Diagonalizing matrices.
\newblock {\em Amer. J. Math. 106}, 6 (1984), 1451--1468.

\bibitem{KR1}
{\sc Kadison, R.~V., and Ringrose, J.~R.}
\newblock {\em Fundamentals of the theory of operator algebras. {V}ol. {I}},
  vol.~15 of {\em Graduate Studies in Mathematics}.
\newblock American Mathematical Society, Providence, RI, 1997.
\newblock Elementary theory, Reprint of the 1983 original.

\bibitem{KR2}
{\sc Kadison, R.~V., and Ringrose, J.~R.}
\newblock {\em Fundamentals of the theory of operator algebras. {V}ol. {II}},
  vol.~16 of {\em Graduate Studies in Mathematics}.
\newblock American Mathematical Society, Providence, RI, 1997.
\newblock Advanced theory, Corrected reprint of the 1986 original.

\bibitem{KP2019}
{\sc Kalenda, O. F.~K., and Peralta, A.~M.}
\newblock Extension of isometries from the unit sphere of a rank-2 {C}artan
  factor.
\newblock {\em Anal. Math. Phys. 11}, 1 (2021), Paper No. 15, 25pp.

\bibitem{kaup1977-maan}
{\sc Kaup, W.}
\newblock Algebraic characterization of symmetric complex {B}anach manifolds.
\newblock {\em Math. Ann. 228}, 1 (1977), 39--64.

\bibitem{kaup1983riemann}
{\sc Kaup, W.}
\newblock A {R}iemann mapping theorem for bounded symmetric domains in complex
  {B}anach spaces.
\newblock {\em Math. Z. 183}, 4 (1983), 503--529.

\bibitem{kaup1997real}
{\sc Kaup, W.}
\newblock On real {C}artan factors.
\newblock {\em Manuscripta Math. 92}, 2 (1997), 191--222.

\bibitem{Ku-RN}
{\sc Kuratowski, K., and Ryll-Nardzewski, C.}
\newblock A general theorem on selectors.
\newblock {\em Bull. Acad. Polon. Sci. S\'{e}r. Sci. Math. Astronom. Phys.
  13\/} (1965), 397--403.

\bibitem{loos1977bounded}
{\sc Loos, O.}
\newblock {\em Bounded symmetric domains and {J}ordan pairs}.
\newblock Lecture Notes. Univ. California at Irvine, 1977.

\bibitem{Paterson-Sinclair}
{\sc Paterson, A. L.~T., and Sinclair, A.~M.}
\newblock Characterisation of isometries between {$C^{\ast} $}-algebras.
\newblock {\em J. London Math. Soc. (2) 5\/} (1972), 755--761.

\bibitem{pearcy1963}
{\sc Pearcy, C.}
\newblock On unitary equivalence of matrices over the ring of continuous
  complex-valued functions on a {S}tonian space.
\newblock {\em Canad. J. Math. 15\/} (1963), 323--331.

\bibitem{radjavi}
{\sc Radjavi, H.}
\newblock Products of hermitian matrices and symmetries.
\newblock {\em Proc. Amer. Math. Soc. 21\/} (1969), 369--372.

\bibitem{Tak}
{\sc Takesaki, M.}
\newblock {\em Theory of operator algebras. {I}}.
\newblock Springer-Verlag, New York-Heidelberg, 1979.

\bibitem{Wright1977}
{\sc Wright, J. D.~M.}
\newblock Jordan {$C\sp*$}-algebras.
\newblock {\em Michigan Math. J. 24}, 3 (1977), 291--302.

\end{thebibliography}

\end{document}